\documentclass[10pt]{amsart}

\usepackage{amsmath,amssymb,amscd,amsthm,graphicx,enumerate}

\usepackage{hyperref}
\usepackage{xcolor}
\usepackage{color}
\hypersetup{breaklinks=true}
\hypersetup{
    colorlinks=true,
    linkcolor=red,
    filecolor=magenta,
    urlcolor=cyan,
    citecolor=green,
}

\numberwithin{equation}{section}
\theoremstyle{plain}

\newtheorem{theorem}{Theorem}[section]
\newtheorem{proposition}[theorem]{Proposition}
\newtheorem{lemma}[theorem]{Lemma}
\newtheorem{corollary}[theorem]{Corollary}

\newtheorem{claim}[theorem]{Claim}

\theoremstyle{remark}
\newtheorem{remark}[theorem]{Remark}

\theoremstyle{definition}
\newtheorem{definition}[theorem]{Definition}

\newcommand{\eps}{\varepsilon}

\providecommand{\abs}[1]{\lvert #1\rvert}
\providecommand{\norm}[1]{\lvert\lvert #1\rvert\rvert}

\newcommand{\measurerestr}{ \,\raisebox{-.127ex}{\reflectbox{\rotatebox[origin=br]{-90}{$\lnot$}}}\, }

\begin{document}
\title{Spectral quantization  for ancient asymptotically cylindrical flows}
\author{Wenkui Du, Jingze Zhu}
\begin{abstract}
 We study ancient mean curvature flows in $\mathbb{R}^{n+1}$  whose tangent flow at $-\infty$ is a shrinking cylinder $\mathbb{R}^{k}\times S^{n-k}(\sqrt{2(n-k)|t|})$, where $1\leq k\leq n-1$. We prove  that the cylindrical profile function $u$ of these flows have the asymptotics  $u(y,\omega,\tau)= (y^\top Qy -2\textrm{tr}(Q))/|\tau| + o(|\tau|^{-1})$ as $\tau\to -\infty$,  where the cylindrical matrix $Q$ is a constant symmetric $k\times k$ matrix whose eigenvalues are quantized to be either 0 or $-\frac{\sqrt{2(n-k)}}{4}$.  Compared with the bubble-sheet quantization theorem in $\mathbb{R}^{4}$ obtained by Haslhofer and the first author, this theorem has full generality in the sense of removing noncollapsing condition and being valid for all dimensions. In addition, we establish symmetry improvement theorem which generalizes the corresponding results of Brendle-Choi and the second author to all dimensions. Finally, we give some geometric applications of the two theorems. In particular, we obtain the asymptotics, compactness and $\textrm{O}(n-k+1)$ symmetry of $k$-ovals in $\mathbb{R}^{n+1}$ which are ancient noncollapsed
flows in $\mathbb{R}^{n+1}$ satisfying full rank condition that $\textrm{rk}(Q)=k$, and we also obtain the classification of ancient noncollapsed flows in $\mathbb{R}^{n+1}$ satisfying vanishing rank condition that $\textrm{rk}(Q)=0$.
\end{abstract}
\maketitle
\tableofcontents
\section{Introduction}
Mean curvature flow $\mathcal{M}=\{M_{t}\}_{t\in [0, T)}\subset \mathbb{R}^{n+1}$  is a family of embedded hypersurfaces evolving by mean curvature vector $\mathbf{H}$,
\begin{equation}
   (\partial_{t}x)^{\perp}=\mathbf{H} \quad\quad (x\in M_{t}).
\end{equation}
Singularity formation usually happens under mean curvature flow when the second fundamental form $A$ blows up at the first singular time $T$, i.e,
\begin{equation}
    \lim_{t\to T}\max_{M_{t}}|A|=\infty.
\end{equation}
In order to apply mean curvature flow to solve geometric and topological problems, one needs to classify its singularity models. The singularity models are obtained from magnifying the original flow via rescaling by a sequence of factors going to infinity and passing to a blowup limit. Any such blowup limit is an ancient solution, i.e. a solution that is defined for all sufficiently negative times.

The most important singularity models are  generalized shrinking cylinders. It has been verified by Colding-Minicozzi \cite{CM_generic} that these generalized cylinders are the only stable shrinkers  and they are conjectured to be generic singularities of mean curvature flow under initial data perturbation. Besides these generalized cylinders, the  generalized ancient asymptotically cylindrical  flows  perturbed from these generalized cylinders can also be potential singularity models. These flows have  $\mathbb{R}^{k}\times S^{n-k}$ as tangent flow at $-\infty$, where $0\leq k\leq n$ (see Definition \ref{def asymptotical cylindrical} for more details).
If the tangent flow at $-\infty$ is either sphere $S^{n}$ or static hyperplane $\mathbb{R}^{n}$, by Huisken's monontonicity formula \cite{Huisken_monotonicity} ancient asymptotically cylindrical  flows have to be sphere $S^{n}$  or static hyperplane $\mathbb{R}^{n}$ respectively. Therefore, in the whole paper we will only consider ancient asymptotically cylindrical  flows with cylinder $\mathbb{R}^{k}\times S^{n-k}$ as tangent flow at $-\infty$, where $1\leq k\leq n-1$.

Ancient asymptotically cylindrical  flows also play important roles in theory of mean convex mean curvature flows pioneered by White \cite{White_size, White_nature}. For  mean convex mean curvature flows, all the blowup singularity models are ancient noncollapsed flows and they are ancient asymptotically cylindrical  flows. Here, an ancient noncollapsed flow means that it is mean-convex (mean curvature $H\geq 0$ on the flow) and there is an $\alpha>0$ so that every space-time point $(p, t)$ on the flow admits interior and exterior balls of radius at least $\alpha/H(p, t)$, c.f. \cite{ShengWang,Andrews_noncollapsing,HaslhoferKleiner_meanconvex, Brendle_inscribed, HK_inscribed}.  While for ancient  flows whose tangent flow at $-\infty$ is given by a neck $\mathbb{R}\times {S}^{n-1}$, or which are uniformly two-convex noncollapsed,  or which are  rotationally symmetric,  a complete classification of such
flows has been obtained recently in a sequence of works \cite{CHH, CHHW, BC1, BC2, ADS1, ADS2, DH_ovals}, the full classification of general ancient asymptotically cylindrical  flows whose tangent flow at $-\infty$ is $\mathbb{R}^{k}\times {S}^{n-k}$ ($2\leq k\leq n-1$) is still widely open.

To deal with  the challenges discussed above when $k\geq 2$, we need the following three key ingredients:
\begin{itemize}
    \item  Lojasiewicz inequality,
     \item  symmetry improvement theorem,
    \item  spectral quantization theorem.
\end{itemize}
The Lojasiewicz inequality established by Colding-Minicozzi in \cite{CM_uniqueness} implies the uniqueness of the cylindrical tangent flow at $-\infty$ for ancient asympototically cylindrical flows. Therefore we can write the renormalized flow as a smooth graph over a fixed cylinder in a ball with a quantitatively large  radius. Recently, Brendle-Choi \cite{BC1, BC2} and Zhu \cite{Zhu, Zhu2} established symmetry improvement theorem  in  $\mathbb{R}^{n+1}$  for ancient noncollapsed solutions in  $\mathbb{R}^{n+1}$ whose tangent flow at $-\infty$ is a neck $\mathbb{R}\times S^{n-1}$ or a bubble-sheet $\mathbb{R}^{2}\times S^{n-2}$. The symmetry improvement theorem  asserts that almost cylindrical regions and almost cap regions become more symmetric along the  directions of spherical factor in the cylindrical tangent flow at $-\infty$  under the evolution.
 Very recently,  Haslhofer and the first author \cite{DH_hearing_shape} established the spectral quantization theorem for ancient noncollapsed solutions in  $\mathbb{R}^4$ whose tangent flow at $-\infty$ is a bubble-sheet $\mathbb{R}^{2}\times S^{1}$. This theorem provides the  non-upward quadratic bending sharp asymptotics of the flow in the central region over the bubble-sheet.
 Based on above works, the classification of all ancient noncollapsed solutions in  $\mathbb{R}^4$  had significant progress in \cite{DH_hearing_shape, CHH_translator, oval_classification_r4}.

In the present paper, we aim to prove spectral quantization theorem  and symmetry improvement theorem for  general  ancient asympototically cylindrical  flows in all dimensions, which are neither a priori convex as in \cite{DH_hearing_shape}  nor a priori  rotationally symmetric as the classification result in \cite{DH_ovals}. Then we give some important applications of these two theorems in the classification of general ancient noncollapsed mean curvature flows.

\subsection{Main results}\label{main results} To describe our main results in detail,
we first introduce some concepts related to Brakke flow  \cite[Def 6.2-6.3]{Ilmanen_book}, which is a weak notion of mean curvature flow. An $n$ dimensional ancient integral Brakke flow is given by a family of
Radon measures  $\mathcal{M}=\{\mu_{t}\}_{t\in (-\infty, T_{\textrm{ext}}]}$ in $\mathbb{R}^{n+1}$ that is integer $n$-rectifiable for almost all times before some time $T_{\textrm{ext}}\leq \infty$. Namely there is some $\mathbb{Z}$-valued multiplicity function $\Theta$ and an $n$-dimensional rectifiable set $M_{t}$, such that
\begin{equation}
    \mu_{t}=\Theta\mathcal{H}^{n}\measurerestr M_{t},
\end{equation}
and satisfies
\begin{equation}
    \frac{d}{dt}\int \Psi d\mu_{t}\leq \int\left(-\Psi |\mathbf{H}|^2+\nabla \Psi \cdot \mathbf{H}\right) d\mu_{t}
\end{equation}
for all test functions $\Psi\in C^{1}_{c}(\mathbb{R}^{n+1}, \mathbb{R}_{+})$. Here, $\frac{d}{dt}$ denotes the limsup of difference quotients, and $\mathbf{H}$ denotes
the mean curvature vector of the associated varifold $V_{\mu_{t}}$, which is defined via the first variation formula and exists almost everywhere at almost all times. The integral on the right hand side is interpreted as $-\infty$ whenever it does not make sense literally.

Then,  for the above  ancient integral Brakke flow $\mathcal{M}$, given a sequence of positive numbers $\lambda_{j}\to 0$ let $\mathcal{D}_{\lambda_{j}}(\mathcal{M})=\{\mu^{\lambda_{j}}_{t}\}_{t\in (-\infty, \lambda^2_{j} T_{\textrm{ext}}]}$ be the Brakke flows that are obtained from $\mathcal{M}$ by the parabolically rescaling $\mathcal{D}_{\lambda_{j}}(x, t)= (\lambda_{j} x, \lambda_{j}^2 t)$, where the rescaling is defined by $\mu^{\lambda_{j}}_{t}(A)=\lambda_j^n \mu_{\lambda^{-2}_{j} t}(\lambda^{-1}_{j}A)$ for any  $A\subset \mathbb{R}^{n+1}$. By Brakke's compactness theorem {\cite[Thm 7.1]{Ilm94}} and  Huisken’s monotonicity formula \cite[Thm 3.1]{Huisken_monotonicity},  we can always pass to a subsequential limit  to obtain an ancient integral Brakke flow $\mathcal{M}^{\infty}=\{\mu^{\infty}_{t}\}$, where the convergence is in the weak convergence sense of of Brakke flows \cite[Thm 7.1]{Ilm94}.  However, if the subsequential limit happens to be smooth with multiplicity one, then by the local regularity theorem \cite{White_regularity} the convergence is actually smooth. In addition, any such limit is called as a tangent flow at $-\infty$ and is self-similarly shrinking. As we have discussed before, the most important self-similarly shrinking solutions of mean curvature flow are self-shrinking  cylinders $\mathbb{R}^{k}\times S^{n-k}(\sqrt{2(n-k)|t|})$, where $1\leq k\leq n-1$. Now, we give the formal definition of ancient asymptotically cylindrical flow.
\begin{definition}\label{def asymptotical cylindrical}
An  ancient integral Brakke flow  $\mathcal{M}=\{\mu_{t}\}_{t\in (-\infty, T_{\textrm{ext}}]}$ in $\mathbb{R}^{n+1}$ is an ancient asymptotically cylindrical flow if some of its tangent flow at $-\infty$ along some sequence of positive numbers $\lambda_{j}\to 0$ under suitable coordinates is a  shrinking cylindrical flow:
\begin{equation}\label{subsequence convergence}
\mathcal{M}^{\infty}=\lim_{j\to \infty}\mathcal{D}_{\lambda_{j}}(\mathcal{M})=\{\mathbb{R}^{k}\times S^{n-k}(\sqrt{2(n-k)|t|})\},
\end{equation}
for some $1\leq k\leq n-1$.
\end{definition}

By local regularity theorem \cite{White_regularity} and uniqueness of cylindrical tangent flow \cite{CM_uniqueness},  the  convergence \eqref{subsequence convergence} in Definition \ref{def asymptotical cylindrical} is smooth and entails full convergence. In particular the $\mathbb{R}^{k}$ factor in the limit is unique. Hence, the renormalized mean curvature flow $\bar M_\tau = e^{\frac{\tau}{2}}  M_{-e^{-\tau}}$ of ancient asymptotically cylindrical flow $M_{t}$ in suitable coordinates
 converges to the  following cylinder as $\tau\to -\infty$
\begin{equation}\label{bubble-sheet_tangent_intro}
\lim_{\tau\to -\infty}\bar{M}_\tau=\Gamma:=\mathbb{R}^{k}\times S^{n-k}(\sqrt{2(n-k)}),
\end{equation}
for some $1\leq k\leq n-1$.

Now, we can write $\bar{M}_\tau$ as a graph of a function $u(\cdot,\tau)$ over $\Gamma\cap B_{\rho(\tau)}$,
\begin{equation}\label{graph_over_cylinder}
\left\{ (y, \sqrt{2(n-k)} \omega)+ u(y,  \omega,\tau)\nu(y, \sqrt{2(n-k)} \omega) :  |y|\leq \rho(\tau) \right\} \subset \bar M_\tau\, ,
\end{equation}
where we suitably choose graphical radius $\rho(\tau)\to \infty$ as $\tau\to -\infty$ and we denote by $\nu$ the outwards unit normal of $\Gamma$.  Our following main theorem describes the asymptotic behaviour of the cylindrical profile function $u(y,\omega,\tau)$.
\begin{theorem}[spectral quantization theorem]\label{spectral theorem}
For any ancient asymptotically cylindrical flow $\mathcal{M}=\{M_t\}$ in $\mathbb{R}^{n+1}$ whose tangent flow at $-\infty$ is given by the cylindrical flow $\mathbb{R}^{k}\times S^{n-k}(\sqrt{2(n-k)|t|})$ for some $1\leq k\leq n-1$, the cylindrical profile function $u$ satisfies the following sharp asymptotics
 \begin{equation}\label{main_thm_ancient_intro}
\lim_{\tau\to -\infty} \Big\|\,
|\tau| u(y, \omega,\tau)- y^\top Qy +2\mathrm{tr}(Q)\, \Big\|_{C^{p}(B_{R})} = 0
\end{equation}
for all $R<\infty$ and all integers $p$, where $Q$ is a constant symmetric $k\times k$-matrix whose eigenvalues are quantized to be either 0 or $-\frac{\sqrt{2(n-k)}}{4}$.
\end{theorem}
This theorem gives the quantitative deviation of the renormalized flow from round cylinder. This theorem also improves previous results \cite{DH_hearing_shape, DH_no_rotation, DH_ovals, CHH, CHHW} in the following ways: (1)  It removes the noncollapsing condition in \cite{DH_hearing_shape, DH_no_rotation}. The hypersurfaces have  a convex non-upward quadratic bending shape over the $\mathbb{R}^k$-factor although no convexity is assumed; (2) It removes the $\textrm{SO}(n-k+1)$ symmetry along the directions of spherical factor in the tangent flow at $-\infty$ in \cite{DH_ovals}. The hypersurfaces become asymptotically $\mathbb{Z}_{2}^{k}\times \mathrm{O}(n-k+1)$-symmetric and have quantized bending asympotitics, where the bending coefficients are quantized to be either $0$ or $-\frac{\sqrt{2(n-k)}}{4}$; (3) It generalizes the results of asymptotics in \cite{DH_hearing_shape, DH_no_rotation, CHH, CHHW} to all dimensions and for all types of cylinders as tangent flow at $-\infty$.

We remark that forwards in time related quantization behaviour, of course with the opposite sign, has been observed by Filippas-Liu \cite{FilippasLiu} for singularities of multidimensional semilinear heat equations, by Zhou \cite{Zhou_dynamics} for certain cylindrical singularities under mean curvature flow and by Sun-Xue \cite{SX} for studying generic isolatedness of cylindrical singularities of mean curvature flow. Our purposes and methods of proving this theorem are very different from these works.

In addition, we also obtain the  general form of symmetry improvement theorem (cylindrical symmetry improvement and cap improvement) for ancient asymptotically cylindrical flows compared with \cite{Zhu2, Zhu, BC1, BC2, ADS2}.  To describe the symmetry improvement theorem, we first give the following definitions. For every space-time point $(\bar{x}, \bar{t})$ and every number $L>0$, we let
\begin{equation}
\hat{\mathcal{P}}(\bar{x}, \bar{t}, L, L^2)=B_{{g}(\bar{t})}(\bar{x}, H(\bar{x}, \bar{t})^{-1}L)\times [\bar{t}-H(\bar{x}, \bar{t})^{-2}L^2, \bar{t}]
\end{equation}
be the normalized (intrinsic) parabolic neighborhood of $(\bar{x}, \bar{t})$, where ${g}(\bar{t})$ is the metric on  $M_{\bar{t}}$ and the mean curvature $H(\bar{x},\bar{t})>0$.
\begin{definition}[$\varepsilon$-close to cylinder]\label{e-close cylinder}
 Let $\mathcal{M}=\{M_t\}$ be a mean curvature flow. We say that a space-time point $(\bar{x}, \bar{t})\in \mathcal{M}$  is $\varepsilon$-close to a cylinder $\mathbb{R}^{k}\times S^{n-k}$ (or  lies on a $\varepsilon$-cylinder $\mathbb{R}^{k}\times S^{n-k}$) if the normalized parabolic neighborhood $\hat{\mathcal{P}}(\bar{x}, \bar{t}, 100n^{5/2}, 100^2n^2)$ is $\varepsilon$-close (in $C^{10}$ norm) to a family of shrinking cylinders $\mathbb{R}^{k}\times S^{n-k}$ after rescaling by the mean curvature $H(\bar{x}, \bar{t})$.
\end{definition}
\begin{definition}\label{e_symmetric}
  Let $\mathcal{M}=\{M_t\}$ be a mean curvature flow and $(\bar{x}, \bar{t})\in\mathcal{M}$ be   a space-time point such that the mean curvature is positive in
$\hat{\mathcal{P}}(\bar{x}, \bar{t}, 100n^{5/2}, 100^2n^5)$.
  We say that  $(\bar{x}, \bar{t})$ is $(\varepsilon, n-k)$-symmetric for some $1\leq k\leq n-1$  if there exists a  normalized set of rotation vector fields
  $\mathcal{K} = \{K_{\alpha} : 1 \leq \alpha \leq \frac{(n-k+1)(n-k)}{2}\}$ such that the conditions
 \begin{itemize}
     \item  $\max_{\alpha}|\left<K_{\alpha}, \nu\right> |H \leq \varepsilon$,
     \item  $\max_{\alpha}|K_{\alpha}|H \leq 5n$
 \end{itemize}
 hold in  $\hat{\mathcal{P}}(\bar{x}, \bar{t}, 100n^{5/2}, 100^2n^5)$, where the vector fields are defined by
  \begin{align}
    K_{\alpha}=SJ_{\alpha}S^{-1}(x-q)\quad S\in \mathrm{O}(n+1)\quad q\in\mathbb{R}^{n+1},
  \end{align}
 \begin{align}\label{J, Jbar}
        J_{\alpha}  =\begin{bmatrix}
   0 & 0\\
    0 &  \bar{J}_{\alpha}
    \end{bmatrix}\in so(n+1),
\end{align}
and $\{\bar{J}_{\alpha}: 1\leq\alpha\leq \frac{(n-k+1)(n-k)}{2}\}$ is an
  orthonormal basis of $so(n-k+1)\subset so(n+1)$.
\end{definition}
Then we state the symmetry improvement theorem (cylindrical symmetry improvement theorem and cap improvement theorem) below.
\begin{theorem}[symmetry improvement theorem]\label{Symmetry improvement intro}
There exist constants $L_0<\infty$ large enough and $\varepsilon_0>0$ small enough depending only on dimension $n$  and satisfying the following properties. Suppose that $\mathcal{M}=\{M_t\}$ is a mean curvature flow  and $(\bar{x}, \bar{t})\in \mathcal{M}$ is a space-time point.
If at least one of the following hold
\begin{itemize}
    \item either  every point $(y, s)\in \hat{\mathcal{P}}(\bar{x},\bar{t},L_0, L_0^2)$  is $\varepsilon_0$-close to a  cylinder $\mathbb{R}^{k}\times S^{n-k}$,
    \item or $\hat{\mathcal{P}}(\bar{x},\bar{t},L_0, L_0^2)$ is $\varepsilon_0$-close to a piece of some $\mathbb{R}^{k-1}$ times $(n-k+1)$ dimensional round bowl soliton after rescaling such that $H(\bar{x},\bar{t})=1$,
\end{itemize}
and  every point in the parabolic neighborhood $\hat{\mathcal{P}}(\bar{x},\bar{t},L_0, L_0^2)$ is $(\varepsilon, n-k)$-symmetric, where $0< \varepsilon \leq \varepsilon_0$, then $(\bar{x}, \bar{t})$
is $(\frac{\varepsilon}{2}, n-k)$-symmetric.
\end{theorem}
\begin{remark}\label{sym remark}
Theorem \ref{Symmetry improvement intro} (symmetry improvement theorem) also holds if we replace the normalized intrinsic parabolic neighborhood  by the following Euclidean parabolic neighborhood with possibly larger  $L_0$ and smaller $\varepsilon_0$ still depending on dimension $n$ in the form of  as used in \cite[Sec 2]{ADS2}
\begin{equation}
    P(\bar{x},\bar{t}, H(\bar{x},\bar{t})^{-1}L_0, H(\bar{x},\bar{t})^{-2}L_0^2)=B(\bar{x}, H(\bar{x}, \bar{t})^{-1}L_{0})\times [\bar{t}-H(\bar{x}, \bar{t})^{-2}L_{0}^2, \bar{t}],
\end{equation}
and similarly if we replace the intrinsic parabolic neighborhood in Definition \ref{e_symmetric} by Euclidean parabolic neighborhood.
\end{remark}
 Theorem \ref{spectral theorem} (spectral quantization theorem) and Theorem \ref{Symmetry improvement intro} (symmetry improvement theorem) are  important ingredients in the  classification of  general ancient asymptotically cylindrical flows or ancient noncollapsed mean curvature flows in $\mathbb{R}^{n+1}$ with arbitrary cylinder $\mathbb{R}^{k}\times S^{n-k}$ as tangent flow at $-\infty$. The classification in the general cases where $1\leq \textrm{rk}(Q)\leq k-1$ seems currently out of reach. In this paper, we focus on the discussion on the extremal rank cases: $\textrm{rk}(Q)=k$ (non-degenerate case) and $\textrm{rk}(Q)=0$ (fully-degenerate case). To this end,  we first give the following definition.
\begin{definition}
A $k$-oval in $\mathbb{R}^{n+1}$ is an ancient noncollapsed flow whose tangent flow at $-\infty$ is $\mathbb{R}^{k}\times S^{n-k}(\sqrt{2(n-k)|t|})$   for some $1\leq k\leq n-1$  and  whose cylindrical matrix $Q$ satisfies the full rank condition that $\textrm{rk}(Q)=k$ in Theorem \ref{spectral theorem}.
\end{definition}
This concept of $k$-ovals in $\mathbb{R}^{n+1}$ generalizes the compact case where the neutral mode is dominant for ancient asymptotically neck flows in \cite{CHH, CHHW} and the concept of bubble-sheet ovals in $\mathbb{R}^{4}$ in \cite{oval_classification_r4}. For $k$-ovals in $\mathbb{R}^{n+1}$ (non-degenerate case), we have the following result.
\begin{theorem}[asymptotics, compactness and symmetry of $k$-ovals in $\mathbb{R}^{n+1}$]\label{compact_symmetry_oval}
Every $k$-oval in $\mathbb{R}^{n+1}$ is compact and $\mathrm{O}(n-k+1)$ symmetric and has the same sharp asymptotics as the unique $ \mathrm{O}(k)\times \mathrm{O}(n-k+1)$ symmetric oval \cite{White_nature, HaslhoferHershkovits_ancient, DH_ovals} up to time shift and parabolic rescaling:
\begin{itemize}
    \item Parabolic region: Given any $R>0$, the cylindrical profile function $u$ for $\tau\to -\infty$ satisfies
    \begin{equation*}
     u(y,\omega,\tau)=\frac{\sqrt{2(n-k)}}{4}\frac{|y|^2-2k}{\tau} +o(|\tau|^{-1})
    \end{equation*}
        uniformly for $|y|\leq R$.
      \item Intermediate region: Let $\bar{u}(z,\omega, \tau)=u(|\tau|^{\frac{1}{2}}z,\omega, \tau)+\sqrt{2(n-k)}$, we have
\begin{equation*}
    \lim_{\tau\rightarrow -\infty}\bar{u}(z,\omega, \tau)=\sqrt{(n-k)(2-|z|^2)}
\end{equation*}
    uniformly on every compact subset of $\{(z, \omega): |z|<\sqrt{2}, \omega \in S^{n-k}\}$.
    \item Tip region:  Set $\lambda(s)=\sqrt{|s|^{-1}\log |s|}$, and fix any unit directional vector $\vartheta\in S^{k-1}\cap (\mathbb{R}^{r}\times \{0\})$. Let $p_{s}\in M_s$ be the point that maximizes $\langle p, \vartheta\rangle$ among all $p\in M_s$. Then as $s\to -\infty$, the rescaled flows
    \begin{equation*}
        {\widetilde{M}}^{s}_{t}=\lambda(s)\cdot(M_{s+\lambda(s)^{-2}t}-p_{s})
    \end{equation*}
    converge to $\mathbb{R}^{k-1}\times N_{t}$, where $N_{t}$ is the $(n+1-k)$-dimensional round translating bowl  in $\mathbb{R}^{n+1-k}$ with speed $1/\sqrt{2}$.
\end{itemize}
\end{theorem}
This theorem generalizes the result \cite[Thm 1.4]{DH_hearing_shape} to all dimensions and removes the condition of rotational symmetry  along the directions of spherical factor  in the unique tangent flow at $-\infty$ in \cite[Thm1.4]{DH_ovals}.
Using this theorem together with Theorem \ref{spectral theorem} and the uniqueness result of $\mathrm{SO}(k)\times \mathrm{SO}(n-k+1)$ symmetric ancient ovals in \cite{DH_ovals}, we obtain the following direct corollary.
\begin{corollary}[classification of partially rotational symmetric solutions]\label{symmetr classification}
For any ancient noncollapsed flow $\mathcal{M}=\{M_t\}$ in $\mathbb{R}^{n+1}$
 with   the cylindrical flow $\mathbb{R}^{k}\times S^{n-k}(\sqrt{2(n-k)|t|})$  as its tangent flow at $-\infty$ for some $1\leq k\leq n-1$ and $\mathrm{SO}(k)$ symmetry on $\mathbb{R}^{k}$ coordinates,  it is (up to time shift and parabolically dilation) either  $\mathbb{R}^{k}\times S^{n-k}(\sqrt{2(n-k)|t|})$ or the unique $\mathrm{O}(k)\times \mathrm{O}(n-k+1)$ symmetric oval, which has the same tangent flow at $-\infty$ and was constructed by White \cite{White_nature} and Haslhofer-Hershkovits \cite{HaslhoferHershkovits_ancient}.
\end{corollary}
 This corollary relaxes the larger rotational symmetry condition in \cite{DH_ovals}\footnote{In \cite{DH_ovals}, Haslhofer and the first author classified all $\mathrm{SO}k)\times \mathrm{SO}(n-k+1)$ symmetric ancient noncollapsed solutions (up to time shift and parabolically dilation) into $S^{n}$, $S^{k-1}\times \mathbb{R}^{n-k+1}$, $S^{k}\times \mathbb{R}^{n-k}$, the unique $\mathrm{SO}(k)\times \mathrm{SO}(n-k+1)$ symmetric oval with $S^{k-1}\times \mathbb{R}^{n-k+1}$ as tangent flow at $-\infty$, the unique $\mathrm{SO}(k)\times \mathrm{SO}(n-k+1)$ symmetric oval with $\mathbb{R}^{k}\times S^{n-k}$ as tangent flow at $-\infty$.} by removing the condition of rotational symmetry  along the directions of spherical factor  in the unique tangent flow at $-\infty$.

If $\textrm{rk}(Q)=0$ (fully-degenerate case) which is equivalent to  that unstable mode is dominant (see Section \ref{sec_quant_thm} for details), we obtain the following result.
\begin{theorem}[fully-degenerate case]\label{Rk=0}
Let $\mathcal{M}=\{M_t\}$ be an ancient noncollapsed mean curvature flow in $\mathbb{R}^{n+1}$ whose tangent flow at $-\infty$ is given by $\mathbb{R}^{k}\times S^{n-k}(\sqrt{2(n-k)|t|})$. If $\mathrm{rk}(Q)=0$ (unstable mode is dominant), then $\mathcal{M}=\{M_t\}$ is either a round shrinking cylinder $\mathbb{R}^k\times S^{n-k}$ or $\mathbb{R}^{k-1}$ times  $(n-k+1)$ dimensional round bowl soliton whose translating direction is orthogonal to $\mathbb{R}^{k-1}$ factor.
\end{theorem}
This theorem improves the results in \cite[Thm 1.10]{CHH_wing} and \cite[Cor 7.5]{DH_blowdown} by removing the noncompactness condition and uniformly three-convexity condition, and generalizes \cite[Thm 1.2]{DH_hearing_shape} to all dimensions. As an application, we can improve our previous result in \cite{DH_blowdown}. We first recall the following definition of blowdown of  the ancient noncollapsed
flow $M_{t}=\partial K_{t}$.
\begin{definition}[blowdown]\label{def_blowdown}
Given any time $t_0$, the \emph{blowdown} of $K_{t_0}$ is defined by
\begin{equation}
\check{K}_{t_0}:=\lim_{\lambda\to 0} \lambda\cdot K_{t_0}.
\end{equation}
\end{definition}
Then we have the following corollary of Theorem \ref{Rk=0}.
\begin{corollary}[blowdown in unstable mode case]\label{blowdown in unstable mode}
Let $M_t=\partial K_{t}$ be a strictly convex ancient noncollapsed mean curvature flow in $\mathbb{R}^{n+1}$ with $\textrm{rk}(Q)=0$ (unstable mode is dominant). Then $M_{t}$ is the $n$ dimensional
round bowl soliton, and in particular  its tangent flow at $-\infty$ is a neck cylinder $\mathbb{R}\times S^{n-1}(\sqrt{2(n-1)|t|})$ and for any time $t_{0}$ the blowdown set $\check{K}_{t_{0}}=\lim_{\lambda\to 0} \lambda K_{t_{0}}$ is the same half line and in particular
\begin{equation}
    \text{dim}\check{K}_{t_{0}}=1.
\end{equation}
\end{corollary}
Under $\textrm{rk}(Q)=0$ (unstable mode is dominant) assumption, this corollary  improves \cite[Thm 7.1]{DH_blowdown} by removing the uniformly three-convexity condition and reducing the upper bound of the blowdown dimension of strictly convex noncollapsed flows  in \cite[Thm 1.2]{DH_blowdown} from $n-2$ to $1$.

\subsection{Organization of the paper}
We organize the paper as follows:

In Section \ref{sec_fine_bubble_sheet}, we set up the fine cylindrical  analysis. Specifically, since $\bar{M}_\tau$ moves by renormalized mean curvature flow, the evolution of the cylindrical  profile function $u(\cdot,\tau)$ over cylinder $\Gamma=\mathbb{R}^{k}\times S^{n-k}(\sqrt{2(n-k)})$, as defined in \eqref{graph_over_cylinder}, is governed by the Ornstein-Uhlenbeck type operator
\begin{equation}
\mathcal L=\triangle_{\mathbb{R}^{k}}
+\frac{1}{2(n-k)}\triangle_{S^{n-k}}-\frac{1}{2}y\cdot\nabla_{\mathbb{R}^{k}}+1\, .
\end{equation}
This operator has the following unstable eigenfunctions, namely
\begin{align}\label{basis_hplus_intro}
1, y_1, y_{2}, \dots, y_{n+1}\, ,
\end{align}
and following neutral eigenfunctions, namely
\begin{align}\label{basis_hneutral_intro}
y^2_{1}-2,\,\, \dots,\,\, y^2_{k}-2,\,\, 2y_{1}y_{2},\,\, 2y_{1}y_{3}, \,\,\dots,\,\, 2y_{k-1}y_{k},
\end{align}
and
\begin{align}\label{basis_hneutral_intro2}
y_{1}y_{k+1},\,\, y_{2}y_{k+1},\,\,\dots\,\,y_{k}y_{n+1},
\end{align}
where $y_{1}, \dots, y_{n+1}$  denote the restriction of the Euclidean coordinate functions to the cylinder $\Gamma$. All the other eigenfunctions are stable. These eigenfunctions form an orthornomal basis of
the Hilbert space $\mathcal{H}$ of Gaussian $L^2$ functions on cylinder $\Gamma$ and give the following decomposition according to sign of modes
\begin{equation}
\mathcal H = \mathcal{H}_+\oplus \mathcal{H}_0\oplus \mathcal{H}_-.
\end{equation}
Let  $P_+, P_0, P_-$ be the orthogonal projections to $\mathcal{H}_+,\mathcal{H}_0,\mathcal{H}_-$, respectively and
\begin{align}
U_+(\tau) := \|P_+ \hat{u}(\cdot,\tau)\|_{\mathcal{H}}^2, \quad
U_0(\tau) := \|P_0 \hat{u}(\cdot,\tau)\|_{\mathcal{H}}^2,\label{def_U_PNM} \quad
U_-(\tau) := \|P_- \hat{u}(\cdot,\tau)\|_{\mathcal{H}}^2.
\end{align}
Then we improve the Merle-Zaag alternative estimates from \cite{MZ}  to a more quantitative version adapted to graphical radius $\rho$ as in \cite[Lem 2.1]{DH_no_rotation}: for $\tau\to -\infty$ we have either
\begin{equation}
    U_0+U_{-}\leq C\rho^{-1}U_+ \quad\quad \text{(unstable mode is dominant)},
\end{equation}
or
\begin{equation}\label{quantitative zero mode}
    U_-+U_{+}\leq C\rho^{-1}U_0 \quad\quad \text{(neutral mode is dominant)}.
\end{equation}
This improved Merle-Zaag estimates will be used for improved error estimates of evolution equation of profile function $u$, which  is an essential ingredient to show that the cylindrical matrix $Q$ in Theorem \ref{spectral theorem} is a constant matrix.

In the case where unstable mode is dominant, it is not hard to see that this is equivalent to $\textrm{rk}(Q)=0$ and the function $u$ decays exponentially, and thus the spectral quantization theorem holds with $Q=0$. Hence, we can focus on the case where neutral mode is dominant. Besides using the Lojasiewicz inequality from Colding-Minicozzi \cite{CM_uniqueness} and  rotated ADS-shrinkers as inner barriers \cite{ADS1, DH_blowdown},   we also need to use  KM-shrinkers as outer barriers \cite{KM}  and asymptotic slope estimates which are key ingredients for overcoming the difficulty of the lack of convexity or noncollapsing condition. Using these ingredients, we can show that the graphical radius
\begin{equation}\label{graphical radius>gamma_intro}
\rho(\tau)\geq |\tau|^{\gamma}
\end{equation}
for some $\gamma>0$, is an admissible graphical radius, i.e.
\begin{equation}\label{small_graph_intro}
\|u(\cdot,\tau)\|_{C^4(\Gamma \cap B_{2\rho(\tau)}(0))} \leq  \rho(\tau)^{-2}.
\end{equation}
\indent We use the cylindrical symmetry improvement theorem to show that the central region of the hypersurfaces is almost $\mathrm{O}(n-k+1)$-symmetric,  specifically that there is some $\eta>0$ such that for all $\tau\ll 0$ we have
\begin{equation}\label{eq_alm_symm_intro}
  \sup_{|y|\leq \rho(\tau)}   |\nabla_{S^{n-k}}(y,\omega,\tau)|\leq e^{-\eta |\tau|^{\gamma}} \, .
\end{equation}
This has an important consequence that out of the eigenfunctions listed in \eqref{basis_hplus_intro}-\eqref{basis_hneutral_intro2} only the the eigenfunctions \eqref{basis_hneutral_intro} can be dominant. Moreover, it also implies that if we perform the Taylor expansion to the evolution equation of the truncated cylindrical function
\begin{equation}
\hat{u}(y, \omega,\tau):=u(y, \omega, \tau) \chi\left(\frac{|y|}{\rho(\tau)}\right) \, ,
\end{equation}
where $\chi$ is a suitable cutoff function, then up to second order we have
\begin{equation}\label{sec_taylor_intro}
 \partial_{\tau}\hat{u}=\mathcal{L}\hat{u}-\frac{1}{2\sqrt{2(n-k)}}\hat{u}^2+\hat{E},
\end{equation}
where the error term   $\hat{E}$ satisfies
\begin{equation}
  \left\langle |\hat{E}|, 1+|y|^2\right\rangle_{\mathcal{H}}\leq C|\tau|^{-\gamma}\|\hat{u} \|_{\mathcal{H}}^2+ e^{-\eta|\tau|^{\gamma}/10}.
\end{equation}

In Section \ref{sec_quant_thm}, we prove Theorem \ref{spectral theorem} (spectral  quantization theorem)  without the noncollapsing or convexity assumption as in \cite[Prop 3.3]{DH_hearing_shape}.  To this end, we consider the expansion
\begin{equation}\label{expansion_123_intro}
 \hat{u} = \sum_{m=1}^{k}\alpha_{mm}(y^2_{m}-2)+\sum_{1\leq i<j\leq k} 2\alpha_{ij}y_{i}y_{j}+\hat{w}\, ,
 \end{equation}
 where the remainder term $\hat{w}$ is controllable thanks to the assumption that the neutral eigenfunctions from \eqref{basis_hneutral_intro} are dominant and thanks to the almost circular symmetry from \eqref{eq_alm_symm_intro}.
Taking also into account \eqref{sec_taylor_intro} and the improved quantitative Merle-Zaag estimates, we then show that the symmetric $k\times k$ spectral coefficients matrix $A(\tau)=(\alpha_{ij}(\tau))$ evolves by the following equation.
\begin{equation}\label{odes0_intro}
 \dot{A}(\tau)=-\beta_{n, k}^{-1} A^{2}(\tau)+O(|\tau|^{-\gamma/2}|A|^2 + e^{-\eta|\tau|^{\gamma}/10}),
\end{equation}
where
\begin{equation}
    \beta_{n, k}=\frac{\sqrt{2(n-k)}}{4},
\end{equation}
and the error estimates is improved compared with \cite[Prop 3.1]{DH_hearing_shape}. Then Theorem \ref{spectral theorem} is reduced to the problem of solving finite dimensional ODE dynamical system.

Motivated by \cite{FilippasLiu}, we will apply diagonalization argument to analyze the above coupled ODE dynamical system. Notice that the eigenvalues $\{\lambda_{i}(\tau)\}_{i=1}^{k}$ of $A(\tau)$ under suitable arrangement are continuously differentiable in time by \cite[Sec 2, Thm 6.8]{kato}.  These eigenvalues satisfy the following  ODE system.
\begin{equation}
    \frac{d}{d\tau}\lambda_{i}=-\beta^{-1}_{n,k}\lambda_{i}^2+o(\sum_{i=1}^{k}\lambda_{i}^2+e^{-\eta|\tau|^{\gamma}/10}) \quad 1\leq i\leq k.
\end{equation}
 Then we use the information from graphical radius and carefully carry out a continuity argument to show that
the error terms with exponential decay  in the above ODE system can be absorbed, and we prove an almost non-upward quadratic bending estimates
\begin{equation}
   Q(\tau)= \max\limits_{1\leq i\leq k}  \lambda_i(\tau) + \varepsilon \min\limits_{1\leq i\leq k} \lambda_i(\tau)<0
\end{equation}
for $\varepsilon>0$ small and $\tau$ negative enough, which means the positive eigenvalues cannot be dominant. Based on these facts, we obtained
\begin{equation}
\frac{c}{|\tau|}\leq \sum_{i=1}^{k}|\lambda_{i}|\leq \frac{C}{|\tau|}.
\end{equation}
Putting this into \eqref{odes0_intro} and using again the digonalization method, we obtain
\begin{equation}
    \frac{d}{d\tau}\lambda_{i}+\beta^{-1}_{n,k}\lambda_{i}^2= O(|\tau|^{-2-\frac{\gamma}{2}}) \quad 1\leq i\leq k,
\end{equation}
where the power  $-2-\frac{\gamma}{2}$  in the error term guarantees the integrability and is used to show that $Q$ is a constant matrix in Theorem \ref{spectral theorem} (spectral  quantization theorem). This is implied by our previous discussion on quantitative Merle-Zaag estimates \eqref{quantitative zero mode} and graphical radius estimates \eqref{graphical radius>gamma_intro} from Lojasiewicz inequality. Finally, by translating back to the discussion for spectral matrix $A$, this  completes  the proof of Theorem \ref{spectral theorem} (spectral  quantization theorem).

In Section \ref{proof_symmetry_improvement}, we will prove the general version of  symmetry improvement theorem in Theorem \ref{Symmetry improvement intro} as in \cite{Zhu2, Zhu, BC1, BC2}. In the cylindrical symmetry improvement case, this is obtained by   estimating Fourier modes of the linearized parabolic Jacobi equation for $\langle \bar{K}, \nu\rangle$ on any  cylinder $\mathbb{R}^k\times S^{n-k}$ ($1\leq k\leq n-1$) and adjusting rotational axes according to the estimates, where $\bar{K}$ is chosen from some normalized set of   rotation vector fields $\bar{\mathcal{K}} = \{\bar{K}_{\alpha}: 1 \leq \alpha \leq \frac{(n-k+1)(n-k)}{2}\}$ and $\nu$ is the normal vector field along the flow. Then we establish cap improvement theorem via iterating cylindrical symmetry improvement theorem, carefully setting up suitable barrier functions along $\mathbb{R}^{k-1}$ factor and applying maximum principle as in \cite[Thm 3.12]{Zhu2, Zhu, BC1, BC2}. The main differences  include more involved  estimates for higher dimensional parabolic equation, axes adjustment and Lemma \ref{Lin Alg Cylinder A}-Lemma \ref{VF close Bowl times lines}.

In Section \ref{rkksection}, we discuss the  ancient noncollapsed flows in $\mathbb{R}^{n+1}$ with full rank $k$ cylindrical matrix $Q$, which are called as $k$-ovals in $\mathbb{R}^{n+1}$.  We obtain the unique sharp asymptotics, compactness and $\mathrm{O}(n-k+1)$ symmetry of  $k$-ovals in $\mathbb{R}^{n+1}$. We first prove  that they have the same unique sharp asymptotics as the $\textrm{SO}(k)\times \textrm{SO}(n-k+1)$ symmetric ancient ovals in \cite[Thm 1.4]{DH_ovals}. The parabolic region sharp asymptotics follows from Theorem \ref{spectral theorem}. The intermediate region sharp asymptotics follows from barrier argument, extended almost symmetry estimates, convexity and method of characteristics.   For obtaining tip asymptotics we showed that  $\mathrm{O}(k-1)$ rotational symmetric $k$ dimensional convex cone $\mathcal{C}^{\infty}_{k}=\{(x, 0)\in \mathbb{R}^{k}\times\mathbb{R}^{n-k+1}: \{x_{1}^2\geq x^2_{2}+,\dots+x^2_{k}, x_{1}\leq 0\}$ is contained in the tip blowup limits of the $k$-ovals in $\mathbb{R}^{n+1}$. This guarantees that we can inductively apply the key ingredient \cite[Thm 1.4]{DH_blowdown} to show that the tip asymtotics looks   $\mathbb{R}^{n+1-k}$ times $(k-1)$ dimensional round bowl soliton. Finally, the $\mathrm{O}(n-k+1)$ symmetry  in Theorem \ref{compact_symmetry_oval} (asymptotics, compactness and symmetry of $k$-ovals in $\mathbb{R}^{n+1}$) follows from sharp asymptotics, convexity, cylindrical symmetry improvement and cap improvement as in \cite[Thm 6.1, Thm 6.7]{DH_hearing_shape}. In the end, we give the proof of Corollary \ref{symmetr classification}.

In Section \ref{rk0section}, we inductively prove Theorem \ref{Rk=0} (fully degenerate  case or equivalently dominant unstable mode case) according to the number of $\mathbb{R}$ factors in the tangent flow at $-\infty$. We assume the solution is not cylindrical. Then we  prove that Theorem \ref{Rk=0} holds for noncompact solutions by contradiction argument. More precisely, we blow up along two directions and obtain two limit flows which split off one line, then we use    the structure of $\mathbb{R}^{k-2}$ times the $(n-k+1)$ dimensional round bowl soliton from the induction assumption and the fine cylindrical theorem in \cite[Thm 6.4]{DH_blowdown} to derive a contradiction. Finally, we exclude the possibility of compact solutions. This completes the proof of Theorem \ref{Rk=0}. In  the end of this section, we give the proof of Corollary \ref{blowdown in unstable mode}.

\bigskip

\noindent\textbf{Acknowledgments.}
We appreciate the communication with Professor Robert Haslhofer. The first author has been supported by the NSERC Discovery Grant and the Sloan Research Fellowship of Professor Robert Haslhofer.
\bigskip
\section{Fine cylindrical analysis}\label{sec_fine_bubble_sheet}

\subsection{Basic cylindrical setup}
Let $M_t$ be an ancient asymptotically cylindrical flow in $\mathbb{R}^{n+1}$. Then its tangent flow at $-\infty$ in suitable coordinates is
\begin{equation}\label{bubble-sheet_tangent}
\lim_{\lambda \rightarrow 0} \lambda M_{\lambda^{-2}t}=\mathbb{R}^{k}\times S^{n-k}(\sqrt{2(n-k)|t|}) \quad\quad (1\leq k\leq n-1).
\end{equation}
In other words, the renormalized mean curvature flow,
\begin{equation}
\bar M_\tau = e^{\frac{\tau}{2}}  M_{-e^{-\tau}},
\end{equation}
converges as $\tau\to -\infty$ to the cylinder
\begin{equation}
\Gamma=\mathbb{R}^{k}\times S^{n-k}(\sqrt{2(n-k)}).
\end{equation}
Assume further that $M_t$ is not a round shrinking cylinder.
Let us fix some admissible graphical radius function $\rho(\tau)$ for $\tau\ll 0$, namely a positive function satisfying
\begin{equation}\label{univ_fns}
\lim_{\tau \to -\infty} \rho(\tau)=\infty, \quad \textrm{and}\quad  -\rho(\tau) \leq 5\rho'(\tau) \leq 0,
\end{equation}
so that $\bar M_\tau$ can be written as a graph of a function $u(\cdot,\tau)$ over $\Gamma\cap B_{2\rho(\tau)}$ with the estimate
\begin{equation}\label{small_graph}
\|u(\cdot,\tau)\|_{C^4(\Gamma \cap B_{2\rho(\tau)}(0))} \leq  \rho(\tau)^{-2}.
\end{equation}
Since $\bar{M}_\tau$ moves by renormalized mean curvature flow, the cylindrical function $u$ evolves by
 \begin{equation}
      \partial_{\tau}u=\mathcal{L}u+E,
 \end{equation}
where $\mathcal L$ is the Ornstein-Uhlenbeck  operator on $\Gamma=\mathbb{R}^{k}\times S^{n-k}(\sqrt{2(n-k)})$ explicitly given  by
\begin{equation}
\mathcal L=\mathcal L=\triangle_{\mathbb{R}^{k}}
+\frac{1}{2(n-k)}\triangle_{S^{n-k}}-\frac{1}{2}y\cdot\nabla_{\mathbb{R}^{k}}+1,
\end{equation}
and the error term $E$  satisfies the pointwise estimate
\begin{equation}\label{error_first_untrunc}
    |E|\leq C\rho^{-2}(|u|+|\nabla u|)
\end{equation}
thanks to \eqref{small_graph}.

\noindent Denote by $\mathcal{H}$ the Hilbert space of Gaussian $L^2$ functions on $\Gamma$, where
\begin{equation}\label{def_norm}
\langle f, g\rangle_{\mathcal{H}}= \frac{1}{(4\pi)^{n/2}} \int_{\Gamma}  f(q)g(q) e^{-\frac{|q|^2}{4}} \, dq\, .
\end{equation}
We also fix a nonnegative smooth cutoff function $\chi$ satisfying $\chi(s)=1$ for $|s| \leq \tfrac{1}{2}$ and $\chi(s)=0$ for $|s| \geq 1$, and consider the truncated function
\begin{equation}
\hat{u}(y,\omega,\tau):=u(y,\omega,\tau) \chi\left(\frac{|y|}{\rho(\tau)}\right).
\end{equation}
Then, we have the following proposition.
\begin{proposition}[{truncated evolution, cf. \cite[Proposition 2.1]{DH_hearing_shape}}]\label{prop_trunc_bubble}
The truncated cylindrical function $\hat{u}$ satisfies
\begin{equation}\label{eq_truncated}
\left\|(\partial_\tau -\mathcal{L})\hat{u} \right\|_{\mathcal{H}}\leq C\rho^{-1} \| \hat{u} \|_{\mathcal{H}}.
\end{equation}
\end{proposition}
Analyzing the spectrum of $\mathcal L$, we can decompose our Hilbert space as
\begin{equation}
\mathcal H = \mathcal{H}_+\oplus \mathcal{H}_0\oplus \mathcal{H}_-,
\end{equation}
where the unstable space $\mathcal{H}_+$ is the eigen-space of positive eigenvalues of $\mathcal{L}$ and is spanned by
\begin{align}\label{basis_hplus}
1, y_1, y_{2}, \dots, y_{n+1}\, ,
\end{align}
 the neutral space $\mathcal{H}_0$  is the eigen-space of zero eigenvalue of $\mathcal{L}$ and is spanned by
\begin{align}\label{basis_hneutral}
y^2_{1}-2,\,\, \dots,\,\, y^2_{k}-2,\,\, 2y_{1}y_{2},\,\, 2y_{1}y_{3}, \,\,\dots,\,\, 2y_{k-1}y_{k},
\end{align}
and
\begin{align}\label{basis_hneutral2}
y_{1}y_{k+1},\,\, y_{2}y_{k+1},\,\,\dots\,\,y_{k}y_{n+1}
\end{align}
and the stable space $\mathcal{H}_0$  is the eigen-space of all negative eigenvalues of $\mathcal{L}$.

Then, we consider the functions
\begin{equation}
U_{\pm}(\tau):= \|P_{\pm} \hat{u}(\cdot,\tau)\|_{\mathcal{H}}^2, \qquad U_0(\tau):= \|P_0 \hat{u}(\cdot,\tau)\|_{\mathcal{H}}^2\, ,
\end{equation}
where ${P}_{\pm}$ and ${P}_0$ denote the orthogonal projections to $\mathcal{H}_{\pm},\mathcal{H}_0$, respectively. Since the positive eigenvalues of $\mathcal{L}$ have  $\frac{1}{2}$ as lower bound and negative eigenvalues of $\mathcal{L}$ have $-\frac{1}{2}$ as upper bound, it is not difficult to see that they satisfy
\begin{align}\label{U_PNM_system}
 \dot{U}_+ &\geq U_+ - C_{0}\rho^{-1} \, (U_+ + U_0 + U_-), \nonumber\\
 | \dot{U}_0  | &\leq C_{0}\rho^{-1} \, (U_+ + U_0 + U_-), \\
 \dot{U}_- &\leq -U_- + C_{0}\rho^{-1} \, (U_+ + U_0 + U_-). \nonumber
\end{align}
Using this  we improve the Merle-Zaag alternative estimates from \cite{MZ}.
\begin{proposition}[{quantitative Merle-Zaag alternative, c.f \cite[Lem 2.1]{DH_no_rotation}}]\label{mz.ode.fine.bubble-sheet}
For $\tau\to -\infty$, either the neutral mode is dominant, i.e.
\begin{equation}\label{case_1}
U_-+U_{+}\leq C\rho^{-1}U_0,
\end{equation}
or the unstable mode is dominant, i.e.
\begin{equation}\label{case_2}
U_-+U_0\leq C\rho^{-1}U_+.
\end{equation}
\end{proposition}
\begin{proof}
We adapt the argument in \cite{MZ}, but with some improvement to obtain a better decay. Specifically, here we work with the time-dependent function $\bar{\varepsilon}(\tau)=C_{0}\rho(\tau)^{-1}$, where $C_0$ is the constant from \eqref{U_PNM_system}.  Then $\dot{\bar{\varepsilon}}=-\dot{\rho}\rho^{-1}{\bar{\varepsilon}}\in (0, \frac{\bar{\bar{\varepsilon}}}{5})$ by definition of admissible graphical radius.   By possibly decreasing $\tau_\ast$, we may assume that $\bar{\varepsilon}<1/100$. We can also assume that the constant $C_{0}=1$ in \eqref{U_PNM_system} by rescaling in time. We will first show that
\begin{equation}\label{-leq0+}
    U_{-}\leq 2{\bar{\varepsilon}}(\tau)(U_{0}+U_{+}).
\end{equation}
Indeed, if at some time $\bar{\tau}\leq \tau_{0}$ the quantity $f:=U_{-}-2{\bar{\varepsilon}}(U_{+}+U_{0})$ was positive, then at this time we would have
\begin{equation}
    \dot{f}\leq -U_{-}+{\bar{\varepsilon}}(1+4{\bar{\varepsilon}})\left(1+\frac{1}{2{\bar{\varepsilon}}}\right)U_{-}-2\dot{{\bar{\varepsilon}}}(U_{+}+U_{0})\leq 0,
\end{equation}
which would imply that $ f(\tau)\geq f(\bar{\tau})>0$
for all $\tau\leq \bar{\tau}$, contradicting with $\lim_{\tau\rightarrow -\infty}f(\tau)=0$. This proves \eqref{-leq0+}.

To conclude the proof, we consider $g(\tau)=8{\bar{\varepsilon}}(\tau)U_{0}-U_{+}$. Then,  two cases can happen:
(1) there are $\tau_{i}\to -\infty$ such that $g(\tau_{i})\geq 0$ or (2) $g(\tau)<0$ for $\tau$ negative enough.  Now, we compute $\dot{g}(\tau)$ at time $\tau$ such that $g(\tau)=0$. Direct computation and \eqref{U_PNM_system} imply that if $\tau$ is negative enough we have
\begin{align}
    \dot{g}(\tau)&\leq 16{\bar{\varepsilon}}^2(U_{0}+U_{+})-\frac{1}{2}U_{+}+2{\bar{\varepsilon}} (U_{+}+U_{0})+8\dot{{\bar{\varepsilon}}} U_{0}\\\nonumber
    &\leq-U_{+}(\tau)(\frac{1}{4}-4{\bar{\varepsilon}}-16{\bar{\varepsilon}}^2-1/5)<0.
\end{align}
This implies that in the case (1) there is some $\tau_{*}\ll 0$ negative enough such that
\begin{equation}\label{case2}
    U_{+}<8{\bar{\varepsilon}}(\tau)U_{0}.
\end{equation}
Now, we discuss case (2). By  \eqref{U_PNM_system} and \eqref{-leq0+}, we have
\begin{equation}\label{two inequalities}
    \dot{U}_{+}\geq \frac{1}{2}U_{+}-2{\bar{\varepsilon}} U_{0}\quad |\dot{U}_{0}|\leq 2{\bar{\varepsilon}}(U_{0}+U_{+}).
\end{equation}
This and \eqref{case2} imply
\begin{equation}\label{U0varepsilon14}
    \dot{U}_{+}\geq \frac{1}{4}U_{+}=\frac{1}{4{\bar{\varepsilon}}}({\bar{\varepsilon}} U_{+})\quad \dot{U}_{0}\leq (2{\bar{\varepsilon}}+\frac{1}{4})U_{+}.
\end{equation}
Hence
\begin{equation}\label{U+geq14}
   {U}_{+}(\tau)\geq \int_{-\infty}^{\tau}\frac{1}{4}U_{+}(s)ds,
\end{equation}
and by monotonicity of ${\bar{\varepsilon}}$ for $\tau'< \tau\ll 0$
\begin{align}
    U_{+}(\tau)-U_{+}(\tau')\geq \frac{1}{4{\bar{\varepsilon}}(\tau)} \int_{\tau'}^{\tau}{\bar{\varepsilon}}(s) U_{+}(s)ds.
\end{align}
Sending $\tau'\to -\infty$ we have
\begin{align}\label{epsillonU+}
    U_{+}(\tau)\geq \frac{1}{4{\bar{\varepsilon}}(\tau)} \int_{-\infty}^{\tau}{\bar{\varepsilon}}(s) U_{+}(s)ds.
\end{align}
Then by \eqref{U0varepsilon14} \eqref{U+geq14} and \eqref{epsillonU+} we have
\begin{equation}
    U_{0}\leq U_{+}+2\int_{-\infty}^{\tau}{\bar{\varepsilon}}(s)U_{+}(s)ds\leq(1+8{\bar{\varepsilon}})U_{+}.
\end{equation}
Inserting this into the second inequality of \eqref{two inequalities}, we have
\begin{align}
    \dot{U}_{0}\leq 2{\bar{\varepsilon}}(2+8{\bar{\varepsilon}})U_{+}\leq 20{\bar{\varepsilon}} U_{+}.
\end{align}
Integrating this and using \eqref{epsillonU+} we obtain that
\begin{equation}\label{U0<eU+}
    U_{0}\leq 80{\bar{\varepsilon}} U_{+}
\end{equation}
holds in case (2).
Finally, \eqref{-leq0+} and \eqref{U0<eU+} imply \eqref{case_1}, \eqref{-leq0+} and \eqref{case2} imply \eqref{case_2}. This completes the proof of the Proposition \ref{mz.ode.fine.bubble-sheet}.
\end{proof}
In the end of this subsection, we remark that for any other choice of admissible graphical radius the same mode stays dominant.
\subsection{Barrier construction and asymptotic slope}
In this subsection, we consider some basic barrier construction and asymptotic slope of ancient asymptotically cylindrical flows. We first  list some general facts about barriers, which we will use in later sections. By \cite[Section 4]{ADS1} and \cite{KM} there is some $L>1$ such that  for every $a\geq L, b>0$ there are  ADS-shrinkers and KM-shrinkers in $\mathbb{R}^{n+2-k}$:
\begin{align}\label{ads_shrinker}
{\Sigma}_a &= \{\textrm{surface of revolution with profile } r=u_a(y_1), 0\leq y_1 \leq a\},\\
\tilde{{\Sigma}}_b &= \{ \textrm{surface of revolution with profile } r=\tilde{u}_b(y_1), 0\leq y_1 <\infty\}.\nonumber
\end{align}
Here, the parameter $a$ captures where the  concave functions $u_a$ meet the $y_1$-axis, namely $u_a(a)=0$, and the parameter $b$ is the asymptotic slope of the convex functions $\tilde{u}_b$, namely $\lim_{x_1\to \infty} \tilde{u}_b'(x_1)=b$.  In \cite[Section 3]{DH_blowdown} the 2d ADS-shrinkers and KM-shrinkers have been $\eta$-shifted ($\eta>0$) and rotated to construct the  hypersurfaces
\begin{align}\label{rotated_barrier}
&\Gamma^{\eta}_a=\{(y, y'')\in  \mathbb{R}^{k}\times \mathbb{R}^{n+1-k}:  (|y|-\eta, y'') \in {\Sigma}_a \},\\
&\tilde{\Gamma}_{b}^\eta=\{(y, y'')\in  \mathbb{R}^{k}\times \mathbb{R}^{n+1-k}:  (|y|-\eta, y'')\in \tilde{\Sigma}_{b}\},\nonumber
\end{align}
where we denote
\begin{equation}
    y=(y_{1},\dots, y_{k})\quad y''=(y_{k+1}, \dots, y_{n+1}).
\end{equation}
By construction the hypersurfaces $\Gamma^{\eta}_a$ and $\tilde{\Gamma}^{\eta}_{b}$ have the following foliation property.
\begin{lemma}[{foliation lemma, \cite[Lem 3.3]{DH_blowdown}}]\label{foli_lemma}
There exist $\delta>0$ and $L_1>2$ such that the hypersurfaces ${\Gamma}^{\eta}_a$, ${\tilde{\Gamma}}^{\eta}_b$, and the cylinder ${\Gamma}:=\mathbb{R}^{k}\times S^{n-k}(\sqrt{2(n-k)})$ foliate the domain
\begin{equation*}
{\Omega}:=\left\{(y, y'') \in \mathbb{R}^{k}\times \mathbb{R}^{n-k+1}:  |y|\geq L+1,\, |y''|^2 \leq 2(n-k)+\delta\right\}.
\end{equation*}
Moreover, denoting by $\nu_{\mathrm{fol}}$ the outward unit normal vector field of this foliation, we have
\begin{equation}\label{negative_div}
\mathrm{div}(\nu_{\mathrm{fol}}e^{-|x|^2/4})\leq 0\;\;\;\textrm{inside the cylinder},
\end{equation}
and
\begin{equation}\label{positive_div}
\mathrm{div}(\nu_{\mathrm{fol}}e^{-|x|^2/4})\geq 0\;\;\;\textrm{outside the cylinder}.
\end{equation}
\end{lemma}
As a corollary of this lemma, the rotated ADS shrinkers act as inner barriers and rotated KM shrinkers act as outer barriers for the renormalized mean curvature flow.
\begin{corollary}[{inner-outer barriers, \cite[Corollary 3.4]{DH_blowdown}}]\label{emma_inner_barrier}
We consider compact domains $\{K_{\tau}\}_{\tau\in [\tau_1,\tau_2]}$, whose boundary evolves by renormalized mean curvature flow. If ${\Gamma}^{\eta}_a$ is contained in $K_{\tau}$ ($\tilde{\Gamma}^{\eta}_b$ is in the closure of $K^{c}_{\tau}$) for every $\tau\in [\tau_1,\tau_2]$, and $\partial K_{\tau}\cap {\Gamma}^{\eta}_a=\emptyset$ ($\partial K_{\tau}\cap \tilde{\Gamma}^{\eta}_b=\emptyset$) for all $\tau<\tau_2$, then
\begin{equation}
\partial K_{\tau_2}\cap {\Gamma}^{\eta}_a\subseteq \partial {\Gamma}_a \quad(\partial K_{\tau_2}\cap {\Gamma}^{\eta}_b\subseteq \partial {\Gamma}_b).
\end{equation}
\end{corollary}
Then we  discuss asymptotic slope of ancient asymptotically cylindrical flows.
The following proposition is important for applying outer KM-barrier argument in later discussion.
\begin{proposition}[{asymptotic slope, cf.\cite[Cor 3.10]{CHH}}]\label{slopecor}
There is smooth function $\bar{\varphi}: \mathbb{R}_{+}\to \mathbb{R}_{+}$ with property $\lim_{\rho\to +\infty}\bar{\varphi}'(\rho)=0$ such that for each $X_{0}=(p_{0}, t_{0})\in \mathcal{M}$, the renormalized flow $\bar{M}^{X_{0}}_{\tau}=e^{\tau/2}(M_{t_{0}-e^{-\tau}}-p_{0})$ satisfies
\begin{equation}
    \bar{M}^{X_{0}}_{\tau}\subset \{(y, y'')\in \mathbb{R}^{k}\times\mathbb{R}^{n-k+1}: |y''|\leq \bar{\varphi}(|y|) \}
\end{equation}
for all $\tau\leq \mathcal{T}(X_{0})$ where $\mathcal{T}(X_{0})>-\infty$ is a constant that only depends on $X_{0}$. In particular, any potential ends must be along directions of $\mathbb{R}^{k}$ with $|y|\to +\infty$.
\end{proposition}
\begin{proof}
Suppose that the conclusion is not true, then there exists $\delta>0$
and a sequence of points $(p_i,t_i)$ with $p_i\in M_{t_i}$, $t_i\leq -1$ and
$\frac{|p_i|}{\sqrt{-t_i}}\rightarrow \infty$, but
\begin{equation}\label{slope claim copy}
    \sum_{i=k+1}^{n+1}x^2_{i} >  \delta \sum_{i=1}^{k}x^2_{i}
\end{equation}
holds at each $p_i$.
Now we rescale the flow around the space-time origin by the factor of $\frac{1}{|p_i|}$ to get $\mathcal{M}^i$, i.e.
\begin{align}
    \mathcal{M}^i = \mathcal{D}_{\frac{1}{|p_i|}}\mathcal{M},
\end{align}
where $(p_i,t_i)$ becomes $\mathcal{D}_{\frac{1}{|p_i|}}(p_i,t_i) = (\frac{p_i}{|p_i|}, \frac{t_i}{|p_i|^2})$.
By passing to a subsequence we may assume that $(\frac{p_i}{|p_i|}, \frac{t_i}{|p_i|^2})\rightarrow (q,0)$ with some $|q|=1$ satisfying
(\ref{slope claim copy})
Since the tangent flow of $\mathcal{M}$ at $-\infty$ is a shrinking cylinder,
$\mathcal{M}^i$ would converge smoothly to $\mathcal{M}^{\infty}$ on each compact set
of $\mathbb{R}^{n+1}\times (-\infty,0)$ and $\mathcal{M}^{\infty}$ is a shrinking cylinder on all negative times.
Using suitable barriers, this implies that  $M^i_0$ converges in Gromov-Hausdorff sense to a subset of $\mathbb{R}^k\times \{0\}$. In particular, $q\in \mathbb{R}^k\times \{0\}$.
However, this is impossible since $q$ satisfies (\ref{slope claim copy}).
\end{proof}
\subsection{Graphical radius}
Throughout this subsection we assume that the neutral mode is dominant. The goal is to construct an improved graphical radius by generalizing \cite[Section 4.3.1]{CHH} to the  cylindrical setting. To this end, we denote by $\rho_0$ the initial choice of graphical radius from the previous subsection, and consider the quantities
\begin{equation}\label{def_alphabar}
    \bar{\alpha}(\tau):=\sup_{\sigma\leq \tau}\left(\int_{\Gamma\cap \{|y|\leq L\}}{u}^2(q,\sigma)e^{-\frac{|q|^2}{4}}dq\right)^{1/2}
\end{equation}
and
\begin{equation}\label{def_beta}
    \beta(\tau):=\sup_{\sigma\leq \tau}\left(\int_{\Gamma}{u}^2(q,\sigma)\chi^2\left(\frac{|q|}{\rho_{0}(\sigma)}\right)e^{-\frac{|q|^2}{4}}dq\right)^{1/2},
\end{equation}
where $L$ is the large enough constant from the shrinker foliation \eqref{ads_shrinker}.

Using inverse Poincare inequality from \cite{DH_blowdown} as in \cite[Lem 4.17]{CHH}, we infer that
\begin{equation}\label{alpha_beta_equivalence}
    \bar{\alpha}(\tau)\leq \beta(\tau) \leq C\bar{\alpha}(\tau).
\end{equation}
Then, we let
\begin{equation}\label{graphrho}
    \rho(\tau):=\beta(\tau)^{-\frac{1}{5}},
\end{equation}
and we have:
\begin{proposition}[{admissibility, cf. \cite[Proposition 2.3]{DH_hearing_shape}}]\label{prop_admiss}
The function $\rho$ is an admissible graphical radius function, i.e. the estimates \eqref{univ_fns} and \eqref{small_graph} hold for $\tau\ll 0$.
\end{proposition}
\begin{proof}
Since $\bar{M}_\tau$ for $\tau\to -\infty$ converges locally uniformly to $\Gamma$, it is clear that
\begin{equation}
\lim_{\tau \to -\infty} \rho(\tau)=\infty\, .
\end{equation}
To proceed, we need the following barrier estimate in the case where no convexity condition is assumed compared with \cite[Claim 2.4]{DH_hearing_shape}.
\begin{claim}[barrier estimate]
There is a constant $C<\infty$ such that
\begin{equation}
 |u (y,\omega,\tau)| \leq C\beta(\tau)^{\frac{1}{2}}
\end{equation}
holds for $|y| \leq C^{-1}\beta(\tau)^{-\frac{1}{4}}$ and $\tau\ll 0$.
\end{claim}
\begin{proof}[Proof of claim]
Let $L$ be the fixed constant from the shrinker foliation \eqref{ads_shrinker} which can be adjusted large enough.
By standard parabolic estimates there is some constant $K<\infty$ such that for $\tau\ll 0$ we have
\begin{equation}\label{C_0.bound.y=L_0}
\sup_{|y|\leq 2L}|u(y, \omega,\tau)|\leq K \beta(\tau).
\end{equation}
Using rotated and shifted ADS-shrinker foliation in \eqref{rotated_barrier} as inner barrier as in \cite[Claim 2.4]{DH_hearing_shape}, one can obtain following lower bound barrier estimate
\begin{equation}
 u (y, \omega, \tau) \geq -C\beta(\tau)^{\frac{1}{2}}
\end{equation}
holds for $|y| \leq C^{-1}\beta(\tau)^{-\frac{1}{4}}$ and $\tau\ll 0$.

 Then, we establish the desired upper bound barrier estimate as in \cite[Prop 4.18]{CHH}. By \cite[Lem 4.11, Sec 8.8]{ADS1}, we know that for $r\geq L$ large enough the profile function $\tilde{u}_{b}(r)$ of KM-shrinker $\tilde{\Gamma}_{b}$ satisfies
\begin{equation}
    \tilde{u}_{b}'(r)=\frac{w(r)}{2r}(\tilde{u}_{b}^2(r)-2(n-k)),
\end{equation}
where $w$ is a function that satisfies
\begin{equation}
    2\leq w(r)\leq 2+\frac{16}{r^2}.
\end{equation}
This implies
\begin{equation}
    0\leq \frac{d}{dr}\log \left(\frac{\tilde{u}_{b}^{2}(r)-2(n-k)}{r^2} \right)\leq \frac{16}{r^3}.
\end{equation}
Integrating this differential inequality and using that the asymptotic slope of $\tilde{u}_{b}$ equals $b>0$, we obtain
\begin{equation}\label{ub estimates}
    2(n-k)+e^{-1}b^2r^2\leq \tilde{u}^2_{b}(r)\leq 2(n-k)+b^2r^2.
\end{equation}
Therefore, for $b>0$ small enough we have
\begin{equation}
   e^{-1}b^2L^2\leq \tilde{u}_{b}(L)^2-2(n-k)\leq 4\sqrt{n-k}(\tilde{u}_{b}(L)-\sqrt{2(n-k)}).
\end{equation}
Now, for $\hat{\tau}\ll 0$ we take $b^2=4(n-k)eL^{-2}_{0}K\beta(\hat{\tau})$ and we have
\begin{equation}
    K\beta(\hat{\tau})\leq  \tilde{u}_{b}(L)-\sqrt{2(n-k)}.
\end{equation}
Using this and \eqref{C_0.bound.y=L_0}, we have that $\tilde{\Gamma}_{b}\cap \{|y|= L\}$ is located outside of $\bar{M}_{\tau}\cap \{|y|= L\}$ for all $\tau\leq \hat{\tau}\ll 0$. Then by the vanishing asymptotic slope of $\bar{M}_{\tau}$ from Proposition \ref{slopecor} and asymptotic slope $b>0$ of $\tilde{u}_{b}$, we know that $\tilde{\Gamma}_{b}\cap \{|y|= h\}$  is located outside of $\bar{M}_{\tau}\cap \{|y|= h\}$ for $h>0$ large enough depending on $\hat{\tau}\ll 0$.  Finally, $\tilde{\Gamma}_{b}\cap \{L\leq |y|\leq h\}$  is located outside of $\bar{M}_{\tau}\cap \{L\leq |y|\leq h\}$ for $\tau$ negative enough. Hence, by avoidance principle, $\tilde{\Gamma}_{b}\cap \{|y|\geq L\}$  is located outside of $\bar{M}_{\tau}\cap \{|y|\geq L\}$ for $\tau\leq \hat{\tau}\ll 0$. By \eqref{ub estimates}, we also have
\begin{equation}
    \tilde{u}^2_{b}(1/\sqrt{b})\leq 2(n-k)+b.
\end{equation}
Using this and KM-shrinker $\tilde{\Gamma}_{b}$ as outer barrier,  we obtain
\begin{equation}
 u (y, \omega, \tau) \leq C\beta(\tau)^{\frac{1}{2}}
\end{equation}
holds for $|y| \leq C^{-1}\beta(\tau)^{-\frac{1}{4}}$ and $\tau\ll 0$. This completes the proof of the claim.
\end{proof}
Now, remembering that  $\rho(\tau)=\beta(\tau)^{-\frac{1}{5}}$, by the claim and standard interior estimates of unrescaled flow as in \cite[Lem 4.16]{oval_classification_r4} we get
\begin{equation}
\|u(\cdot,\tau)\|_{C^4(\Sigma\cap B_{2\rho(\tau)}(0))}\leq \rho(\tau)^{-2}
\end{equation}
for $\tau\ll 0$. Moreover, by definition of $\beta$ we clearly have $\dot{\rho}\leq 0$. Finally, using the assumption that the neutral eigenfunctions dominate and the weighted $L^2$-estimate from \cite[Proposition 4.1]{DH_blowdown} as in \cite[(186)]{CHH} we see that
\begin{equation}
\left|\frac{d}{d\tau} \beta^2\right| = o(\beta^2),
\end{equation}
which implies $ |\dot{\rho}|\leq \rho/5$ for $\tau \ll 0$.  This finishes the proof of proposition.
\end{proof}

\begin{proposition}[{graphical radius, c.f.  \cite[Proposition 2.4]{DH_hearing_shape}}]\label{radius_lower_bound}
There is a constant $\gamma>0$, so that $\rho(\tau)=\beta(\tau)^{-\frac{1}{5}}$ for $\tau\ll 0$ satisfies
\begin{equation}\label{gamma radius}
    \rho(\tau)\geq |\tau|^{\gamma}.
\end{equation}
\end{proposition}
\begin{proof}
Consider the Gaussian area functional
\begin{equation}
F(M)=(4\pi)^{-n/2}\int_M e^{-\frac{|q|^2}{4}} \, .
\end{equation}
By Colding-Minicozzi \cite[Theorem 6.1, Lem 6.9]{CM_uniqueness}, there exist $\nu\in (0,1/2)$ such that
\begin{equation}
\left(F(\Gamma)-F(\bar{M}_{\tau})\right) \leq C|\tau|^{-1/\eta},
\end{equation}
and
\begin{equation}\label{tail_decay_sum}
\sum_{j=J}^{\infty} \left(F(\bar{M}_{-j-1})-F(\bar{M}_{-j})\right)^{1/2} \leq C(\nu)J^{-\nu}
\end{equation}
holds. By Cauchy-Schwarz inequality, \eqref{tail_decay_sum} and Huisken's monotonicity formula, we obtain
\begin{equation}
\int_{-\infty}^{\tau}\int_{\bar{M}_{\tau'}} \left|H(q)+\frac{q^{\perp}}{2}\right|e^{-\frac{|q|^2}{4}} d\mu_{\tau'}(q)\, d\tau' \leq C|\tau|^{-\nu}\, .
\end{equation}
This estimates, \cite[Lemma A.48]{CM_uniqueness}, Proposition \ref{prop_admiss} (admissibility) and weighted $L^2$-estimate from \cite[Prop 4.1]{DH_blowdown} imply that the assertion holds with $\gamma=\nu/20$.

\end{proof}

\subsection{Almost cylindrical symmetry}
The goal of this subsection is to prove that the spherical derivative $\nabla_{S^{n-k}}u$ decays exponentially as $\tau \to -\infty$. To show this we will generalize the arguments from \cite[Section 4.3.3]{CHH} to the cylindrical setting.

As in the previous subsection we assume that the neutral mode is dominant. Thanks to Proposition \ref{radius_lower_bound} (graphical radius) we can from now on work with the graphical radius
\begin{equation}
\rho(\tau):=|\tau|^\gamma.
\end{equation}
The (cylindrical) symmetry improvement theorem  in Theorem \ref{Symmetry improvement intro} and Remark \ref{sym remark} imply that we can find constants $\eps_0>0$ and $L_{0}<\infty$ with the following significance. Let $\varepsilon\leq\varepsilon_{0}$ and $X\in \mathcal{M}$. If every $X'\in \mathcal{M}\cap P(X,  H^{-1}(X) L_0)$ is $\varepsilon_{0}$-close
 to a cylinder and $(\varepsilon, n-k)$-symmetric (see Definition \ref{e_symmetric}),  then $X$ is $(\eps/2, n-k)$-symmetric. (for proof of Theorem \ref{Symmetry improvement intro}, see Section \ref{proof_symmetry_improvement}).
\begin{lemma}[{spherical symmetric improvement, cf.\cite[Lemma 4.23]{CHH}}]\label{improvement}
There exists a constant $\eta_{0}>0$, so that for $\tau\ll0$ all space-time points $X\in \mathcal{M}$ corresponding to points in $\bar{M}_{\tau}\cap \{|y|\leq \rho(\tau)\}$ are $(2^{-\eta_{0}\rho(\tau)}, n-k)$-symmetric.
\end{lemma}
\begin{proof}
Let $\varepsilon_{0}>0$, $ L_0<\infty$ be the constants from cylindrical symmetric improvement in Theorem \ref{Symmetry improvement intro}. Let $\tau_\ast$ be sufficiently negative and $\tau_0\leq\tau_\ast$. We set $R:=2\rho(\tau_0)$. Let $q\in \bar{M}_{\tau}$ be a point with $|y|\leq R-2 L_0$ and $\tau \leq \tau_0$, and denote by $X\in \mathcal{M}$ the corresponding space-time point in the unrescaled flow. Using \eqref{small_graph} we see that every $X'\in \mathcal{M}\cap P(X, H^{-1}(X) L_0)$ is $\eps_0$-close to a cylinder $\mathbb{R}^{k}\times S^{n-k}$ and is $(\eps_0, n-k)$-symmetric. Hence, by Theorem \ref{Symmetry improvement intro} the space-time point $X$ is $(\eps_0/2, n-k)$-symmetric. Similarly, for any $q\in \bar{M}_{\tau}$ with $|y| \leq R-4L_0$
and $\tau\leq \tau_0$, the argument above shows that the corresponding space-time point $X\in \mathcal{M}$ is $(\eps_0/4, n-k)$-symmetric.
Iterating this $k$ times, where $k$ is the largest integer such that $2kL_0<\frac{R}{2}$, we  get that for all $q\in \bar{M}_{\tau}$ with $|y| \leq R/2$ and $\tau\leq \tau_0$ the corresponding space-time point $X\in\mathcal{M}$ is $(\eps_0/2^k, n-k)$-symmetric.  Choosing $\eta_0=1/ 4L_0$, this proves the lemma.
\end{proof}

\begin{proposition}[{almost spherical symmetry}]\label{utheta}
There exists a constant $\eta\in(0, 1/5)$, such  that for all $\tau\ll 0$ we have
\begin{equation}\label{eta small}
    |\nabla_{S^{n-k}}u|\leq e^{-\eta \rho(\tau)},
\end{equation}
whenever $|y|\leq \rho(\tau)$.
\end{proposition}
\begin{proof}
For each $\tau\leq \tau_\ast$ choose a point $q_\tau\in \bar{M_\tau}$ with $y=0$. By Lemma \ref{improvement} (spherical symmetric improvement) the corresponding space-time point $X_\tau=(x_\tau,-e^{-\tau})\in\mathcal{M}$ in the unrescaled flow is $(2^{-\eta\rho(\tau)}, n-k)$-symmetric, i.e. there exists a normalized set of rotation vector fields $K_{\alpha, \tau}$ where $\alpha=1,\dots, \frac{(n-k+1)(n-k)}{2}$ such that
\begin{equation}
|K_{\alpha, \tau}|H\leq 5n \,\,\, \textrm{and}\,\, \, |\langle K_{\alpha, \tau}, \nu\rangle|H\leq 2^{-\eta_0\rho(\tau)}   \qquad \textrm{in}\quad P(X_\tau, 100n^{5/2}/ H(X_\tau)).
\end{equation}
By  Lemma \ref{Vector Field closeness on genearalized cylinder_appendix}, we have

\begin{equation}
   \inf_{(\omega_{\alpha\beta,\tau-1})\in \mathrm{O}(\frac{(n-k+1)(n-k)}{2})}\sup_{x\in B_{10}(x_{\tau})} |K_{\alpha, \tau}(x) -\!\!\!\! \sum_{\beta=1}^{\frac{(n-k+1)(n-k)}{2}}\!\!\!\!\!\!\!\omega_{\alpha\beta, \tau-1} K_{\beta, \tau-1}(x)|\leq C2^{-\eta_0\rho(\tau)}.
\end{equation}
Without loss of generality, we may assume that
\begin{equation}
  \sup_{x\in B_{10}(x_{\tau})} |K_{\alpha, \tau}(x) -   K_{\alpha, \tau-1}(x)|\leq C2^{-\eta_0\rho(\tau)}.
\end{equation}
for otherwise, we can replace $K_{\alpha,\tau-m}$ by $\sum_{\beta=1}^{\frac{(n-k+1)(n-k)}{2}}\tilde{\omega}_{\alpha\beta, \tau-m} K_{\beta, \tau-m}(x)$ , where $\tilde{\omega}_{\tau-m} = \Pi_{j=1}^{m}\omega_{\tau-j}$ and $\omega_{\tau-j}=(\omega_{\alpha\beta, \tau-j})\in\mathrm{O}(\frac{(n-k+1)(n-k)}{2})$.

Moreover, since the tangent-flow at $-\infty$ is a cylinder $\mathbb{R}^{k}\times S^{n-k}(\!\sqrt{2(n-k)})$, the vector fields $K_{\alpha, \tau}(x)$ converge for $\tau\to -\infty$ to $K_{\alpha, \infty}(x)=S_{\alpha, \infty}J_{\alpha}S^{-1}_{\alpha, \infty}x$ where  $S_{\alpha, \infty}\in \mathrm{O}(n-k+1)$ by Lemma \ref{Lin Alg Cylinder A}. Hence, we infer that
\begin{equation}
\sup_{x\in B_{10}(x_{\tau})} |K_{\alpha, \tau}(x) - K_{\alpha, \infty}(x)| \leq  C\sum_{m=0}^{\infty}2^{-\eta_0 \rho(\tau-m)}.
\end{equation}
Recalling that $\rho(\tau)=|\tau|^\gamma$, this sum can be safely bounded by $Ce^{-\eta \rho(\tau)}$, where $\eta:=\eta_0/10$. Similarly, for any $q\in \bar{M}_{\tau}$ with $|y|\leq \rho(\tau)$ and $\tau\leq \tau_\ast$ the corresponding space-time point $X_{(q,\tau)}$ is
$(C/2^{\eta_0 \rho(\tau)}, n-k)$-symmetric with a normalized set of rotation vector fields $\mathcal{K}_{(q,\tau)}$, we infer that
\begin{equation}
\sup_{x\in B_{10}(x_{\tau})} |K_{\alpha, (q,\tau)}(x) - K_{\alpha,\infty}(x)|\leq C\rho(\tau) 2^{-\eta_0\rho(\tau) }\leq  Ce^{-\eta \rho(\tau)}.
\end{equation}
Hence, $X_{(q,\tau)}$ is $(C/2^{\eta \rho(\tau)}, n-k)$-symmetric with respect to $\mathcal{K}_\infty=\{K_{\alpha, \infty}: 1\leq \alpha\leq \frac{(n-k)(n-k+1)}{2}\}$. Notice that
\begin{align}
    \langle K_{\alpha,\infty}, \nu\rangle\sqrt{1+ \frac{|\nabla_{S^{n-k}}u|^2}{(u+\sqrt{2(n-k)})^2} + |\nabla_{\mathbb{R}^k}u|^2} &= -\langle  K_{\alpha,\infty}\omega, \nabla_{S^{n-k}}u\rangle,
\end{align}
we conclude
\begin{equation}
|\nabla_{S^{n-k}}u| \leq C\sum_{\alpha=1}^{\frac{(n-k)(n-k+1)}{2}}H |\langle K_{\alpha,\infty}, \nu \rangle |\leq Ce^{-\eta \rho(\tau)}.
\end{equation}
Slightly decreasing $\eta$, this proves the proposition.
\end{proof}

\bigskip

\subsection{Evolution expansion}
As before, we assume that the neutral mode is dominant and work with the graphical radius $\rho(\tau)=|\tau|^\gamma$. In particular, we define $\hat{u}$ using this $\rho$.
The goal of this subsection is to prove:
\begin{proposition}[evolution expansion]\label{evolution equation u2} The function $\hat{u}$ evolves by
\begin{align}
   \partial_{\tau}\hat{u}=\mathcal{L}\hat{u}-\frac{1}{2\sqrt{2(n-k)}}\hat{u}^2 -\tfrac{|\nabla_{S^{n-k}}\hat{u}|^2}{2^{3/2}(n-k)^{3/2}}-\tfrac{\hat{u}\triangle_{S^{n-k}}\hat{u}}{\sqrt{2}(n-k)^{3/2}} +\hat{E},
\end{align}
where the error term satisfies the weighted $\mathcal{H}$-norm estimate
\begin{equation}
\left\langle |\hat{E}|, 1+|y|^2\right\rangle_{\mathcal{H}}\leq C\rho^{-1}\|\hat{u} \|_{\mathcal{H}}^2+ e^{-\eta\rho}.
\end{equation}
\end{proposition}
\begin{proof}
We adapt the similar argument in \cite[Prop 2.8]{DH_hearing_shape}. We denote $\{\partial_{\alpha}\}_{\alpha=1,\ldots,k}$ for derivatives along the $\mathbb{R}^k$-factor, $\{\nabla_i\}_{i=k+1,\ldots,n}$ for (covariant) derivatives along the $S^{n-k}$-factor and $\nabla$ for the gradient over the cylinder $\mathbb{R}^{k}\times S^{n-k}$. Then, by \cite[Prop A.1]{DH_hearing_shape} the profile function $u$ evolves by
\begin{align}\label{v_evolution}
    \partial_{\tau}u=& \frac{A_{\alpha\beta}\partial_{\alpha}\partial_{\beta}u+B_{ij}\nabla_{i}\nabla_{j}u-2\partial_{\alpha} u  \nabla_i u \nabla_i\partial_{\alpha} u -(\sqrt{2(n-k)}+u)^{-1}|\nabla_{S^{n-k}} u|^2}{(1+|\partial u|^2)(\sqrt{2(n-k)}+u)^2+|\nabla_{S^{n-k}} u|^2}\\
    & -\frac{n-k}{\sqrt{2(n-k)}+u}+\frac{1}{2}\left(\sqrt{2(n-k)}+u- z_\alpha \partial_\alpha u\right),\nonumber
\end{align}
where
\begin{equation}\label{A}
    A_{\alpha\beta}=[(1+|\partial u|^2)(\sqrt{2(n-k)}+u)^2+|\nabla_{S^{n-k}} u|^2]\delta_{\alpha\beta}-(\sqrt{2(n-k)}+u)^{2}\partial_{\alpha}u \partial_{\beta}u,
\end{equation}
and
\begin{equation}\label{B}
    B_{ij}=(1+|\partial u|^2+(\sqrt{2(n-k)}+u)^{-2}|\nabla_{S^{n-k}} u|^2)\delta_{ij}-(\sqrt{2(n-k)}+u)^{-2}\nabla_{i}u\nabla_{j}u.
\end{equation}
This and $|\dot\rho|\leq \frac{1}{5}|\rho|$ imply
\begin{align}\label{evolution,sphere,E}
   \partial_\tau \hat{u}=\mathcal{L}  \hat{u}-\tfrac{1}{2\sqrt{2(n-k)}}  \hat{u}^2-\tfrac{1}{2^{3/2}(n-k)^{3/2}}|\nabla_{S^{n-k}}\hat{u}|^2-\tfrac{1}{ \sqrt{2}(n-k)^{3/2}} \hat{u}\triangle_{S^{n-k}}\hat{u}+\hat{E},
\end{align}
where
\begin{align}\label{est_eq}
|\hat{E}|\leq &C\chi(|u|+|\nabla u|)^2(|u|+|\nabla u|+|\nabla^2u|)\nonumber\\
&+ C |\chi'|\rho^{-1} \big(|\nabla u|+|u||y|\big)+ C |\chi''|\rho^{-2}|u|\\
&+ C \chi(1-\chi)\big(|u|^2+|\nabla u|^2+|\nabla^2u|^2\big).\nonumber
\end{align}
To bound this in the $(1+|y|^2)$-weighted $\mathcal{H}$-norm, we first observe that $\hat{E}$ is supported in the ball $\{|y|\leq \rho\}$, so in particular by the definition of graphical radius we have the inequality $|u|+|\nabla u| + |\nabla^2 u| \leq \rho^{-2}$
at our disposal.
Using this, for $|y|\leq \rho^{1/2}$ we can estimate
\begin{equation}
|\hat{E}|(1+|y|^2)\leq  C\rho^{-1} \left(|u|^2 + |\nabla u|^2\right).
\end{equation}
Together with \cite[Prop 4.1]{DH_blowdown} this implies
\begin{equation}
\int_{|y|\leq \rho^{1/2}}|\hat{E}|(1+|y|^2)\, e^{-\frac{|y|^2}{4}} \, dy\leq C\rho^{-1}\|\hat{u}\|_{\mathcal{H}}^2\, .
\end{equation}

Finally, for $|y|\geq \rho^{1/2}$ the coarse bound $|\hat{E}|(1+|y|^2) \leq C \rho^2$ yields the tail estimate
\begin{align}
\int_{\rho/2\leq |y|\leq \rho}|\hat{E}|(1+|y|^2) \, e^{-\frac{|y|^2}{4}}\, dy  \leq  e^{-\rho/5} \, .
\end{align}
This finishes the proof of the proposition.
\end{proof}

\section{Proof of the spectral quantization theorem}\label{sec_quant_thm}
In this section, we prove Theorem \ref{spectral theorem} (spectral quantization theorem). As before, we work with the graphical radius $\rho(\tau)=|\tau|^\gamma$, and in particular define $\hat{u}$ and $U_\pm,U_0$ with respect to this $\rho$. By
Proposition \ref{mz.ode.fine.bubble-sheet} (quantitative Merle-Zaag alternative) for $\tau\to -\infty$ either the neutral mode is dominant, i.e.
$U_-+U_+\leq C\rho^{-1}U_0$,
or the unstable mode is dominant, i.e.
$U_-+U_0\leq C\rho^{-1}U_+$. If the unstable mode is dominant, then by \eqref{U_PNM_system} we get
$\dot{U}_{+} \geq \tfrac{1}{2}U_{+}$  for all $\tau\ll 0$.
Integrating this differential inequality yields $U_+(\tau) \leq C e^{\tau/2}$. Thus, recalling that $U_+ = \| P_+ \hat{u}\|_{\mathcal{H}}^2$ and using again the assumption that the unstable mode is dominant we infer that $\| \hat{u} \|_{\mathcal{H}} \leq Ce^{\tau/4}$. Together with standard interpolation inequalities this implies that for any $R<\infty$ and $k<\infty$  we have
$\| u(\cdot,\tau) \|_{C^{k}(B_R)} \leq Ce^{\tau/5}$. Hence, if the the unstable mode is dominant then the conclusion of Theorem \ref{spectral theorem} holds with $Q=0$. We can thus assume from now on in this whole section that the neutral mode is dominant, i.e.
\begin{equation}\label{neu_dom}
U_-+U_+\leq C\rho^{-1}U_0.
\end{equation}
In this case, $\mathrm{rk}(Q)=0$ in  Theorem \ref{spectral theorem} cannot happen, for otherwise  $\hat{u}(\tau)/\|u(\tau)\|_{\mathcal{H}}$ will converge to $0$, which contradicts that neutral mode is dominant.
Assuming also that the flow is not a round shrinking cylinder as before, our goal is to show that the conclusion of Theorem \ref{spectral theorem} holds for some $k\times k$ constant symmetric matrix $Q$ with
\begin{equation}\label{rankgeq 1}
    \mathrm{rk}(Q)\geq 1.
\end{equation}
\subsection{Derivation of the spectral ODE system}
In this subsection, we consider the following expansion in $\mathcal{H}$-norm sense,
\begin{equation}\label{expansion_123}
 \hat{u} = \sum_{m=1}^{k}\alpha_{mm}(y^2_{m}-2)+\sum_{1\leq i<j\leq k} 2\alpha_{ij}y_{i}y_{j}+\hat{w}\, .
\end{equation}
Let
\begin{equation}\label{psi_defi}
    \psi_{mm}=y^2_{m}-2\quad \psi_{ij}=2y_{i}y_{j},
\end{equation}
and $A=(\alpha_{ij})$ be the symmetric $k\times k$ spectral coefficients matrix with
\begin{equation}\label{def_exp_coeffs}
 \alpha_{ij} = \|\psi_{ij}\|_{\mathcal{H}}^{-2} \langle  \psi_{ij},\hat{u}\rangle_{\mathcal{H}}.
\end{equation}
The first goal is to derive the following ODE system for the spectral coefficients:
\begin{proposition}[spectral ODE system]\label{odes-1}
The spectral coefficients matrix $A(\tau)$ satisfies the following  ODE system:
\begin{equation}\label{odes0}
 \dot A=- \beta_{n, k}^{-1} A^2+E,
\end{equation}
where
\begin{equation}
    \beta_{n, k}=\frac{\sqrt{2(n-k)}}{4}
\end{equation}
and
\begin{equation}
    |E|\leq C|\tau|^{-\frac{\gamma}{2}}|A|^2 + Ce^{-\eta|\tau|^{\gamma}/10}
\end{equation}
holds with $\gamma>0$  from \eqref{gamma radius} and $\eta$ from \eqref{eta small}.
\end{proposition}
Then, we have the following remainder estimates.
\begin{lemma}[{remainder estimate, cf. \cite[Lem 3.2]{DH_hearing_shape}}]\label{rotationdecay}
The remainder $\hat{w}$ satisfies the weighted $\mathcal{H}$-norm estimate
\begin{equation}\label{remainder1}
|\left\langle\hat{w}^2, 1+|y|^2 \right\rangle_{\mathcal{H}}| \leq  C\rho^{-1}|A|^2+Ce^{-\eta\rho/3}.
\end{equation}
\end{lemma}
\begin{proof}
The proof follows from Proposition \ref{evolution equation u2} and similar argument as in \cite[Lem 3.2]{DH_hearing_shape}. For the readers' convenience, we include the details. Using Proposition \ref{utheta} and $U_-+U_+\leq  C\rho^{-1}U_0$, it follows that
\begin{equation}\label{remainder1inproof}
\|\hat{w}\|_{\mathcal{H}}^2 \leq  C\rho^{-1}|A|^2 + Ce^{-\eta\rho}.
\end{equation}
Then, we project the equation from Proposition \ref{evolution equation u2} (evolution equation) to the orthogonal complement of neutral mode eigenfunctions in \eqref{basis_hneutral_intro}, and argue as above to obtain
\begin{equation}\label{eq_w}
\partial_\tau \hat{w}= \mathcal{L} \hat{w} + g,
\end{equation}
where by \eqref{evolution,sphere,E} and \eqref{eta small}
\begin{equation}\label{est_g}
\| g\|_{\mathcal{H}}\leq\frac{1}{2\sqrt{2(n-k)}}\|\hat{u}^2\|_{\mathcal{H}}+ \|\hat{E}\|_{\mathcal{H}}+Ce^{-\eta\rho}.
\end{equation}
Moreover by \eqref{est_eq}, the definition of cut-off function $\chi$ and admissible graphical radius $\rho$,  for $|y|\leq \rho^{1/2}$ we have
\begin{equation}
    |\hat{E}|\leq \rho^{-4}(|u|+|\nabla u|),
\end{equation}
and for $\rho^{1/2}\leq |y|\leq 2\rho$, we have $|\hat{E}|\leq C$. These estimates, inverse Poincare inequality  \cite[Prop 4.1]{DH_blowdown} and Gaussian tail estimates imply that
\begin{equation}
    \|\hat{E}\|_{\mathcal{H}}\leq C\rho^{-4} \|\hat{u}\|_{\mathcal{H}}+Ce^{-\eta\rho}.
\end{equation}
Using this, \eqref{est_g} and again the definition of graphical radius, we have
\begin{equation}\label{g_estimates}
   \| g\|_{\mathcal{H}}\leq  \frac{1}{2\sqrt{2(n-k)}}\rho^{-2}\|\hat{u}\|_{\mathcal{H}}+C\rho^{-4} \|\hat{u}\|_{\mathcal{H}}+Ce^{-\eta\rho}.
\end{equation}
Then, we use \eqref{remainder1inproof}, \eqref{g_estimates} and repeat the gradient estimates in \cite[Lem 3.2]{DH_hearing_shape}.
Given $\tau_\ast\ll 0$, using \eqref{eq_w} and integration by parts we compute
\begin{align}
\frac{d}{d\tau}\int {e^{\tau_\ast-\tau}} \hat{w}^2e^{-q^2/4} &= \int  {e^{\tau_\ast-\tau}}(2\hat{w}g-2|\nabla \hat{w}|^2 -\hat{w}^2)\, e^{-q^2/4}\nonumber\\
& \leq \int  {e^{\tau_\ast-\tau}}(g^2-2|\nabla \hat{w}|^2)\, e^{-q^2/4},
\end{align}
and
\begin{align}
 \frac{d}{d\tau}\int (\tau-\tau_\ast)|\nabla \hat{w}|^2 e^{-q^2/4}
 &=\int \left(|\nabla \hat{w}|^2  -2(\tau-\tau_\ast) (\mathcal{L}  w)(\mathcal{L}\hat{w}+g)\right)\, e^{-q^2/4}\nonumber\\
 &\leq\int \left(|\nabla \hat{w}|^2  +\tfrac{1}{2}(\tau-\tau_\ast) g^2)\right)\, e^{-q^2/4}\, .
\end{align}
For $\tau\in [\tau_\ast-1,\tau_\ast]$ this yields
\begin{align}
 \frac{d}{d\tau}\int \left((\tau-\tau_\ast)|\nabla \hat{w}|^2+\tfrac{e^{\tau_\ast-\tau}}{2} \hat{w}^2\right)\, e^{-q^2/4}\leq \int g^2\, e^{-q^2/4}\, .
 \end{align}
Hence, together with \eqref{remainder1inproof} and \eqref{est_g} we infer that
\begin{equation}
\|\nabla \hat{w}(\tau)\|_{\mathcal{H}}^2 = o\left(\max_{\tau' \in [\tau,\tau+1]}|A(\tau')|^2\right) + Ce^{-\eta\rho(\tau)/2}.
\end{equation}
Finally, by the Merle-Zaag ODE inequalities  \eqref{U_PNM_system} for $\tau\ll 0$ we have
\begin{equation}
\max_{\tau' \in [\tau,\tau+1]}|A(\tau')|^2\leq 2 |A(\tau)|^2\, .
\end{equation}

\begin{equation}
\|\nabla \hat{w}(\tau)\|_{\mathcal{H}}^2 \leq C\rho^{-1}|A|^2 + Ce^{-\eta\rho(\tau)/2}.
\end{equation}
Together with Ecker's weighted Sobolev inequality \cite[page 109]{Ecker_logsob}, this implies the assertion.
\end{proof}

We can now prove the main result of this subsection.

\begin{proof}[{Proof of Proposition \ref{odes-1}}]
Projecting the evolution equation of truncated profile function $\hat{u}$ in Proposition \ref{evolution equation u2} (evolution expansion) to eigenfunctions
\begin{equation}
    \psi_{ij} =
    \begin{cases}
      & y_i^2 - 2 \quad \text{if $i = j$},\\
      & 2y_iy_j \quad \text{if $i \neq j$},
    \end{cases}
\end{equation}
we have
\begin{align}\label{ak}
\dot{\alpha}_{\ell m}&=  \|\psi_{\ell m}\|_{\mathcal{H}}^{-2}\left\langle  \mathcal{L}\hat{u}-\frac{1}{2\sqrt{2(n-k)}} \hat{u}^2,\psi_{\ell m}\right\rangle_{\mathcal{H}}\\\nonumber
&+\|\psi_{\ell m}\|_{\mathcal{H}}^{-2}\left\langle  -\tfrac{1}{2^{3/2}(n-k)^{3/2}}|\nabla_{S^{n-k}}\hat{u}|^2-\tfrac{1}{\sqrt{2}(n-k)^{3/2}} \hat{u}\triangle_{S^{n-k}}\hat{u}+\hat{E},\psi_{\ell m}\right\rangle_{\mathcal{H}}.
\end{align}
Next, using $\mathcal{L}\psi_{\ell m} =0$ and integration by parts we see that
\begin{equation}
\langle \mathcal{L}\hat{u},\psi_{\ell m} \rangle_{\mathcal{H}} = 0.
\end{equation}
Moreover, using Proposition \ref{utheta} (almost circular symmetry), writing the function $\hat{u}\triangle_{S^{n-k}}\hat{u}=\text{div}_{S^{n-k}}(\hat{u}\nabla_{S^{n-k}}\hat{u})- |\nabla_{S^{n-k}}\hat{u}|^2$, and using integration by parts we can estimate
\begin{equation}
|\langle |\nabla_{S^{n-k}}\hat{u}|^2,\psi_{\ell m} \rangle_{\mathcal{H}}| + |\langle \hat{u}\triangle_{S^{n-k}}\hat{u},\psi_{\ell m} \rangle_{\mathcal{H}}|  \leq  C e^{-\eta\rho}.
\end{equation}
Furthermore, remembering the expansion \eqref{expansion_123} we compute
\begin{equation}
 \langle \hat{u}^2, \psi_{\ell m} \rangle
= \sum_{i,j = 1}^k \sum_{p,q = 1}^k \alpha_{ij}\alpha_{pq}\langle \psi_{ij}\psi_{pq}, \psi_{\ell m} \rangle+ 2\langle  \psi_{\ell m} \sum_{i,j = 1}^k \alpha_{ij}\psi_{ij}, \hat{w} \rangle + \langle \hat{w}^2, \psi_{\ell m} \rangle.
\end{equation}
By Lemma \ref{rotationdecay} (remainder estimates) we have
\begin{equation}
    2\left|\langle  \psi_{\ell m} \sum_{i,j = 1}^k \alpha_{ij}\psi_{ij}, \hat{w} \rangle\right| + \left|\langle \hat{w}^2, \psi_{\ell m} \rangle\right|\leq   C\rho^{-1/2}|A|^2+Ce^{-\eta\rho/6}.
\end{equation}
Combining the above facts we infer that
\begin{equation}\label{alpha_rough_equation}
   \Dot{\alpha}_{\ell m} = -\frac{1}{2\sqrt{2(n-k)}}\norm{\psi_{\ell m}}_\mathcal{H}^{-2}\sum_{i,j = 1}^N \sum_{p,q = 1}^N \alpha_{ij}\alpha_{pq}\langle \psi_{ij}\psi_{pq}, \psi_{\ell m} \rangle + E_{\ell m},
\end{equation}
where $|E_{\ell m}|\leq C\rho^{-1/2}|A|^2 + Ce^{-\eta\rho(\tau)/6}$.
 An elementary computation on the inner product of these neutral mode eigenfunctions shows that for the indices $i, j,p,q, \ell, m = 1, ..., k$, the only nonvanishing terms of $\langle \psi_{ij}\psi_{pq}, \psi_{\ell m} \rangle $ in \eqref{alpha_rough_equation} have the following relations:
\begin{equation}
   \langle \psi_{ii}\psi_{ii}, \psi_{ii} \rangle = 8\norm{\psi_{ii}}_\mathcal{H}^2,
\end{equation}
\begin{equation}
    \langle \psi_{ij}\psi_{ij}, \psi_{ii} \rangle = 8\norm{\psi_{ii}}_\mathcal{H}^2\quad (i\neq j),
\end{equation}
\begin{equation}
    \langle \psi_{im}\psi_{i\ell}, \psi_{\ell m} \rangle = 8\norm{\psi_{ii}}_\mathcal{H}^2\quad (i \neq m \neq \ell \neq i),
\end{equation}
\begin{equation}
    \norm{\psi_{ii}}_\mathcal{H}^2 = \frac{1}{2}\norm{\psi_{pq}}_\mathcal{H}^2\quad (p\neq q).
\end{equation}
Combining the above facts and $\rho(\tau)=|\tau|^{\gamma}$, we infer
\begin{equation}\label{aode}
    \dot{\alpha}_{\ell m} = -\frac{4}{\sqrt{2(n-k)}}\sum_{i = 1}^k \alpha_{im}\alpha_{i\ell}  +O( |\tau|^{-\frac{\gamma}{2}}\abs{\vec{\alpha}}^2 + e^{-\eta|\tau|^{\gamma}/10}),
\end{equation}
which is equivalent to
\begin{equation}\label{A_ode}
    \dot A=- \beta_{n, k}^{-1} A^2+ E,
\end{equation}
with  $\beta_{n, k}=\frac{\sqrt{2(n-k)}}{4}$ and $|E|\leq C|\tau|^{-\frac{\gamma}{2}}|A|^2 + Ce^{-\eta|\tau|^{\gamma}/10}$. This completes the proof of the proposition.
\end{proof}

\bigskip

\subsection{Quantized asymptotics of the spectral ODE system}
We can now conclude the proof of Theorem \ref{spectral theorem} (spectral quantization theorem), which we restate here in a technically sharper way.
\begin{theorem}[spectral quantization theorem]\label{spectral theorem_restated}
For any ancient  mean curvature flow in $\mathbb{R}^{n+1}$ whose tangent flow at $-\infty$ is given by $\mathbb{R}^{k}\times S^{n-k}(\sqrt{2(n-k)|t|})$, the cylindrical function $u$, truncated at the graphical radius $\rho(\tau)=|\tau|^\gamma$, where $\gamma$ is the exponent from Proposition \ref{radius_lower_bound} (graphical radius), satisfies
 \begin{equation}\label{main_thm_ancient_H norm}
\lim_{\tau\to -\infty} \Big\|\,
|\tau| \hat{u}(y, \omega,\tau)- y^\top Qy +2\mathrm{tr}(Q)\, \Big\|_{\mathcal{H}} = 0,
\end{equation}
and
 \begin{equation}\label{main_thm_ancient}
\lim_{\tau\to -\infty} \Big\|\,
|\tau| u(y, \omega,\tau)- y^\top Qy +2\mathrm{tr}(Q)\, \Big\|_{C^{p}(B_R)} = 0,
\end{equation}
for all $R<\infty$ and all integers $p\geq 0$, where $Q$ is a constant symmetric $k\times k$-matrix whose eigenvalues are quantized to be either 0 or $-\frac{\sqrt{2(n-k)}}{4}$.
 \end{theorem}
Before the proof of this theorem, we recall that we have assumed that neutral mode is dominant in the whole section since the theorem holds if unstable mode is dominant, which is equivalent to $\textrm{rk}(Q)=0$. Then, we need the following lemma to prove this theorem.
\begin{lemma}\label{uul}
we have the estimates
\begin{equation}\label{utau-1}
    \frac{c}{|\tau|}\leq |A|\leq \frac{C}{|\tau|}.
\end{equation}
\end{lemma}
\begin{proof}[proof of lemma]
Because the norm of symmetric matrices is determined by the norm of their eigenvalues, we only need to show  there are constants $0<c<C<\infty$ such that
\begin{equation}\label{cCinequality}
\frac{c}{|\tau|}\leq \sum_{i=1}^{k}|\lambda_{i}(\tau)|\leq \frac{C}{|\tau|}.
\end{equation}
Motivated by \cite{FilippasLiu}, we first compute the ODE system satisfied by these eigenvalues. By \cite[Sec 2, Thm 6.8]{kato} and continuously differentiable property of symmetric matrices $A(\tau)$, under suitable arrangement the eigenvalues $\lambda_{1}(\tau), \dots, \lambda_{k}(\tau)$ of $A(\tau)$ are continuously differentiable in time. Now let $P(\tau)$ be the projection to the eigenspaces spanned by the orthonormal basis $\{\psi^{(j)}_{i}(\tau): 1\leq i\leq h_{j}\}$ of eigenvalue $\lambda_{j}(\tau)$, which means
\begin{equation}\label{Aeigenvalues}
    A(\tau)\psi^{(j)}_{i}(\tau)=\lambda_{j}(\tau)\psi^{(j)}_{i}(\tau).
\end{equation}
By \cite[Thm 1 page 45]{rellich1969perturbation}, we have
\begin{equation}
    P(\tau)\dot{A}(\tau)\psi^{(j)}_{i}(\tau)=\dot{\lambda}_{j}(\tau)\psi^{(j)}_{i}(\tau).
\end{equation}
On the other hand, by \eqref{Aeigenvalues} we have
\begin{equation}\label{A2eigenvalues}
     P(\tau){A}^2(\tau)\psi^{(j)}_{i}(\tau)={\lambda}^{2}_{j}(\tau)\psi^{(j)}_{i}(\tau).
\end{equation}
The above two equations and Proposition \ref{odes-1} imply that
\begin{equation}\label{lambdai ode}
    \left|\frac{d}{d\tau}\lambda_{i}+\beta^{-1}_{n,k}\lambda_{i}^2\right|\leq \|P(\tau)E(\tau)\|_{\mathcal{H}}\leq \|E(\tau)\|_{\mathcal{H}}\leq C|\tau|^{-\gamma}\sum_{i=1}^{k}\lambda_{i}^2+e^{-\eta|\tau|^{\gamma}/10}.
\end{equation}
This implies
\begin{equation}\label{lambda__eq}
    \frac{d}{d\tau}\lambda_{i}=-\beta^{-1}_{n,k}\lambda_{i}^2+\tilde{E}_{i},
\end{equation}
where $\tilde{E}_{i}=o(\sum_{i=1}^{k}\lambda_{i}^2+e^{-\eta|\tau|^{\gamma}/20})$.
Then, we show that \eqref{cCinequality} holds.
We first notice that
\begin{equation}\label{tau10}
\limsup_{\tau\to- \infty}\tau^{10}\left(\sum_{i=1}^{k}\lambda^2_{i}\right)(\tau)=+\infty.
\end{equation}
Indeed, suppose towards a contradiction that $\sum_{i=1}^{k}\lambda^2_{i}(\tau)\leq C \tau^{-10}$ for all $\tau$ negative enough, by
\eqref{alpha_beta_equivalence}, \eqref{graphrho} and monotonicity of $\tau^{-10}$ we have $\rho(\tau)\geq c\tau^2$.  Then \eqref{U_PNM_system} and the fact that neutral mode is dominant imply that $|\log U_{0}|\leq C$ for $\tau$ negative enough. This contradicts  that $U_{0}$ converges to $0$.

Now, we define the following $C^{1}$ function
\begin{equation}
    \alpha(\sigma)=\tau^2 \left(\sum_{i=1}^{k}\lambda^2_{i}(\tau)\right)\quad \tau=-e^{\sigma}.
\end{equation}
Then, by \eqref{lambda__eq} we have
\begin{equation}
    \frac{d\alpha}{d\sigma}=2\alpha-2\beta^{-1}_{n,k}\left(\sum_{i=1}^{k}(\tau\lambda)^3_{i}\right)+2\tau^3\tilde{E}\sum_{i=1}^{k}\lambda_{i},
\end{equation}
where $\tilde{E}=o(\sum_{i=1}^{k}\lambda_{i}^2+e^{-\eta|\tau|^{\gamma}/20})$.

By \eqref{tau10} we can find a sequence of numbers $\sigma_{i}\to \infty$ large enough  such that
\begin{equation}\label{sequence_sigma}
    \alpha(\sigma_{i})\geq e^{-10\sigma_{i}}>0\quad i\geq 1.
\end{equation}

By $\sum_{i=1}^{k}|\lambda_{i}|\to 0$ and definition of $\tilde{E}$, we have
\begin{equation}\label{alpha differential inequality}
   \frac{d\alpha}{d\sigma}\geq 2\alpha-c'\alpha^{3/2}+\bar{E},
\end{equation}
where  $c'>0$ is a constant and the error term $|\bar{E}|=|\tau^3 e^{-\eta e^{\gamma\sigma}/20}o(1)|\leq c' e^{-30\sigma}$ for $\sigma\geq \sigma_{1}$ large enough.
Then, we define
\begin{equation}
    J=\{\sigma\geq \sigma_{1}: \alpha^{3/2}(\sigma)\geq e^{-30\sigma}\}.
\end{equation}
By \eqref{sequence_sigma}, $\sigma_{1}\in J\not=\emptyset$ and is closed. Now, if $\hat{\sigma}\in J$, we have $\alpha(\hat{\sigma})>0$ and
\begin{equation}
    \frac{d\alpha}{d\sigma}(\hat{\sigma})\geq 2\alpha(1-c'\alpha^{1/2})(\hat{\sigma}).
\end{equation}
In the first case if  $c'^2\alpha(\hat{\sigma})< 1$,      $\frac{d\alpha}{d\sigma}(\hat{\sigma})> 0$, then there is a $\delta>0$ such that $[\hat{\sigma}, \hat{\sigma}+\delta)\subset J$. In the second case if  $\alpha(\hat{\sigma})>\frac{1}{2c'^2}>e^{-10\hat{\sigma}}$, By continuity of $\alpha$, we can decrease $\delta>0$ such that $[\hat{\sigma}, \hat{\sigma}+\delta)\subset J$. Combining the two cases, we conclude

\begin{equation}\label{J=half interval}
     J=[\sigma_{1}, +\infty).
\end{equation}
In particular, we obtain that \begin{equation}\label{sum lower bound}
e^{-\eta |\tau|^{\gamma}/20}\leq \left(\sum_{i=1}^{k}\lambda^2_{i}\right).
\end{equation}
This and  \eqref{lambda__eq} imply that we have
\begin{equation}\label{lambda__eqo}
    \frac{d}{d\tau}\lambda_{i}=-\beta^{-1}_{n,k}\lambda_{i}^2+o(\sum_{i=1}^{k}\lambda_{i}^2).
\end{equation}
Then, by \eqref{sum lower bound} we have $\sum_{i=1}^{k}|\lambda_{i}|>0$.  Summing up the equations in \eqref{lambda__eqo} and using elementary estimates, we have
\begin{equation}\label{differential_sum<0}
    \frac{d}{d\tau}\sum^{k}_{i=1} \lambda_{i}\leq -(2\beta_{n,k})^{-1}\sum^{k}_{i=1} \lambda^2_{i}\leq  -(2k\beta_{n,k})^{-1}\left(\sum^{k}_{i=1} \lambda_{i}\right)^2<0.
\end{equation}
Then, we see that $\sum^{k}_{i=1} \lambda_{i}(\tau)$ is decreasing. This and $\lim_{\tau\to \infty}\sum^{k}_{i=1} \lambda_{i}(\tau)\to 0$ imply that
\begin{equation}\label{sum<0}
\sum^{k}_{i=1} \lambda_{i}(\tau)<0
\end{equation}
for all sufficiently negative times. Integrating the above differential inequality \eqref{sum<0} we have
\begin{equation}\label{|sum|upperbound}
    \left|\sum^{k}_{i=1} \lambda_{i}(\tau)\right|\leq \frac{C}{|\tau|}.
\end{equation}
Now, for any $\varepsilon > 0$ small enough, we can find  $\tau_0 < 0$ sufficiently negative such that
\begin{align}\label{quantitative error}
    |\tilde{E}_i| < \frac{\varepsilon^2}{4\beta_{n,k}k}\sum_{i=1}^k\lambda_i^2
\end{align}
for all $\tau \leq \tau_0$.
Let
\begin{equation}
   Q = \max\limits_{1\leq i\leq k}  \lambda_i + \varepsilon \min\limits_{1\leq i\leq k} \lambda_i
\end{equation}
and
\begin{equation}
    J' = \{\tau \leq \tau_0: Q(\tau)\geq 0\}.
\end{equation}
Then, we have the following claim.
\begin{claim}[almost non-upward  quadratic bending estimates]
 \begin{equation}
     J'=\emptyset.
 \end{equation}
\end{claim}
\begin{proof}[Proof of claim]
Suppose towards a contradiction that $J'\not=\emptyset$.
 Clearly,  $J'$ is closed by the continuity of $\max\limits_{1\leq i\leq k}  \lambda_i$ and $\min\limits_{1\leq i\leq k}  \lambda_i$.   For any element $\bar{\tau} \in J'$, we may assume without loss of generality that $\max\limits_{1\leq i\leq k}\lambda_i(\bar{\tau}) =\lambda_1(\bar{\tau})$ and $\min\limits_{1\leq i\leq k} \lambda_i(\bar{\tau})  = \lambda_k(\bar{\tau}) $.
Since $Q(\bar{\tau})\geq 0$, we have
\begin{align}
    \sum_{i=1}^{k}\lambda_i^2(\bar{\tau})\leq \frac{k}{\varepsilon^2}\lambda_1^2(\bar{\tau}).
\end{align}
Then we recall the definition of left upper derivative $\bar{D}^-$:
\begin{align}
    \bar{D}^-f(x) = \limsup\limits_{h\rightarrow 0^+}\frac{f(x)-f(x-h)}{h}.
\end{align}
By definition, (\ref{lambda__eq}) and (\ref{quantitative error}), it has the following property:
\begin{align} \label{derivative_Q_0}
    \bar{D}^{-}\min\limits_{1\leq i\leq k}\lambda_i (\bar{\tau})
    &= \limsup\limits_{h\rightarrow 0^+}\frac{ \lambda_k (\bar{\tau})-\min\limits_{1\leq i\leq k}\lambda_i (\bar{\tau}-h)}{h} \\ \notag
    &= \max\limits_{1\leq i\leq k}\limsup\limits_{h\rightarrow 0^+}\frac{ \lambda_k (\bar{\tau})-\lambda_i (\bar{\tau}-h)}{h}\\ \notag
    &\leq   \max\limits_{1\leq i\leq k}\limsup\limits_{h\rightarrow 0^+}\frac{ \lambda_i (\bar{\tau})-\lambda_i (\bar{\tau}-h)}{h} \\\nonumber
    &\leq  \max\limits_{1\leq i\leq k} \lambda_i'(\bar{\tau}) \\\nonumber
    &\leq \frac{\varepsilon^2}{4\beta_{n,k} k}\sum_{i=1}^k\lambda_i^2(\bar{\tau}).\nonumber
\end{align}
 From (\ref{lambda__eq}), (\ref{quantitative error})  and the definition of left upper derivative, we have
\begin{align} \label{derivative_Q_1}
    \bar{D}^-\max\limits_{1\leq i\leq k}\lambda_i (\bar{\tau}) &= \limsup\limits_{h\rightarrow 0^+}\frac{ \lambda_1 (\bar{\tau})-\max\limits_{1\leq i\leq k}\lambda_i (\bar{\tau}-h)}{h}\\\nonumber
    &\leq
    \lambda_1'(\bar{\tau})\\\nonumber
    &\leq -\beta^{-1}_{n,k}\lambda_1^2(\bar{\tau})+\frac{\varepsilon^2}{4\beta_{n,k} k}\sum_{i=1}^k\lambda_i^2(\bar{\tau}).
\end{align}

Putting  \eqref{derivative_Q_0} and \eqref{derivative_Q_1} together we get
\begin{equation} \label{derivative_Q_3}
    \bar{D}^- Q(\bar{\tau})    \leq -\beta^{-1}_{n,k}\lambda_1^2 + \frac{2\varepsilon^2}{4\beta_{n,k} k}\sum_{i=1}^k\lambda_i^2 \leq - \frac{\lambda_1^2}{2\beta_{n,k}} < 0.
\end{equation}
Hence, there is a small $\delta > 0$ such that $(\bar{\tau}-\delta,\bar{\tau}]\subset J'$.
Consequently, $J' = (-\infty,\tau_0]$. This and (\ref{derivative_Q_1}) imply that   $\max\limits_{1\leq i\leq k}\lambda_i(\tau)$ on $J'$ satisfies
\begin{align}
     \bar{D}^{-}\max\limits_{1\leq i\leq k}\lambda_i(\tau) \leq -\beta^{-1}_{n,k}(\max\limits_{1\leq i\leq k}\lambda_i)^2(\tau) + \frac{\varepsilon^2}{4\beta_{n,k} k}\sum_{i=1}^k\lambda_i^2(\tau) < 0,
\end{align}
which implies
\begin{align}
    \sup\limits_{(-\infty,\tau_0]}\max\limits_{1\leq i\leq k}\lambda_i \geq \max\limits_{1\leq i\leq k}\lambda_i(\tau_0) > 0.
\end{align}
However, this contradicts  the fact that $  \sum_{i=1}^{k}\lambda_i^2(\tau)\rightarrow 0$ as $\tau\rightarrow -\infty$. Hence $J'$ is an empty set.
\end{proof}
Then we have $\limsup\limits_{\tau\rightarrow-\infty} \frac{\lambda_i}{\sum_{i=1}^{k} |\lambda_i| } \leq 0$ for each $1\leq i\leq k$ and the strict inequality holds for some $1\leq i\leq k$ since the neutral mode is dominant. Hence there is a large constant $C$ such that for all sufficiently negative time $\tau$ the following estimates holds.
\begin{align}\label{|sum|equivalentsum||}
    \frac{1}{2}|\sum_{i=1}^{k} \lambda_i| \leq \sum_{i=1}^{k} |\lambda_i| \leq C|\sum_{i=1}^{k} \lambda_i|.
\end{align}
Then, this together with (\ref{|sum|upperbound}) gives the right side of (\ref{cCinequality}). Namely
 \begin{align}
     \sum_{i=1}^{k}|\lambda_i| \leq \frac{C}{|\tau|}.
 \end{align}
Finally, we plug (\ref{|sum|equivalentsum||}) into (\ref{lambda__eqo}) and get
\begin{align}
    \frac{d}{d\tau}\sum_{i=1}^{k}\lambda_i\geq -C(n, k)(\sum_{i=1}^{k}\lambda_i)^2,
\end{align}
where $0<C(n, k)<\infty$ is a constant depending on $n, k$. Integrating the differential inequality and using (\ref{|sum|equivalentsum||}) imply the left side of (\ref{cCinequality}). Namely
 \begin{align}
      \frac{c}{|\tau|}\leq  \frac{1}{2}|\sum_{i=1}^{k}\lambda_i|  \leq \sum_{i=1}^{k}|\lambda_i|.
 \end{align}
This completes the proof of estimates \eqref{cCinequality} and the lemma.
\end{proof}
Now, we prove Theorem \ref{spectral theorem_restated} (spectral quantization theorem).
\begin{proof}[Proof of the Theorem \ref{spectral theorem_restated} and Theorem \ref{spectral theorem}]
By  Lemma \ref{uul} and Proposition \ref{odes-1}  we have
\begin{equation}\label{A_ode2}
    \dot A=- \beta_{n, k}^{-1} A^2+ O(|\tau|^{-2-\frac{\gamma}{2}}),
\end{equation}
where  $\beta_{n, k}=\frac{\sqrt{2(n-k)}}{4}$  and $\gamma>0$ is from \eqref{gamma radius}.

Then by  \eqref{utau-1} and \eqref{lambdai ode}  in proof of Lemma \ref{uul} and the norm representation of symmetric matrix, we have
\begin{equation}\label{lambdai_ode}
    \frac{d}{d\tau}\lambda_{i}+\beta^{-1}_{n,k}\lambda_{i}^2=O(|\tau|^{-2-\frac{\gamma}{2}})\quad 1\leq i\leq k.
\end{equation}
Then, since the power in the error of above ODEs is $-2-\frac{\gamma}{2}<-2$ ,  we can integrate the ODEs and find a constant $0<\mu\ll \gamma/2$ small enough as in \cite[Lem 5.1]{FilippasLiu} such that either
\begin{equation}
    \lambda_{i}=\beta_{n, k}\tau^{-1}+O(\tau^{-1-\mu}),
\end{equation}
or
\begin{equation}
    \lambda_{i}=O(\tau^{-1-\mu}).
\end{equation}
By  diagonalization method and above dichotomy of asymptotics of eigenvalues $\lambda_{1}, \dots, \lambda_{k}$, we obtain that solutions of \eqref{odes0} satisfy the following quantized sharp asymptotics:
\begin{equation}\label{ARtau9}
    A(\tau)=\frac{\beta_{n,k}}{\tau}R(\tau)\text{diag}\{I_{k-r}, O_{r}\}R(\tau)^{-1}+o(\tau^{-1-\mu}),
\end{equation}
where $R(\tau)\in \mathrm{SO}(k)$ and $0\leq r\leq k$.\\\\
Then, we  need to show that $R(\tau)$ is a time independent constant matrix by the similar argument  in the proof of \cite[Prop 5.1]{FilippasLiu}. If $r=0$ or $r=k$, $R$ is already constant matrix and the conclusion is obviously true, so without loss of generality we may assume that $1\leq r\leq k-1$. Let $\gamma_{\ell}=\{i_{1},\dots, i_{\ell}\}\subset \{1,\dots, k\}$. We let
$A_{\gamma_{\ell}}$ be the determinant of submatrix of $A$ by selecting rows and columns according to $\gamma_{\ell}$. By the above quantized asymptotics of $\lambda_{i}$, we have
\begin{equation}\label{Agammaell estimate}
    |A_{\gamma_{\ell}}|\leq \frac{C}{|\tau|^\ell}\quad trA=\frac{(k-r)\beta_{n,k}}{\tau}+O(\tau^{-1-\mu})\quad detA=\delta^{0}_{r}\beta^{k}_{n,k}\tau^{-k}+O(\tau^{-k-\mu}).
\end{equation}
Then, we write \eqref{A_ode2} in the following equivalent form
\begin{equation}\label{aode2}
    \dot{\alpha}_{ij} = -\beta^{-1}_{n, k}\sum_{m = 1}^k \alpha_{mi}\alpha_{mj}  +O( |\tau|^{-2-\frac{\gamma}{2}})\quad 1\leq i, j \leq k.
\end{equation}
Using \eqref{aode2} and $0<\mu\ll \frac{\gamma}{2}$,  direct computation shows that $A_{\gamma_{\ell}}$ satisfies the ODE system
\begin{equation}
    \frac{d}{d\tau}A_{\gamma_{\ell}}=-\beta_{n,k}^{-1}\left[A_{\gamma_{\ell}}trA+\sum_{q\notin \gamma_{\ell}}A_{\gamma_{\ell}\cup\{q\}}+O(\tau^{-\ell-1-\mu})\right]\quad \ell=1,\dots, k-1.
\end{equation}
In particular, for $\ell=k-1$ using \eqref{Agammaell estimate} we have
\begin{equation}
      \frac{d}{d\tau}A_{\gamma_{k-1}}=\frac{(r-k)}{\tau}A_{\gamma_{k-1}}+O(\tau^{-k-\mu}).
\end{equation}
Solving this first order ODE (notice that $k\geq 1$ always holds), we obtain \begin{equation}
    A_{\gamma_{k-1}}=C_{\gamma_{k-1}}\tau^{r-k}+O(\tau^{-k+1-\mu}),
\end{equation}
where $C_{\gamma_{k-1}}$ is a constant. Inductively, we can find some constants $C_{\gamma_{\ell}}$ such that
\begin{equation}
    \begin{cases}
      &A_{\gamma_{\ell}}=O(\tau^{-\ell-\mu}),\quad \text{if $k-r+1\leq \ell \leq k$},\\
      &A_{\gamma_{\ell}}=C_{\gamma_{\ell}}\tau^{-\ell}+O(\tau^{-\ell-\mu}), \quad \text{if $1\leq \ell\leq k-r$.}
    \end{cases}
\end{equation}
By the asymptotics of $A_{\gamma_{1}}, A_{\gamma_{2}}$ and $a^2_{ij}=a_{ii}a_{jj}-A_{\{i,j\}}$, we obtain that \begin{equation}\label{ARtau}
    A(\tau)=\frac{\beta_{n,k}}{\tau}R_{0}\text{diag}\{I_{k-r}, O_{r}\}R_{0}^{-1}+O(\tau^{-1-\mu}),
\end{equation}
where $R_{0}\in \mathrm{SO}(k)$ is a constant matrix. Let
\begin{equation}\label{Q constant}
Q=-\beta_{n,k}R_{0}\text{diag}\{I_{k-r}, O_{r}\}R_{0}^{-1}.
\end{equation}
Then, we obtain
\begin{equation}
\lim_{\tau\to -\infty} \Big\|\,
|\tau| \hat{u}(y, \omega,\tau)- y^\top Qy +2\mathrm{tr}(Q)\, \Big\|_{\mathcal{H}} = 0,
\end{equation}
 where $Q$ is the constant symmetric $k\times k$-matrix in \eqref{Q constant} with  quantized eigenvalues $0$ or $-\frac{\sqrt{2(n-k)}}{4}$. Finally, by standard interior estimates \cite[Thm A.1]{CHH_wing}, we have
  \begin{equation}
\lim_{\tau\to -\infty} \Big\|\,
|\tau| u(y, \omega,\tau)- y^\top Qy +2\mathrm{tr}(Q)\, \Big\|_{C^{p}(B_R)} = 0
\end{equation}
holds for all $R<\infty$ and all integers $p\geq 0$. This completes the proof of Theorem \ref{spectral theorem_restated}.
\end{proof}
\bigskip
\section{Proof of the symmetry improvement theorem}\label{proof_symmetry_improvement}
 In this section, we prove Theorem \ref{Symmetry improvement intro} (symmetry improvement theorem) including cylindrical symmetry improvement theorem and cap improvement theorem. For  readers' convenience, we state the two symmetry improvement theorems separately and give proofs of them below.
\begin{theorem}[cylindrical symmetry improvement theorem]\label{Symmetry improvement I}
There exist constants $L_0<\infty$ large enough and $\varepsilon_0>0$ small enough depending only on dimension  $n$  and satisfying the following properties. Suppose that $\mathcal{M}=\{M_t\}$ is a mean curvature flow  and $(\bar{x}, \bar{t})\in \mathcal{M}$ is a space-time point. If every point in the parabolic neighborhood $\hat{\mathcal{P}}(\bar{x},\bar{t},L_0, L_0^2)$ is $\varepsilon_{0}$-close to a  cylinder $\mathbb{R}^{k}\times S^{n-k}$  and $(\varepsilon, n-k)$-symmetric, where $0< \varepsilon < \varepsilon_0$, then $(\bar{x}, \bar{t})$
is $(\frac{\varepsilon}{2}, n-k)$-symmetric.
\end{theorem}
\begin{proof}
\textbf{Step 1:} Without loss of generality, we may assume $\bar{t}=-1$. By  property of $\varepsilon_{0}$-close to a cylinder $\mathbb{R}^{k}\times S^{n-k}$ in the normalized parabolic neighborhood $\hat{\mathcal{P}}(\bar{x},-1,L_0, L_0^2)$, where we can assume $L_{0}>200n^3$. After appropriate scaling we can write $M_t$ as a radial graph over the shrinking cylinder $ \mathbb{R}^{k}\times S^{n-k}(\sqrt{2(n-k)|t|})$. Namely
\begin{align}\label{cylindrical_coordinates}
     \left\{(z, r\omega) \Big| r=r(z, \omega, t), z\in \mathbb{R}^{k}\cap B(0, \frac{L_0}{2}), \omega \in S^{n-k} \right\} \subset M_t.
\end{align}
Let $0=\lambda_0 < \lambda_1= \lambda_2=\dots \lambda_{n-k+1}=n-k<2(n-k+1)=\lambda_{n-k+1} \leq\lambda_{n-k+2}\leq ...$ be eigenvalues of $-\triangle_{S^{n-k}}$ and let $Y_i(\omega)$ be the eigenfunction corresponding to $\lambda_i$ such that $\{Y_m\}$ forms an  orthonormal basis of $L^{2}(S^{n-k})$. In addition for $1\leq j \leq n-k+1$, the  corresponding eigenfunction $Y_{j}$ of $\lambda_{j}$ is $(n-k+1)^{\frac{1}{2}}\text{Volume}(S^{n-k})^{-\frac{1}{2}}$ times the restriction of $(k+j)$-th coordinate  function of $\{0\}\times \mathbb{R}^{n-k+1}\subset \mathbb{R}^{n+1}$ on $S^{n-k}$.

\noindent \textbf{Step 2: }
For each point $(x_{0},t_{0})$ in the parabolic neighborhood $\hat{\mathcal{P}}(\bar{x},\bar{t},L_0,L_0^2)$, by $(\varepsilon, n-k)$-symmetric assumption we can find a  normalized set of rotation vector fields $\mathcal{K}^{(x_{0},t_{0})} = \{K^{(x_{0},t_{0})}_{\alpha}: 1\leq \alpha \leq \frac{(n-k)(n-k+1)}{2}\}$ such that
\begin{itemize}
    \item $\max_{\alpha}|\langle K_{\alpha}^{(x_{0},t_{0})}, \nu \rangle | H \leq \varepsilon$,
    \item $\max_{\alpha}|  K_{\alpha}^{(x_{0},t_{0})}| H \leq 5n$
\end{itemize}
hold in  $\hat{P}(x_{0},t_{0}, 100n^{5/2}, 100^2n^5)$.
In particular, we let $\bar{\mathcal{K}} = \mathcal{K}^{(\bar{x}, -1)}$ be a  normalized set of rotation vector fields corresponding to $(\bar{x}, -1)$ and by property of $\varepsilon_{0}$-close to a cylinder $\mathbb{R}^{k}\times S^{n-k}$ we also have
\begin{equation}\label{K0}
|K_{\alpha}^{(x_0,t_0)},\nu\rangle|\leq C\varepsilon (-t_{0})^{1/2}.    \end{equation}
By repeatedly applying Lemma \ref{Vector Field closeness on genearalized cylinder_appendix} we have
\begin{align}\label{K difference}
    \inf_{(\omega_{\alpha\beta})\in \mathrm{O}(\frac{(n-k+1)(n-k)}{2})}\sup_{B_{10nL_0 H(\bar{x})^{-1}}(\bar{x})}\max_{\alpha}|\bar{K}_{\alpha}- \sum_{\beta=1}^{\frac{(n-k+1)(n-k)}{2}}\omega_{\alpha\beta}K_{\beta}^{(x_0,t_0)}|\leq C(L_0)\varepsilon.
\end{align}

In particular, by property of $\varepsilon_{0}$-close to a cylinder $\mathbb{R}^{k}\times S^{n-k}$  and Lemma \ref{Vector Field closeness on genearalized cylinder_appendix}, the axes of rotation of $\bar{\mathcal{K}}$ differ from $\mathbb{R}^{k}$ spanned by standard basis  \{$e_{1}, ... ,e_{k}$\} by at most $C(L_0)\varepsilon_{0}$.
Rotating and translating the entire picture by at most $C(L_0)\varepsilon_{0}$, we may assume $\bar{K}_{\alpha} = J_{\alpha}x$ for each $\alpha = 1,2,...,\frac{(n-k+1) (n-k)}{2}$, where  \begin{align}\label{J, Jbar sec4}
        J_{\alpha}  =\begin{bmatrix}
   0 & 0\\
    0 &  \bar{J}_{\alpha}
    \end{bmatrix}\in so(n+1),
\end{align}
 $\{\bar{J}_{\alpha}: 1\leq\alpha\leq \frac{(n-k+1)(n-k)}{2}\}$ is an orthonormal basis of $so(n-k+1)\subset so(n+1)$. To simplify notations, for every fixed $1\leq \alpha \leq \frac{(n-k)(n-k+1)}{2}$ we let
\begin{equation}\label{def_u}
\bar{v}(z, \omega, t) = \langle \bar{K}_{\alpha}, \nu\rangle (z, \omega, t),
\end{equation}
where $\nu$ is the unit normal vector field on $M_{t}$.  By \eqref{K difference} we can find constants $a_{ij}$  depending on $(x_0, t_0)$, where $0\leq i \leq k$, $1\leq j \leq n-k+1$, such that $|a_{ij}|\leq C(L_0)\varepsilon$ and the left hand side of following estimates vanishes on cylinder and in particular by property of $\varepsilon_{0}$-close to a cylinder $\mathbb{R}^{k}\times S^{n-k}$
\begin{align}
    \left|\langle \bar{K}_{\alpha}-K_{\alpha}^{(x_0,t_0)},\nu\rangle - \sum_{j=1}^{n-k+1}\left(a_{0j} + \sum_{i=1}^{k} a_{ij}z_i\right)Y_j \right|\leq C(L_0)\varepsilon\varepsilon_0
\end{align}
holds in $\hat{\mathcal{P}}(x_0,t_0,100n^{\frac{5}{2}},100^2 n^5)$.

This and \eqref{K0} imply that
\begin{align}\label{PartB fine estimate for u}
    \left|\bar{v} - \sum_{j=1}^{n-k+1}\left(a_{0j} + \sum_{i=1}^{k} a_{ij}z_i\right)Y_j \right|\leq C(L_0)\varepsilon\varepsilon_0 + C(-t_{0})^{\frac{1}{2}}\varepsilon
\end{align}
holds in  $\hat{\mathcal{P}}(x_0,t_0,100n^{\frac{5}{2}},100^2 n^5)$.
Then, we aim to improve this error estimates by suitably adjusting these constants. To this end, we notice that the function $\bar{v}$ satisfies parabolic Jacobi equation
\begin{align}\label{PartB Jacobi Equation}
    \frac{\partial \bar{v}}{\partial t} = \Delta \bar{v} + |A|^2\bar{v}
\end{align}
and the rough estimate
\begin{align}
    |\bar{v}| \leq C(L_0) \varepsilon.
\end{align}
By this and the  standard interior estimate we have higher derivative estimates
\begin{align}
    |\nabla \bar{v}| + |\nabla^2 \bar{v}| \leq C(L_0) \varepsilon.
\end{align}
Hence we can write the parabolic Jacobi equation (\ref{PartB Jacobi Equation}) as its linearization plus controllable errors in the  cylindrical coordinates:
\begin{align}
    \left|\frac{\partial \bar{v}}{\partial t}-\sum_{i=1}^{k}\frac{\partial^2 \bar{v}}{\partial z_i^2} - \frac{\Delta_{S^{n-k}}\bar{v}}{-2(n-k)t} - \frac{1}{-2t}\bar{v}\right|\leq C(L_0)\varepsilon\varepsilon_0,
\end{align}
which holds for $|z|\leq \frac{L_0}{4}$ and $-\frac{L_0^2}{16}\leq t \leq -1$.
Then, for $\ell > 1$, we define
\begin{align}
    \Omega_{\ell}&(\bar{z}) = \{z=(z_1,...,z_k)\in \mathbb{R}^k: |z_i-\bar{z}_i|\leq \ell\}\quad   \Omega_{\ell}  = \Omega_{\ell}(0),\\
    \Gamma_{\ell}   &=\{(z, \omega, t)\in \Omega_{\ell} \times S^{n-k}\times [-{\ell}^2,-1] : \ t = -{\ell}^2 \text{ or } \max_{1\leq i\leq k}|z_i| = {\ell}  \}.
\end{align}
Let $\tilde{v}$ to be the solution to the linearized parabolic Jacobi equation:
\begin{align}
    \frac{\partial \tilde{v}}{\partial t} = \sum_{i=1}^{k}\frac{\partial^2 \tilde{v}}{\partial z_i^2} + \frac{\Delta_{S^{n-k}}\tilde{v}}{-2(n-k)t} + \frac{1}{-2t}\tilde{v}
\end{align}
in $\Omega_{\frac{L_0}{4\sqrt{n}}}\times S^{n-k}\times [-\frac{L_0^2}{16n},-1]$ with the boundary condition $\tilde{v} = \bar{v}$ on
$\Gamma_{\frac{L_0}{4\sqrt{n}}}$.

By the maximum principle,
\begin{align}\label{utildeu}
    |\bar{v}-\tilde{v}|\leq C(L_0)\varepsilon \varepsilon_0.
\end{align}
Next, we aim to improve estimates of $w$ in \eqref{PartB fine estimate for u}  by estimating Fourier coefficients of $\tilde{v}$:
\begin{align}
    v_m = \int_{S^{n-k}}\tilde{v}(z,\omega, t)Y_m(\omega) d\omega.
\end{align}

Direct computation shows that $v_m$ satisfies the following equation
\begin{align}\label{PartB Equation for Fourier coefficients}
    \frac{\partial v_m}{\partial t} = \sum_{i=1}^{k}\frac{\partial ^2 v_m}{\partial z_i^2} + \frac{n-k-\lambda_m}{-2(n-k)t}v_m.
    \end{align}
Let $\hat{v}_m = v_m (-t)^{\frac{n-k-\lambda_m}{2(n-k)}}$.
Then $\hat{v}_m$ satisfies the linear heat equation
\begin{align}
    \frac{\partial \hat{v}_m}{\partial t} = \sum_{i=1}^{k}\frac{\partial^2 \hat{v}_m}{\partial z_i^2}.
\end{align}
\textbf{Case 1:} $m\geq n-k+2$.

In this case, $\lambda_m\geq 2(n-k+1)$. Taking $L^2$ inner product with $Y_m$ on both sides of (\ref{PartB fine estimate for u}) and multiplying both sides by $(-t)^{\frac{n-k-\lambda_m}{2(n-k)}}$, we obtain that
\begin{align}\label{Fine estimate for higher modes}
    |\hat{v}_m| \leq (C(L_0)\varepsilon\varepsilon_0 + C\varepsilon)(-t)^{1-\frac{\lambda_m}{2(n-k)}}.
\end{align}
Using the  solution formula of the heat equation with initial and boundary conditions, we obtain
\begin{align}\label{Heat equation solution formula}
    \hat{v}_m (z,t) &\!\!=\!\! \int_{\Omega_{\frac{L_0}{4\sqrt{n}}}} \!\!\!\!K_{t+\frac{L_0^2}{16n}}(z,y)\hat{v}_m(y,\!-\frac{L_0^2}{16n})dy \\
    &-\!\!\int_{-\frac{L_0^2}{16n}}^{t}\int_{\partial\Omega_{\frac{L_0}{4\sqrt{n}}}}\!\!\!\!\!\!\partial_{\nu_y}K_{t-s}(z,y)\hat{v}_m(y, s)dy ds, \notag
\end{align}
where one can use infinite reflection method to show that the Dirichlet heat kernel (Green's function of heat equation with Dirichlet boundary condition) is given by
\begin{align}
    K_t(z,y) =\!\!\!\!\!\!\sum_{\delta\in \{1,-1\}^{k},l\in\mathbb{Z}^{k}}\!\! \frac{(-1)^{\frac{\delta_{1}+\dots+\delta_{k}-k}{2}}}{(4\pi t)^{\frac{k}{2}}}\prod_{i=1}^{k}
    \exp\left(-\frac{\left(z_i-\delta_iy_i-(4l_{i}+1-\delta_{i})\frac{L_0}{4\sqrt{n}}\right)^2}{4t}\right),
\end{align}
and has the similar heat kernel estimates as in \cite[Appendix 5.2]{Zhu}
\begin{equation}
   \int_{\partial\Omega_{\frac{L_0}{4\sqrt{n}}}}\left|\partial_{\nu_y} K_{t-s}(z,y)\right|dy\leq C(k)\frac{L^{k}_{0}}{(t-s)^{1+\frac{k}{2}}}e^{-\frac{L_0^2}{1000n(t-s)}}.
\end{equation}

Putting this into (\ref{Heat equation solution formula}) and using  heat kernel estimates \cite[Appendix 5.2]{Zhu2} and (\ref{Fine estimate for higher modes}), we have the following estimate  in 
$\Omega_{(100n)^{3}}\times [-(100n)^6,-1]$
\begin{align}\label{hat vm estimates}
    &\quad|\hat{v}_m  (z,t)|\leq  (C(L_0)\varepsilon\varepsilon_0 + C\varepsilon)\left(\frac{L_0^2}{16n}\right)^{1-\frac{\lambda_m}{2(n-k)}}\\
    &+(C(L_0)\varepsilon\varepsilon_0 + C\varepsilon)\int_{-\frac{L_0^2}{16n}}^t \frac{L_0^k}{(t-s)^{1+\frac{k}{2}}}e^{-\frac{L_0^2}{1000n(t-s)}}(-s)^{1-\frac{\lambda_m}{2(n-k)}}ds \notag\\
    &\leq (C(L_0)\varepsilon\varepsilon_0 + C\varepsilon)\Bigg[\left(\frac{L_0^2}{16n}\right)^{1-\frac{\lambda_m}{2(n-k)}} +\int_{-\frac{L_0^2}{16n}}^t L_0^{-2}e^{-\frac{L_0^2}{2000n(t-s)}}(-s)^{1-\frac{\lambda_m}{2(n-k)}}ds\Bigg] \notag\\
    &\leq (C(L_0)\varepsilon\varepsilon_0 + C\varepsilon)\Bigg[\left(\frac{L_0^2}{16n}\right)^{1-\frac{\lambda_m}{2(n-k)}} +\int^t_{(1+\frac{1}{\sqrt{\lambda_m}})t}L_0^{-2}e^{-\frac{L_0^2}{2000n(t-s)}}ds \notag\\
    &+L_0^{-\frac{1}{n-k}}\int_{-\frac{L_0^2}{16n}}^{(1+\frac{1}{\sqrt{\lambda_m}})t} (-s)^{\frac{1-\lambda_m}{2(n-k)}}ds\Bigg] \notag\\
    &\leq (C(L_0)\varepsilon\varepsilon_0 + C\varepsilon)\Bigg[\left(\frac{L^2_0}{16n}\right)^{1-\frac{\lambda_m}{2(n-k)}} +\frac{-t}{L_0^2\sqrt{\lambda_m}}e^{\frac{L_0^2\sqrt{\lambda_m}}{2000nt}}\notag\\
    &+ L_0^{-\frac{1}{n-k}}\left(1+\frac{1}{\sqrt{\lambda_m}}\right)^{\frac{2n-2k+1-\lambda_m}{2(n-k)}}\Bigg]. \notag
\end{align}

Note that all the eigenvalues of $-\triangle_{S^{n-k}}$ are in the form of $l(l+n-k-1)$ for integers $l\geq 1$ with multiplicity $N_l =
\tbinom{n+l-k}{n-k}-\tbinom{n+l-k-2}{n-k}$. Moreover, we know that for each eigenvalue $\lambda_{m}$, the corresponding eigenfunction $Y_{m}$ satisfies  $\sup_{S^{n-k}}|Y_m|\leq C\lambda_{m}^{n-k}$ and $\lambda_{m}\approx m^{\frac{2}{n-k}}$  as $m\to +\infty$. Together with the above estimates \eqref{hat vm estimates}, we  conclude that
\begin{align}
    \sup_{S^{n-k}}\sum_{m\geq n-k+2}|\hat{v}_m Y_m| \leq (C(L_0)\varepsilon\varepsilon_0 + C\varepsilon)L_0^{-\frac{1}{n-k}}
\end{align}
holds in $\Omega_{(100n)^{3}}\times [-(100n)^6,-1]$.

\textbf{Case 2:} $1\leq m \leq n-k+1$.

In this case, $\lambda_m = n-k$. Therefore the equation (\ref{PartB Equation for Fourier coefficients}) becomes
\begin{align}\label{PartB Heat Equation firt mode}
    \frac{\partial v_m}{\partial t} = \sum_{i=1}^{k}\frac{\partial^2 v_m}{\partial z_i^2}.
\end{align}

Taking $L^2$ inner product with $Y_m$ in (\ref{PartB fine estimate for u}), we obtain that
\begin{align}\label{rough vm}
    \left|v_m-\left(a_{0m} + \sum_{i=1}^{k} a_{im}z_i\right)\right|\leq (C(L_0)\varepsilon\varepsilon_0 + C\varepsilon)(-t)^{\frac{1}{2}}
\end{align}
holds in $\Omega_{(-t_0)^{\frac{1}{2}}}\times [2t_0,t_0]$.

Then interior gradient estimate implies
\begin{align}
    |\frac{\partial^2 v_m}{\partial z_i\partial z_j }|\leq (C(L_0))\varepsilon\varepsilon_0 + C\varepsilon)(-t)^{-\frac{1}{2}}\quad 1\leq i, j \leq k
\end{align}
holds in $\Omega_{\frac{L_0}{8\sqrt{n}}}\times [-\frac{L_0^2}{64n},-1]$.

Next, we observe that each second derivative $\frac{\partial^2 v_m}{\partial z_i\partial z_j }$ ($1\leq i, j\leq k$) satisfies the same heat equation
(\ref{PartB Heat Equation firt mode}). Hence, estimating solution formula similarly as above implies
\begin{align}\label{Heat equation solution formula b}
    |\frac{\partial^2 v_m}{\partial z_i\partial z_j}| &\leq \int_{\Omega_{\frac{L_0}{4\sqrt{n}}}} K_{t+\frac{L_0^2}{16n}}(z,y)\left|\frac{\partial^2 v_m}{\partial z_i\partial z_j}(y,-\frac{L_0^2}{16n})\right|dy \\
    - &\int_{-\frac{L_0^2}{16n}}^{t}\int_{\partial\Omega_{\frac{L_0}{4\sqrt{n}}}}|\partial_{\nu_y}K_{t-s}(z,y)|\left|\frac{\partial^2 v_m}{\partial z_i\partial z_j}(y,s)\right|dy ds \notag\\
    \leq & (C(L_0)\varepsilon\varepsilon_0 + C\varepsilon)\left[\left(\frac{L_0^2}{16n}\right)^{-\frac{1}{2}} +\int_{-\frac{L_0^2}{16n}}^t L_0^{-2}(-s)^{-\frac{1}{2}}ds\right] \notag\\
    \leq & (C(L_0)\varepsilon\varepsilon_0 + C\varepsilon)L_0^{-1} \notag
\end{align}
holds in $\Omega_{(100n)^{3}}\times [-(100n)^6,-1]$.

This means that for each $0\leq i \leq k$ and $1\leq m \leq n-k+1$. we can find constants $A_{im}$,  such that
\begin{align}
     \left|v_m -   \left(A_{0m} + \sum_{i=1}^{k}A_{im}z_i\right)  \right|\leq C(L_0)\varepsilon\varepsilon_0 + CL_0^{-1}\varepsilon
\end{align}
and
\begin{equation}
    |A_{im}|\leq
    C(L_0)\varepsilon\varepsilon_0 + CL_0^{-1}\varepsilon
\end{equation}
hold in $\Omega_{(100n)^{3}}\times [-(100n)^6,-1]$.

\textbf{Case 3:} $m=0$.

Under cylindrical coordinates defined in \eqref{cylindrical_coordinates}, the normal vector $\nu$ satisfies
\begin{align}
    \nu \sqrt{1+ \frac{|\nabla_{S^{n-k}}r|^2}{r^2} + |D_{\mathbb{R}^k}r|^2} = (\omega, -D_{\mathbb{R}^k}r) + (-\frac{\nabla_{S^{n-k}}r}{r}, 0).
\end{align}

Hence for the position vector  $x= (r\omega,z)$,
\begin{align}
    \bar{v}\sqrt{1+ \frac{|\nabla_{S^{n-k}}r|^2}{r^2} + |D_{\mathbb{R}^k}r|^2} = \langle J_{\alpha}x, (\omega-\frac{\nabla_{S^{n-k}}r}{r}, -D_{\mathbb{R}^k}r)  \rangle.
\end{align}
Then we have
\begin{align}\label{uidentity}
    \bar{v}\sqrt{1+ \frac{|\nabla_{S^{n-k}}r|^2}{r^2} + |D_{\mathbb{R}^k}r|^2} &= (r\omega, z) J_{\alpha}(\omega-\frac{\nabla_{S^{n-k}}r}{r}, -D_{\mathbb{R}^k}r)^{T} \\\nonumber
    &= -\langle \bar{J}_{\alpha}\omega, \nabla_{S^{n-k}}r\rangle.
\end{align}
Integrating the identity over the unit sphere $S^{n-k}$, we obtain
\begin{align}
    \int_{S^{n-k}}\bar{v}\sqrt{1+ \frac{|\nabla_{S^{n-k}}r|^2}{r^2} + |D_{\mathbb{R}^k}r|^2}d\omega  &= - \int_{S^{n-k}} \langle \bar{J}_{\alpha}\omega, \nabla_{S^{n-k}}r\rangle d\omega \\\nonumber
    &= \int_{S^{n-k}}r \text{div}_{S^{n-k}}(\bar{J}_{\alpha}\omega)d\omega =0,
\end{align}
where we used the fact that $J_{\alpha}\omega$ is a divergence free vector field on $S^{n-k}$.

Since $|\bar{v}|\leq C(L_0)\varepsilon$ and $|\frac{\nabla_{S^{n-k}}r}{r}|+|D_{\mathbb{R}^k}r| \leq C(L_0)\varepsilon_0$, we
have
\begin{align}
    \left|\int_{S^{n-k}}\bar{v}(z, \omega ,t)d\omega \right|\leq C(L_0)\varepsilon\varepsilon_0
\end{align}
in $\Omega_{(100n)^{3}}\times [-(100n)^6,-1]$. This gives $|v_0|\leq C(L_0)\varepsilon_0\varepsilon$ and completes the 0-mode analysis.

\textbf{Step 3: } Adjusting the axis. From the previous steps and \eqref{utildeu}, we can find constants  $A_{\alpha, ij}$ ($1\leq \alpha\leq \frac{(n-k)(n-k+1)}{2}$, $0\leq i \leq k$ and $1\leq j \leq n-k+1$) such that
\begin{align}\label{Adjust u_alpha}
    \left|\langle \bar{K}_{\alpha},\nu\rangle - \sum_{j=1}^{n-k+1}\left(A_{\alpha, 0j} + \sum_{i=1}^{k} A_{\alpha, ij}z_i\right)Y_j \right|\leq C(L_0)\varepsilon\varepsilon_0 + CL_0^{-\frac{1}{n-k}}\varepsilon
\end{align}
and
\begin{equation}\label{Asmall}
    |A_{\alpha, ij}|\leq C(L_0)\varepsilon.
\end{equation}

Let us define
\begin{align}
    F_j(z) = \int_{S^{n-k}}r(z, \omega, -1)Y_j(\omega)d\omega\quad j=1,\dots, n-k+1.
\end{align}
We aim to use these functions to construct anti-symmetric matrices and translation.
We first  extract information about $A_{\alpha,ij}$ from \eqref{Adjust u_alpha}. Therefore, for $1\leq j \leq n-k+1$ we project \eqref{Adjust u_alpha} to each spherical eigenfunction $Y_j$ and obtain
\begin{align}\label{A_alpha,ij estimate}
    \left|\int_{S^{n-k}}\langle\bar{K}_{\alpha},\nu\rangle Y_j d\omega -A_{\alpha,0j} - \sum_{i=1}^{k}A_{\alpha,ij}z_i\right|\leq C(L_0)\varepsilon\varepsilon_0 + CL_0^{-\frac{1}{n-k}}\varepsilon.
\end{align}
By \eqref{def_u} and \eqref{uidentity} we have
\begin{align}
    \int_{S^{n-k}}\langle\bar{K}_{\alpha},\nu\rangle Y_j \sqrt{1+ \frac{|\nabla_{S^{n-k}}r|^2}{r^2} + |D_{\mathbb{R}^k}r|^2} + \langle \bar{J}_{\alpha}\omega,\nabla_{S^{n-k}}r\rangle Y_j d\omega = 0.
\end{align}
Arguing as in the previous step and using  \eqref{uidentity}, $|\bar{v}|\leq C(L_0)\varepsilon$ and $|\frac{\nabla_{S^{n-k}}r}{r}|+|D_{\mathbb{R}^k}\bar{v}| \leq C(L_0)\varepsilon_0$ we have
\begin{align}\label{divergence formula 1}
    \left|\int_{S^{n-k}}\langle\bar{K}_{\alpha},\nu\rangle Y_j + \text{div}_{S^{n-k}}(rJ_{\alpha}\omega)  Y_j d\omega\right| \leq C(L_0)\varepsilon_0\varepsilon.
\end{align}
Direct computation shows
\begin{align}\label{Integrate by part 1}
    \int_{S^{n-k}}\text{div}_{S^{n-k}}(rJ_{\alpha}\omega)  Y_j(\omega) d\omega&=
     -\int_{S^{n-k}} \langle r\bar{J}_{\alpha}\omega, \nabla_{S^{n-k}}Y_j(\omega)\rangle d\omega\\
     &=-\sum_{i=1}^{n-k+1}\int_{S^{n-k}}\bar{J}_{\alpha,ji} r Y_i(\omega) d\omega \notag\\
     &=-\sum_{i=1}^{n-k+1}\bar{J}_{\alpha,ji}F_i. \notag
\end{align}
Combining (\ref{A_alpha,ij estimate}) (\ref{divergence formula 1}) and (\ref{Integrate by part 1}) we have
\begin{align}\label{A_alpha,ij estimate 2}
    |A_{\alpha,0j} + \sum_{i=1}^{k}A_{\alpha,ij}z_i - \sum_{i=1}^{n-k+1}\bar{J}_{\alpha,ji}F_i|\leq C(L_0)\varepsilon\varepsilon_0 + CL_0^{-\frac{1}{n-k}}\varepsilon.
\end{align}

Then, we define $b=(0,\dots,0, b_{1},\dots, b_{n-k+1})\in \mathbb{R}^{k}\times \mathbb{R}^{n-k+1}$ and $P\perp so(k)\oplus so(n-k+1)\subset so(n+1)$ as follows:
\begin{align}
    b_j &= F_j(0)\quad 1\leq j\leq n-k+1,\\
    P_{k+j,i} &= \frac{1}{2}(F_j(e_i)-F_j(-e_i))\quad   1\leq i\leq k, 1\leq j\leq n-k+1,
\end{align}
where $e_i = (0,...,0,1,0,...,0)$ is the unit vector in $\mathbb{R}^k$ with the only nonzero factor in the $i$-th coordinate. Then by \eqref{uidentity}, $|\bar{v}|\leq C(L_0)\varepsilon$ and $|\frac{\nabla_{S^{n-k}}r}{r}|+|D_{\mathbb{R}^k}\bar{v}| \leq C(L_0)\varepsilon_0$, we have  $|\nabla_{S^{n-k}}r|\leq C(L_{0})\varepsilon$ in $\Omega_{\frac{L_{0}}{4\sqrt{n}}}\times [-1-\frac{L_{0}^2}{16n}, -1]$. By  the construction, the above estimates and $ \int_{S^{n-k}}r(\omega_{0},z, -1)Y_j(\omega)d\omega=0$ for any fixed $\omega_{0}\in S^{n-k}$, we have
\begin{align}\label{[P,J] computation 1}
    [P,J_{\alpha}]_{k+j,i} = -\sum_{l=1}^{n-k+1}\frac{1}{2}\bar{J}_{\alpha,jl}(F_l(e_i)-F_l(-e_i))
\end{align}
and
\begin{equation}\label{Pbestimates}
    |P|+|b|\leq C(L_0)\varepsilon.
\end{equation}
Combining (\ref{A_alpha,ij estimate 2}) and \eqref{[P,J] computation 1} we obtain
\begin{align}
    |A_{\alpha,0j}-(J_{\alpha} b)_{k+j}| \leq & C(L_0)\varepsilon\varepsilon_0 + CL_0^{-\frac{1}{n-k}}\varepsilon
\end{align}
and
\begin{align}\label{approximate A0} |A_{\alpha,ij}+[P,J_{\alpha}]_{k+j,i}|\leq &  C(L_0)\varepsilon\varepsilon_0 + CL_0^{-\frac{1}{n-k}}\varepsilon.
\end{align}

Now we let $S = \exp(P)$ and
\begin{align}
    \tilde{K}_{\alpha}(x) = SJ_{\alpha}S^{-1}(x-b).
\end{align}
Using estimates \eqref{Adjust u_alpha}, \eqref{Pbestimates}-\eqref{approximate A0} and  property of $\varepsilon_{0}$-close to a cylinder $\mathbb{R}^{k}\times S^{n-k}$  we have that
\begin{align}
    \left|\langle\bar{K}_{\alpha}-\tilde{K}_{\alpha},\nu\rangle - \sum_{j=1}^{n-k+1}\left(A_{\alpha, 0j} + \sum_{i=1}^{k} A_{\alpha, ij}z_i\right)Y_j \right|\leq C(L_0)\varepsilon\varepsilon_0 + CL_0^{-\frac{1}{n-k}}\varepsilon
\end{align}
holds in $\hat{\mathcal{P}}(x_0,t_0,100n^{\frac{5}{2}},100^2n^5)$.
This and \eqref{Adjust u_alpha}   imply that
\begin{align}
    \max_{\alpha}|\langle \tilde{K}_{\alpha},\nu\rangle|H \leq C(L_0)\varepsilon\varepsilon_0 + CL_0^{-\frac{1}{n-k}}\varepsilon
\end{align}
holds in $\hat{\mathcal{P}}(x_0,t_0,100n^{\frac{5}{2}},100^2n^5)$.

Finally, we take $L_0$ large enough and subsequently $\varepsilon_0$ small enough,  Theorem \ref{Symmetry improvement I} then follows.
\end{proof}
Then we discuss cap improvement theorem as in \cite[Thm 3.12]{Zhu2}. Let $L_{0}, \varepsilon_{0}$ be constant from Theorem \ref{Symmetry improvement I} (cylindrical symmetry improvement theorem).
\begin{theorem}[cap improvement theorem]\label{cap improvement theorem}
There exist constants $L_1 > L_0$ and $\varepsilon_1<\varepsilon_0$ satisfying the following property: for a mean curvature flow  $\mathcal{M}=\{M_t\}$ and a space-time point $(\bar{x},\bar{t})$, if $\hat{\mathcal{P}}(\bar{x},\bar{t},L_1,L_1^2)$ is $\varepsilon_1$-close to a piece of $\mathbb{R}^{k-1}\times \Sigma$  after rescaling such that $H(\bar{x},\bar{t})=1$, where $\Sigma$ is the unique $(n-k+1)$ dimensional round bowl, and each point in $\hat{\mathcal{P}}(\bar{x},\bar{t},L_1,L_1^2)$ is $(\varepsilon, n-k)$-symmetric, then $(\bar{x},\bar{t})$ is $(\frac{\varepsilon}{2}, n-k)$-symmetric.
\end{theorem}

\begin{proof}
Throughout the proof, $L_1$ is assumed to be large enough depending only on $\varepsilon_0, L_0, n$ and $\varepsilon_1$ is assumed
to be small enough depending only on $L_0, \varepsilon_0, n, L_1$,  and $j$ is a  fixed large integer depending only on $L_0, \varepsilon_0,n$

We assume that $\Sigma$ is the $(n-k+1)$ dimensional round bowl soliton in the subspace $\{0\}\times \mathbb{R}^{n-k+2}\subset \mathbb{R}^{n+1}$ with tip $\Sigma$  at the origin and translates towards positive direction of $x_{k}$ axis with unit speed. Let $e_{k}$ be the unit
vector in the positive $x_{k}$ direction that coincide with the inward normal vector of $\Sigma$ at the origin. We write $\Sigma_t=\Sigma + te_{k}$. Then $\Sigma_t$ is the translating mean curvature flow  and $\Sigma = \Sigma_0$.

After rescaling we may assume that $H(\bar{x},\bar{t}) = 1$ at $\bar{t}=0$. By assumption, there exists a scaling factor $\kappa>0$
such that the parabolic neighborhood $\hat{\mathcal{P}}(\bar{x},0,L_1,L_1^2)$ can be written a graph over the translating
$\mathbb{R}^{k-1}\times\kappa^{-1}\Sigma_t$ with graph norm $\varepsilon_1$ small in $C^{10}$.
By the above setup, the maximal mean curvature of $\kappa^{-1}\Sigma_t$ is $\kappa$. Let $p_t \in \kappa^{-1}\Sigma_t$ be the tip. Let the affine plane $P_t = \mathbb{R}^{k-1}\times \{p_t\}$. The splitting directions  are assumed to be $x_{1},...,x_{k-1}$ axes.
By our normalization, $H(\bar{x},0) = 1$. The maximality of $\kappa$ together with the approximation ensures that $\kappa\geq 1-C\varepsilon_1$.

We define $d(\bar{x},P_t) = \min_{a\in P_t}|\bar{x}-a|$ to be the Euclidean distance between $\bar{x}$ and $P_t$.

\textbf{Step 1} By the structure of the bowl soliton and approximation, there exists $\Lambda_{*}<\infty$ depending only on $\varepsilon_0, L_0,n$ such that
if $d(\bar{x},P_0)\geq \Lambda_{*}$, then for every space-time point $(y,s)\in \hat{\mathcal{P}}(\bar{x},0,L_0,L_0^2)$ we have
\begin{align}
    \hat{\mathcal{P}}(y,s,100n^\frac{5}{2}, 100^2n^5)\subset \hat{\mathcal{P}}(\bar{x},0,L_1,L_1^2),
\end{align}
and $(y,s)$ is $\varepsilon_0$-close  to a cylinder $\mathbb{R}^{k}\times S^{n-k}$. Then we can apply Theorem \ref{Symmetry improvement I} to
conclude that $(\bar{x},\bar{t})$ is $(\frac{\varepsilon}{2}, n-k)$-symmetric and we are done.

\textbf{Step 2} We assume that $d(\bar{x},l_0)\leq \Lambda_{*}$.

By the structure of the bowl soliton, we get
\begin{align}
    \frac{1}{2} < \kappa < C.
\end{align}
We define a series of set $Int_{j}(\hat{\mathcal{P}}(\bar{x},0,L_1,L_1^2))$ inductively by
\begin{align}
    Int_{0}(\hat{\mathcal{P}}(\bar{x},0,L_1,L_1^2))=  \hat{\mathcal{P}}(\bar{x},0,L_1,L_1^2),
\end{align}
and
\begin{align}
    &(y,s)\in Int_j(\hat{\mathcal{P}}(\bar{x},0,L_1,L_1^2)) \\
    &\Longleftrightarrow \hat{\mathcal{P}}(y,s,10^6n^2L_0,10^{12}n^4L_0^2) \subset Int_{j-1}(\hat{\mathcal{P}}(\bar{x},0,L_1,L_1^2)).
\end{align}

We abbreviate $Int_j(\hat{\mathcal{P}}(\bar{x},0,L_1,L_1^2))$ as $Int(j)$, and we notice that when
\begin{equation}
10^{-8j-8}n^{-2j-2}L_0^{-j-1}L_1\geq \Lambda_1,
\end{equation}
we have
\begin{align}
    \hat{\mathcal{P}}(\bar{x},0,10^{-8j}n^{-2j}L_0^{-j}L_1, 10^{-16j}n^{-4j}L_0^{-2j}L_1^2)\subset Int(j).
\end{align}
Then arguing as in \cite[Thm 3.12]{Zhu2}, we can show that there exists $\Lambda_1\gg \Lambda_{*}$ depending only on $\varepsilon_0, L_0$ such that for any  $(x,t)\in Int_j(\hat{\mathcal{P}}(\bar{x},0,L_1,L_1^2))$, if $d(x,P_t)\geq 2^{\frac{j}{100}}\Lambda_1$, then $(x,t)$ is $(2^{-j}\varepsilon, n-k)$-symmetric, where $j\in \mathbb{N}_{+}$.

\textbf{Step 3:}
Then, we define some  regions as follows.
    \begin{align*}
      \Omega_j^t=&\{x\in M_t \ \big|  \max\{|x_{1}|,\dots, |x_{k-1}|\}\leq W_j, d(x,P_t)\kappa\leq D_j\},\\
      \Omega_j=&\bigcup\limits_{t\in[-1-T_j,-1]} \Omega_j^t,\\
      \partial^1\Omega_j=&\{(x,s)\in\partial\Omega_j|d(x,P_s)\kappa= D_j\},\\
      \partial^2\Omega_j=&\{(x,s)\in\partial\Omega_j \ \big| \max\{|x_{1}|,\dots, |x_{k-1}|\}\leq W_j\},
    \end{align*}
    where $D_j=2^{\frac{j}{100}}\Lambda_1$, $W_j=T_j=2^{\frac{j}{50}}\Lambda_1^2$.

    By repeatedly applying Lemma \ref{Vector Field closeness on genearalized cylinder_appendix} and
    Lemma \ref{VF close Bowl times lines}, we obtain a  normalized set of rotation vector fields $\mathcal{K}^{(j)} =\{K_{\alpha}^{(j)}, 1\leq \alpha\leq\frac{(n-k+1)(n-k)}{2}\}$ such that
    \begin{align}
      &\max_{\alpha}|\left<K^j_{\alpha},\nu\right>|H\leq C(W_j+D_j+T_j)^2 2^{-j}\varepsilon  \text{ on }\partial^1\Omega_j^t \label{eq1},\\
      &\max_{\alpha}|\left<K^j_{\alpha},\nu\right>|H \leq C(W_j+D_j+T_j)^2 \varepsilon \text{ on }\Omega_j \label{eq2},\\
      &\max_{\alpha}|K^j_{\alpha}|H \leq 2n\text{ on }\Omega_j \label{eq3}.
    \end{align}

    Now  for each fixed large integer $j$ depending only on $L_0, \varepsilon_0, n$, we define single variable function
    \begin{equation}
        \phi(s)=\frac{c_n^2}{n}D_j^{-1}\log(\cosh(s)).
    \end{equation}
    Then $\phi'$ and $\phi''$ satisfy the following $L^{\infty}$-norm estimates:
    \begin{align}
        |\phi'|_{\infty}\leq \frac{c_n}{n}D_j^{-\frac{1}{2}} \quad         |\phi''|_{\infty}\leq \frac{c_n^2}{n}D_j^{-1}.    \end{align}

    Then, we let $\displaystyle\Phi(x) = \sum_{l=1}^{k-1}\phi(x_l)$.
        For each $K^j_{\alpha}$ ($1\leq \alpha\leq\frac{(n-k+1)(n-k)}{2}$), we define the following function
      \begin{equation}
           f(x,t)=e^{-\Phi(x)+\lambda(t-\bar{t})}\frac{\langle K_{\alpha}^j,\nu\rangle}{H-\mu},
      \end{equation}
     where $\lambda,\mu$ will be determined as follows in \eqref{def lambda mu}. By the asymptotics of the bowl soliton, we can find $c_n\in (0,1)$ such that
    $H\geq 2c_n D_j^{-1/2}$ in $\Omega_j$ and we let
    \begin{equation}\label{def lambda mu}
        \lambda=\frac{c_n^2}{n}D_j^{-1}, \mu=k c_n D_j^{-1/2}.
    \end{equation}

     Then the evolution equation for  $f$ is
     \begin{align}\label{EQNforf}
       (\partial_t -\Delta )f      =&\left(\lambda-\frac{\mu|A|^2}{H-\mu}-\partial_t\Phi+\Delta \Phi +|\nabla\Phi|^2+2\frac{\left<\nabla \Phi,\nabla H\right>}{H-\mu}\right)f\\
       &+2\left<\nabla f,\nabla\Phi+\frac{\nabla H}{H-\mu}\right>. \notag
     \end{align}

    Using \eqref{def lambda mu} and computing similarly as in \cite[Thm 3.12]{Zhu2} we have
   \begin{align}
        &\lambda-\frac{\mu|A|^2}{H-\mu}-\partial_t\Phi+\Delta \Phi +|\nabla\Phi|^2+2\frac{\left<\nabla \Phi,\nabla H\right>}{H-\mu} \\\nonumber
        \leq& \lambda - \frac{4\mu^2}{n} + kC(L_1)\varepsilon_1|\phi'|_{\infty} + k|\phi''|_{\infty} + k|\phi'|_{\infty}^2 + kC(L_1)\varepsilon_1|\phi'|_{\infty} \mu^{-1} < 0
    \end{align}
     and
     \begin{equation}
    |f|\leq C2^{-\frac{j}{5}}\varepsilon \quad \text{on $\partial\Omega_j$},
     \end{equation}
 where $C$ depends only on $n$. Then maximum principle gives
     \begin{align}\label{maxf}
        \sup\limits_{\Omega_j}|f|\leq &\sup\limits_{\partial\Omega_j}|f|\leq  2^{-\frac{j}{5}}C\varepsilon.
     \end{align}
    Next, we have $|\langle K_{\alpha}^j,\nu\rangle|H\leq 2^{-j/10}C_{1}\varepsilon$ holds in $\hat{\mathcal{P}}(\bar{x}, 0, 100n^{5/2}, 100^2n^5)$. Then, arguing again as in \cite[Thm 3.12]{Zhu2}, we first fix a large $j$ such that $\Omega_j$ contains the parabolic neighborhood $\hat{\mathcal{P}}(\bar{x},\bar{t},100n^{\frac{5}{2}}, 100^2n^5)$ and $2^{-\frac{j}{10}}C_{1} <\frac{1}{2}$. Next, we find $L_1$ large enough depending only on $n,L_0,\varepsilon_0,j$,  finally we find $\varepsilon_1$ small enough depending on all the above constants. We will obtain that
     $(\bar{x},\bar{t}) = (\bar{x},0)$ is $(\frac{\varepsilon}{2}, n-k)$-symmetric.




\end{proof}

\section{Asymptotics and symmetry of non-degenerate noncollapsed solutions}\label{rkksection}
In this section, we prove Theorem \ref{compact_symmetry_oval} (asymptotics, compactness and symmetry of $k$-ovals in $\mathbb{R}^{n+1}$).  To this end, we need the to first establish  the sharp asymptotic for $k$-ovals in $\mathbb{R}^{n+1}$. In this section we assume that $\mathcal{M}=\{M_t\}$  is ancient noncollapsed solution whose tangent flow  at $-\infty$ is given by $\mathbb{R}^{k}\times S^{n-k}(\sqrt{2(n-k)|t|})$. By \cite[Lem 3.5]{DH_ovals} $M_{t}$ is uniformly $(k+1)$-convex. Then we have the following theorem.
\begin{theorem}\label{Rk=k asymptotics}
If $\mathrm{rk}(Q)=k$, then the ancient noncollapsed flow $\mathcal{M}=\{M_t\}$  is compact and satisfies the following sharp asymptotics under suitable coordinates:
\begin{itemize}
    \item Parabolic region: Given any $R>0$, the cylindrical profile function $u$ for $\tau\to -\infty$ satisfies
    \begin{equation*}
     u(y,\omega,\tau)=\frac{\sqrt{2(n-k)}}{4}\frac{|y|^2-2k}{\tau} +o(|\tau|^{-1})
    \end{equation*}
        uniformly for $|y|\leq R$.
      \item Intermediate region: Let $\bar{u}(z,\omega, \tau)=u(|\tau|^{\frac{1}{2}}z,\omega, \tau)+\sqrt{2(n-k)}$, we have
\begin{equation*}
    \lim_{\tau\rightarrow -\infty}\bar{u}(z,\omega, \tau)=\sqrt{(n-k)(2-|z|^2)}
\end{equation*}
    uniformly on every compact subset of $\{(z, \omega): |z|<\sqrt{2}, \omega \in S^{n-k}\}$.
    \item Tip region:  Setting $\lambda(s)=\sqrt{|s|^{-1}\log |s|}$, and given any direction $\vartheta\in S^{k-1}\cap (\mathbb{R}^{r}\times \{0\})$ letting $p_{s}\in M_s$ be the point that maximizes $\langle p, \vartheta\rangle$ among all $p\in M_s$, as $s\to -\infty$ the rescaled flows
    \begin{equation*}
        {\widetilde{M}}^{s}_{t}=\lambda(s)\cdot(M_{s+\lambda(s)^{-2}t}-p_{s})
    \end{equation*}
    converge to $\mathbb{R}^{k-1}\times N_{t}$, where $N_{t}$ is the $(n+1-k)$-dimensional round translating bowl in $\mathbb{R}^{n+1-k}$ with speed $1/\sqrt{2}$.
\end{itemize}
\end{theorem}
\begin{proof}
 The sharp asymptotics in parabolic region follows from Theorem \ref{spectral theorem_restated} (spectral quantization theorem) in full rank case and standard interior estimates.  Moreover, there exist $\tau_\ast>-\infty$ and an increasing function $\delta:(-\infty,\tau_\ast)\to (0,1/100)$ with $\lim_{\tau\to -\infty}\delta(\tau)=0$ such that for $\tau\leq \tau_\ast$ we have

    \begin{equation}\label{inw_quad22}
\sup_{|y|\leq \delta(\tau)^{-1}} \left|    u(y,\omega,\tau) - \frac{\sqrt{2(n-k)}}{4}\frac{|y|^2-2k}{\tau} \right|\leq \frac{\delta(\tau)}{|\tau|}\, .
    \end{equation}
By convexity and the   quadratic bending sharp asymptotics in parabolic region, we have that $M_t$ is compact.

Then we discuss the sharp asymptotics in intermediate region. We first show the lower bound. For any compact subset $K\subset \{ |z|^2<2\}$, we have
\begin{equation}\label{lower_bound}
    \liminf_{\tau\rightarrow-\infty} \inf_{z\in K, \omega\in S^{n-k}} \left(\bar{u}(z, \omega, \tau)-\sqrt{(n-k)(2-|z|^2)}\right)\geq 0.
\end{equation}
By \cite[Lemma 4.4]{ADS1} there exists an increasing positive function $M(a)$ with $\lim_{a\to \infty}M(a)= \infty$, such that the profile function $u_a$ of the ADS-barrier $\Sigma_a$ defined in \eqref{ads_shrinker} for $0\leq r\leq M(a)$ satisfies
\begin{equation}
    u_{a}(r)\leq \sqrt{2(n-k)}(1-\frac{r^2-3}{2a^2})\, .
\end{equation}
We fix $\tau_\ast$ negative enough, and for $\tau\leq\tau_\ast$ we set
\begin{equation}
    L(\tau)=\min\{\delta(\tau)^{-1}, M(|\tau|^{\frac{1}{2}}), |\tau|^{\frac{1}{2}-\frac{1}{100}}\} \, ,\quad \hat{a}(\tau)=\sqrt{\frac{2|\tau|}{1+L(\tau)^{-1}}}\, ,
\end{equation}
where $\delta(\tau)$ is the function from sharp asymptotics \eqref{inw_quad22} in parabolic region.
Then, for $\hat{\tau}\leq\tau_\ast$, we get
\begin{align}\label{b2}
    u_{\hat{a}(\hat{\tau})}\left(L(\hat{\tau})+3L(\hat{\tau}\right)^{-1})
    \leq \sqrt{2(n-k)} -\frac{\sqrt{2(n-k)}L(\hat{\tau})^2}{4|\hat{\tau}|}\, .
\end{align}
On the other hand, by  parabolic region asymptotics we have
\begin{equation}\label{b1}
       \sqrt{2(n-k)}+ u(y,\omega, \tau)\geq \sqrt{2(n-k)} -\frac{\sqrt{2(n-k)}L(\hat{\tau})^2}{4|\hat{\tau}|}\, ,
    \end{equation}
whenever $|y|=L(\hat{\tau})$ and $\tau\leq \hat{\tau}$.
Hence, if we consider the shifted and rotated hypersurfaces
\begin{equation}\label{rotated_barrier_rest}
\Gamma^{\eta}_a=\{(y, y'')\in  \mathbb{R}^{k}\times \mathbb{R}^{n+1-k}:  (|y|-\eta, y'') \in {\Sigma}_a \},
\end{equation}
with $\eta=\hat{\eta}(\hat{\tau})=3L(\hat{\tau})^{-1}$ and $a=\hat{a}(\hat{\tau})$ as above, then applying the inner barrier principle from Corollary \ref{emma_inner_barrier}, we infer that
\begin{equation}
   \sqrt{2(n-k)}+u(y,\omega,\tau) \geq    u_{\hat{a}(\hat{\tau})}\left(|y|+\hat{\eta}(\hat{\tau})\right) \, ,
\end{equation}
whenever  $ |y|\geq L(\hat{\tau})$ and  $\tau\leq \hat{\tau}$. Moreover, by \cite[Lemma 4.3]{ADS1}, we have
\begin{equation}
u_a(r)=\sqrt{2(n-k)(1-\frac{r^2}{a^2})}+o(1)
\end{equation}
uniformly in $r$ as $a\to \infty$. Together with convexity we conclude that
\begin{equation}
    \liminf_{\tau\rightarrow-\infty} \inf_{z\in K, \omega\in S^{n-k}} \left(\bar{u}(z,\omega, \tau)-\sqrt{(n-k)(2-|z|^2)}\right)\geq 0.
\end{equation}
This proves \eqref{lower_bound}.

Observe that by the lower bound, for any angle $\vartheta\in S^{k-1}$ we have
 \begin{equation}
 \max_{p\in M_t}\langle p, \vartheta \rangle \geq  \sqrt{(2-o(1))|t|\log|t|}\, .
\end{equation}
In particular, the cylindrical profile function $u(y,\omega,\tau)$ is well-defined whenever $|y|\leq \sqrt{(2-o(1))|\tau|}$.
To proceed, we prove almost spherical symmetry estimates away from the tip region.
\begin{claim}[almost spherical symmetry estimates]\label{lemma_almost_symm_further}
For every $\delta>0$, there exist constants $\eta>0$ and $\tau_\ast>-\infty$, such  that for all $\tau\leq\tau_\ast$ we have
\begin{equation}
 \sup_{|y|\leq \sqrt{(2-\delta)|\tau|}}    |\nabla_{S^{n-k}}u(y,\omega,\tau)|\leq e^{-\eta |\tau|^{1/2}}\, .
\end{equation}
\end{claim}

\begin{proof}[Proof of claim] We first claim that given any $\eps_0>0$ we can find $T_0>-\infty$ so that all points $X=(p,t)\in \mathcal{M}$ with $|p|\leq \sqrt{(2-\delta/2)|t|\log|t|}$ and $t\leq T_0$ lie on the center of an $\eps_0$-cylinder. Indeed, if there exists a sequence $t_i\to -\infty$ and points $p_i\in M_{t_i}$ that do not lie on the center of an $\eps_0$-cylinder, then by the global convergence theorem \cite[Theorem 1.12]{HaslhoferKleiner_meanconvex}, the above lower bound estimates \eqref{lower_bound} in intermediate region, inward quadratic bending asymptotics in parabolic region and convexity, after passing to a subsequence
\begin{equation}
\widetilde{M}^i_t:=H(p_i,t_i)\cdot(M_{t_i+H(p_i,t_i)^{-2}t}-p_{i})
\end{equation}
would converge to a limit $M^{\infty}_{t}$ that splits off $k$ lines.  Hence by  these properties,  uniformly $k+1$ convexity of $M_{t}$  and \cite[Lemma 3.14]{HaslhoferKleiner_meanconvex}, $M^{\infty}_{t}$ is a round shrinking $\mathbb{R}^{k}\times S^{n-k}$. This is a contradiction and thus establishes the claim. Using the property that every point under consideration is $\eps_0$-close to a cylinder $\mathbb{R}^{k}\times S^{n-k}$, and arguing similarly as in the proof of Proposition \ref{utheta} (almost spherical symmetry) the assertion follows.
\end{proof}

We can now establish a matching upper bound. For any compact subset $K\subset \{ |z|^2<2\}$ we have
\begin{equation}\label{upper_bound}
    \limsup_{\tau\rightarrow-\infty} \sup_{z\in K} \left(\bar{u}(z,\omega, \tau)-\sqrt{(n-k)(2-|z|^2)}\right)\leq 0.
\end{equation}
By convexity and negative definite quadratic form estimates, we have
\begin{equation}\label{negative definite}
    \triangle_{\mathbb{R}^{k}} v-\frac{\nabla^2_{\mathbb{R}^{k}}v(\nabla_{\mathbb{R}^{k}}v, \nabla_{\mathbb{R}^{k}}v )}{1+|\nabla_{\mathbb{R}^{k}}v|^2}\leq 0.
\end{equation}
By the evolution equation \eqref{v_evolution}, \eqref{negative definite} and Claim \ref{lemma_almost_symm_further} (almost spherical symmetry estimates) the graphical function $v=\sqrt{2(n-k)}+u$ satisfies
\begin{align}\label{v_evolution_inequ}
    v_\tau\leq
     -\frac{n-k}{v}+\frac{1}{2}\left(v-  y\cdot \nabla_{\mathbb{R}^{k}}v\right)+ C e^{-\eta |\tau|^{1/2}} \,
\end{align}
for $\tau\ll0$ and $|y|\leq \sqrt{(2-\delta)|\tau|}$, where $C=C(\delta)<\infty$. Now, given any  $\vartheta\in S^{k-1}$ and $\omega\in S^{n-k}$, considering the function
\begin{equation}
w(\rho,\tau):=v(\rho\vartheta, \omega, \tau)^2-2(n-k)\, ,
\end{equation}
we infer that
\begin{equation}
   w_\tau\leq w-\tfrac12 y\cdot \nabla_{\mathbb{R}^{k}}w + Ce^{-\eta|\tau|^{1/2}}\, .
\end{equation}
Hence, for every $\rho_{0}>0$ we have
\begin{equation}
    \frac{d}{d\tau}(e^{-\tau}w(\rho_{0} e^{\frac{\tau}{2}}, \tau))\leq Ce^{-\tau-\eta|\tau|^{1/2}}
\end{equation}
for $\tau\ll 0$. Integrating this inequality, we obtain for every $\lambda\in (0, 1]$ that
\begin{equation}\label{ODEv}
    w(\rho, \tau)\leq \lambda^{-2}w(\lambda \rho, \tau+2\log\lambda)+o(|\tau|^{-1}).
\end{equation}
On the other hand, by  sharp asymptotics \eqref{inw_quad22} in parabolic region, given any $A<\infty$, the inequality
 \begin{equation}\label{limitv}
        w(\rho, \tau)\leq |\tau|^{-1}(n-k)({2k-\rho^2})+o(|\tau|^{-1})
    \end{equation}
    holds for $\rho\leq A$.
Thus, for $\rho\geq A$  we take $\lambda = A\rho^{-1}$ in \eqref{ODEv}, $\rho  = A$ and use \eqref{limitv} we obtain
\begin{equation}
    w(\rho,\tau)\leq -\frac{(n-k)(1-2kA^{-2})\rho^2}{|\tau|+2 \log(\rho/A)}+o(|\tau|^{-1}).
\end{equation}
This implies the upper bound \eqref{upper_bound} and completes the proof of sharp asymptotics in intermediate region.

Then we prove the sharp asymptotics in tip region. For any directional vector $\vartheta\in S^{k-1}\cap (\mathbb{R}^{k}\times \{0\})$, we let $p_{t}\in M_t$ be the point that maximizes $\langle p, \vartheta\rangle$ among all $p\in M_t$. After rotating coordinates we can assume without loss of generality that $\vartheta=(1, 0,\dots, 0)\in S^{k-1}$. Consider the distance of $p_t$  from the origin, namely
\begin{equation}
d(t):=|p_t|\, .
\end{equation}
Using the above sharp asympototics in  intermediate region and convexity, we see that
\begin{equation}\label{diameter_asymptotics}
  d(t)=\sqrt{2|t| \log|t|}(1+o(1)).
\end{equation}
Moreover, by Hamilton's Harnack inequality \cite{Hamilton_Harnack} we have
\begin{equation}
    \frac{d}{dt}H(p_t)\geq 0.
\end{equation}
Together with the mean curvature flow equation this yields
\begin{equation}\label{tip curvature}
  \frac{H(p_t)}{\sqrt{|t|^{-1}\log|t|}}=\frac{1}{\sqrt{2}}+o(1).
\end{equation}
Next, for any unit vector $\varphi$ on the sphere $ S^{k-2}$ in $\mathbb{R}^{k-1}$ spanned by $e_{2},\dots, e_{k}$, which are standard basis of $x_{2}, \dots, x_{k}$ coordinates, we consider the  point $q_{t}(\varphi) \in K_t$  that attains
\begin{equation}
 \max_{q\in M_t}\langle q, \varphi\rangle.
\end{equation}
By convexity, $K_t$  contains a $k$ dimensional   convex cone  $\mathcal{C}_{k}^t$   in $\mathbb{R}^{k}$ whose boundary connects tip point $p_t$ and every $q_{t}(\varphi)$ with $\varphi \in S^{k-2}$.

By the global convergence theorem \cite[Theorem 1.12]{HaslhoferKleiner_meanconvex} and \eqref{diameter_asymptotics}, the sequence of flows shifted by $p_{s_{i}}$ and parabolically rescaled by $\lambda(s_{i})=\sqrt{|s_{i}|^{-1}\log |s_{i}|}$
\begin{equation}
\widetilde{K}^{s_i}_t:=\lambda(s_i)\cdot(M_{s_i+\lambda(s_i)^{-2}t}-p_{s_i})
\end{equation}
 converges subsequentially to a noncompact ancient noncollapsed limit flow  $K^\infty_t$. Using  sharp asymptotics in intermediate region and construction of $\mathcal{C}_{k}^{t}$ in above,  the flow $K^\infty_t$ also  contains the following   $k$ dimensional $O(k-1)$ symmetric noncompact  convex cone
 \begin{equation}\label{cone C infty}
 \mathcal{C}^{\infty}_{k}= \{(x, 0)\in \mathbb{R}^{k}\times\mathbb{R}^{n-k+1}: \{x_{1}^2\geq x^2_{2}+,\dots+x^2_{k}, x_{1}\leq 0\}.
 \end{equation}
By \cite[Thm 1.4]{DH_blowdown}, \eqref{bubble-sheet_tangent_intro}, compactness and uniform $(k+1)$-convexity of  $K^\infty_t$ from uniform $(k+1)$-convexity of  $M_{t}$, $K^\infty_t\subset \{x_1\leq 0\}$   must split  at least  one line $P^{1}$ in the space spanned by $e_{2},\dots e_{k}$.  Let us write $K^\infty_t=P^{1}\times K^{\infty, 1}_t$. By construction, $K^{\infty, 1}_t$ contains a ray along $-e_{1}$ direction and is still a noncompact ancient noncollapsed  and contains a $(k-1)$ dimensional $O(k-2)$ symmetric convex cone. Then again by \cite[Thm 1.4]{DH_blowdown} and uniformly $k$-convexity of $K^{\infty, 1}_t$, the flow $K^{\infty, 1}_t$ must split a line.  By  construction, we can inductively apply \cite[Thm 1.4]{DH_blowdown} to conclude that   $M^{\infty}_t=\partial K^\infty_t=P^{k-1}\times N_{t}$, where $P^{k-1}$ is a $(k-1)$ dimensional subspace  spanned by $e_2, \dots, e_{k}$ and $N_{t}$ is two-convex noncompact ancient noncollapsed flow containing $-e_{1}$ axis. By \cite{BC2} and \eqref{tip curvature}, $N_{t}$ must be the unique  $(n+1-k)$-dimensional round  bowl soliton  translating  along   $-e_1$ direction with speed $1/\sqrt{2}$. Finally, by \cite{BC2} round bowl soliton is the unique uniformly two-convex noncompact ancient noncollapsed solution, the subsequential convergence entails full convergence. This concludes the proof of the Proposition \ref{Rk=k asymptotics}.
 \end{proof}
Now, we give the proof of $\mathrm{O}(n-k+1)$ symmetry of the $k$-ovals in $\mathbb{R}^{n+1}$ in Theorem \ref{compact_symmetry_oval}.
\begin{proof}[Proof of Theorem \ref{compact_symmetry_oval}]
The sharp asymptotics and compactness in Theorem \ref{compact_symmetry_oval} have been verified in above Proposition \ref{Rk=k asymptotics}. Namely $k$-ovals in $\mathbb{R}^{n+1}$ are compact and  have the same unique sharp asymptotics as the $\mathrm{SO}(k)\times \mathrm{SO}(n-k+1)$ symmetric ovals  \cite[Thm 1.4]{DH_ovals} with $\mathbb{R}^{k}\times S^{n-k}(\sqrt{2(n-k)|t|})$ as tangent flow at $-\infty$.

Then, we similarly as in proof of  \cite[Prop 6.7]{DH_hearing_shape} to show the $\mathrm{O}(n-k+1)$ symmetry. Let $\varepsilon <\varepsilon_{1}<\varepsilon_{0}$ be the constants from   Theorem \ref{Symmetry improvement I} (cylindrical improvement theorem) and   Theorem \ref{cap improvement theorem} (cap improvement theorem). Then, for any $0<\eps'\ll \varepsilon$ small enough, by the sharp asymptotics in Proposition \ref{Rk=k asymptotics}, there exists $t_\ast>-\infty$, such that for $t\leq t_\ast$ every $(p,t)\in\mathcal{M}$ is $\eps'$-close either to a cylinder $\mathbb{R}^{k}\times S^{n-k}$ or to a piece of $\mathbb{R}^{k-1}$ times $(n-k+1)$ dimensional round bowl. For otherwise we blow up the flow $\mathcal{M}$ along a sequence of points $(p_{i}, t_{i})$ by mean curvature rescaling at this sequence of points. Depending on $H(p_i,t_i)\textrm{dist}(p_i,T_{t_i})$ the sequence of rescaled flows either converges to infinity or to a finite constant, where $T_{t_{i}}$ is the tip set of $M_{t_{i}}$, and by convexity and tip asymptotics in Proposition \ref{Rk=k asymptotics} the limit is either a cylinder $\mathbb{R}^{k}\times S^{n-k}$ or  a piece of  $\mathbb{R}^{k-1}$ times $(n-k+1)$ dimensional round bowl respectively, which gives a contradiction.
In particular, by our choice of constants any such point $(p, t)$ is $(\eps, n-k)$-symmetric. Hence, by cylindrical improvement and cap improvement for $t\leq t_\ast$ every $(p,t)\in\mathcal{M}$  is $(\eps/2, n-k)$-symmetric. Iterating this, we infer that given any positive integer $j$, for $t\leq t_\ast$ every $(p,t)\in\mathcal{M}$  is $(\eps/2^j, n-k)$-symmetric. Since $j$ is arbitrary, this implies that $M_t$ is $\mathrm{O}(n-k+1)$-symmetric for $t\leq t_\ast$. Finally, by uniqueness of closed smooth solutions of the mean curvature flow the $\mathrm{O}(n-k+1)$-symmetry is preserved forwards in time. This concludes the proof of the $\mathrm{O}(n-k+1)$ symmetry of $k$-ovals in $\mathbb{R}^{n+1}$ in Theorem \ref{compact_symmetry_oval}.
\end{proof}
In the end of this section,
 we give the proof of Corollary \ref{symmetr classification}.
\begin{proof}[Proof of Corollary \ref{symmetr classification}]
Let $M_{t}$ be any ancient noncollapsed  with    the cylindrical flow $\mathbb{R}^{k}\times S^{n-k}(\sqrt{2(n-k)|t|})$ as its tangent flow at $-\infty$ for some $1\leq k\leq n-1$, and we assume $M_{t}$ is not cylindrical. By the  $\mathrm{SO}(k)$ symmetry  on $\mathbb{R}^{k}$, we have the corresponding cylindrical matrix $Q$ has $\textrm{rk}(Q)=0$ or $\textrm{rk}(Q)=k$. If $\textrm{rk}(Q)=0$, this is equivalent to that the flow has dominant unstable mode by discussion in beginning of Section \ref{sec_quant_thm}. Hence there is a unique nonvanishing fine cylindrical vector  associated to the flow by \cite[Thm 6.4]{DH_blowdown}, which contradicts the $\mathrm{SO}(k)$ symmetry on $\mathbb{R}^{k}$. If $\textrm{rk}(Q)=k$, by Theorem \ref{compact_symmetry_oval} and $\mathrm{SO}(k)$ symmetry  on $\mathbb{R}^{k}$ we know that $M_{t}$ is
$\mathrm{SO}(k)\times
\mathrm{SO}(n-k+1)$ symmetric and has  $\mathbb{R}^{k}\times S^{n-k}(\sqrt{2(n-k)|t|})$ as its tangent flow at $-\infty$. By \cite[Thm 1.1]{DH_ovals}, up to time shift and parabolically dilation, $M_{t}$ is the unique $
\mathrm{O}(k)\times
\mathrm{O}(n-k+1)$ symmetric oval, which has $\mathbb{R}^{k}\times S^{n-k}(\sqrt{2(n-k)|t|})$ as its tangent flow at $-\infty$ and was constructed by White \cite{White_nature} and Haslhofer-Hershkovits \cite{HaslhoferHershkovits_ancient}. This completes the proof of the corollary.
\end{proof}

\section{Classification of fully-degenerate noncollapsed solutions}\label{rk0section}
In this section, we will give a proof of  Theorem \ref{Rk=0} (fully-degenerate case). Recall that  $\textrm{rk}(Q)=0$  is equivalent to that the unstable mode is dominant. To prove this result, we first introduce the definition of cylindrical scale which will be used in later proof.
\begin{definition}[$\eps$-cylindrical with scale $r$]\label{eps_cyl}
We say that $\mathcal M$ is \emph{$\varepsilon$-cylindrical around $X$ at scale $r$}, if the flow $\mathcal{M}_{X,r}$, which is obtained from $\mathcal M$ by translating $X$ to the space-time origin and parabolically rescaling by $1/r$,  is ($C^{\lfloor1/\varepsilon \rfloor}$ sense) $\varepsilon$-close in $B(0,1/\varepsilon)\times (-2,-1]$ to the evolution of a round shrinking cylinder $\mathbb{R}^k\times S^{n-k}(\sqrt{2(n-k)|t|})$ with axes through the origin.
\end{definition}
We fix a small enough parameter $\varepsilon>0$ quantifying the quality of the cylinders. Given $X=(x,t)\in\mathcal{M}$, we analyze the solution around $X$ at the diadic scales $r_j=2^j$, where $j\in \mathbb{Z}$. Using Huisken's monotonicity formula \cite{Huisken_monotonicity} and quantitative differentiation (see e.g. \cite{CHN_stratification}), for every $X\in \mathcal{M}$, we can find an integer $J(X)\in\mathbb{Z}$ such that
\begin{equation}\label{eq_thm_quant1}
\textrm{$\mathcal M$ is not $\varepsilon$-cylindrical around $X$ at scale $r_j$ for all $j<J(X)$},
\end{equation}
and
\begin{equation}\label{eq_thm_quant2}
\textrm{$\mathcal M$ is $\tfrac{\varepsilon}{2}$-cylindrical around $X$ at scale $r_j$ for all $j\geq J(X)+N$}.
\end{equation}
\begin{definition}[cylindrical scale]\label{def_Z(X)}
The \emph{cylindrical scale} of $X\in\mathcal{M}$ is defined by
\begin{equation}
Z(X)=2^{J(X)}.
\end{equation}
\end{definition}
Then, we start proving Theorem \ref{Rk=0} (fully-degenerate case).
\begin{proof}[Proof of Theorem \ref{Rk=0}]
We assume throughout the proof that the flow $\mathcal{M}$ is not a round shrinking $\mathbb{R}^k\times S^{n-k}$. As we have discussed at the  beginning of Section \ref{sec_quant_thm} our assumption $\mathrm{rk}(Q)=0$ is then equivalent to the assumption that in Proposition \ref{mz.ode.fine.bubble-sheet} (Merle-Zaag alternative) the unstable mode is dominant.
 Then by \cite[Thm 6.4]{DH_blowdown}(fine cylindrical theorem) there is a universal nonvanishing fine cylindrical vector $(a_{1},\dots, a_{k})\not=0$ associated to our flow $\mathcal{M}$, such that for any space-time point $X$ after suitable recentering in the $x_{k+1},\dots x_{n+1}$-subspace the profile function $u^X$ of the renormalized flow $\bar{M}_\tau^X$ centered at $X$ satisfies
\begin{equation}
u^X = e^{\tau/2}(a_1y_1+\dots, a_k y_k) + o(e^{\tau/2})
\end{equation}
for all $\tau\leq \tau_\ast(Z(X))$, depending only on an upper bound for the cylindrical scale $Z(X)$ as defined in Definition \ref{def_Z(X)}.\\

We will prove the Theorem \ref{Rk=0} by induction on the number of $\mathbb{R}$ factors in its tangent flow at $-\infty$. By the classification result in \cite[Thm 8.3]{CHHW} or \cite{BC2}, Theorem \ref{Rk=0} holds for all ancient noncollapsed solutions with neck  $\mathbb{R}\times {S}^{n-1}$ as tangent flow at $-\infty$ and dominant unstable mode. Then, we make the induction assumption that for all ancient noncollapsed solutions with $\mathbb{R}^{m}\times S^{n-m}$ as tangent flow at $-\infty$ and with dominant unstable mode, where $1\leq m\leq k-1$, Theorem \ref{Rk=0} holds and the solutions are $\mathbb{R}^{m-1}$ times $(n-m+1)$ dimensional round bowl whose translating direction is orthogonal to $\mathbb{R}^{m-1}$ factor. We aim to show that Theorem \ref{Rk=0} holds in the case $m=k$ to conclude the proof. Let $\mathcal{M}=\{M_t\}$ be any  ancient noncollapsed mean curvature flow in $\mathbb{R}^{n+1}$ whose tangent flow at $-\infty$ is given by $\mathbb{R}^{k}\times {S}^{n-k}$  and for which in the Merle-Zaag alternative the unstable mode is dominant. We first claim:
\begin{claim}\label{unstable noncompact max splitting}
  Under the induction assumption, all noncompact ancient noncollapsed solutions with dominant unstable mode and  $\mathbb{R}^{k}\times {S}^{n-k}$ as tangent flow at $-\infty$ must be $\mathbb{R}^{k-1}$ times $(n-k+1)$ dimensional round bowl soliton whose translating direction is orthogonal to $\mathbb{R}^{k-1}$ factor.
\end{claim}
\begin{proof}[Proof of Claim \ref{unstable noncompact max splitting}]
We first make the following assertion below.  Claim \ref{unstable noncompact max splitting} follows immediately from the assertion and the induction assumption.

\textbf{Assertion:} Under the induction assumption, all noncompact ancient noncollapsed solutions with   $\mathbb{R}^{k}\times{S}^{n-k}$ as tangent flow at $-\infty$ and dominant unstable mode must split off a line.

\noindent Suppose towards a contradiction that the assertion does not hold. Namely there is a noncompact  ancient noncollapsed solution $M_{t}=\partial K_{t}$ with $\mathbb{R}^{k}\times {S}^{n-k}$ as tangent flow at $-\infty$ and with dominant unstable mode, but it does not split off a line. Suppose that its  blowdown $\check{K}=\lim_{\lambda \to 0}\lambda K_{t_{0}}$ is a $\ell$ dimensional strictly convex cone, where $1\leq \ell\leq k$. Choosing suitable coordinates, we may assume that $\check{K}\setminus\{0\}$ is contained in $(\mathbb{R}^{\ell}\times \{0 \})\cap \{x_1> 0\}$ and contains the positive $x_1$-axis.
Let $P$ be the $x_1x_2$-plane contained in $\mathbb{R}^{\ell}$. Observe that $\partial \check K\cap P$  consists of two rays $R^\pm$ with unit directional vectors $e_{\pm}$ satisfying $e_{+}\cdot e_{2}\geq 0$ and $e_{-}\cdot e_{2}\leq 0$, and  $R^{+}$ may possibly be identical to $R^{-}$.

By a space-time translation, we may also assume that $t_0=0$ and that $0\in M_0$ is the point in $M_0$ with smallest $x_1$-value. Now, for every $h>0$, let $x_h^{\pm}\in M_0\cap \{x_1=h\}\cap P$ be a point which maximizes/minimizes the value of $x_2$ in $K_0\cap \{x_1=h\}\cap P$. Then, we have
\begin{equation}\label{boundedZ}
\sup_h Z(x^{\pm}_h)\leq C<\infty.
\end{equation}
For otherwise if $Z(x^{\pm}_h,0)\to \infty$, arguing as in \cite[Claim 4.2]{DH_hearing_shape}, we can rescale the flow $\mathcal{M}$ by  $Z(x^{\pm}_h,0)^{-1}$ and shift $X_i^\pm$ to the origin  and pass through a subsequential limit flow.  By  \cite[Theorem 1.14]{HaslhoferKleiner_meanconvex}  the limit flow is an ancient noncollapsed flow with dominant unstable mode. The limit flow is not a cylinder by construction and definition of cylindrical scale. However,  $Z(x^{\pm}_h,0)^{-1}\to 0$ implies that the limit flow has vanishing fine cylindrical vector. This contradicts  the nonvanishing property of unique  fine cylindrical vector \cite[Thm 6.4]{DH_blowdown}.

Take $h_i\to \infty$ and consider the sequence $\mathcal{M}^{i,\pm}:=\mathcal{M}-(x_{h_i}^{\pm},0)$.
By definiton, we have $x_2\geq 0$ on $M_0^{i,-}\cap \{x_1=0\}$ and $x_2\leq 0$ on $M_0^{i,+}\cap \{x_1=0\}$.
By \eqref{boundedZ} and \cite[Theorem 1.14]{HaslhoferKleiner_meanconvex}, any subsequential limit $\mathcal{M}^{\infty,\pm}=\{ M_t^{\infty,\pm}\}$ is an ancient noncollapsed flow with $\mathbb{R}^k \times S^{n-k}$ as tangent flow at $-\infty$.
Moreover, since $\mathcal{M}^{i, \pm}$ is not rescaled and has dominant unstable mode, we see that $\mathcal{M}^{\infty,\pm}$ has  dominant unstable mode, with the same fine cylindrical vector $(a_1,\ldots , a_k)$ as our original flow $\mathcal{M}$. Namely,
\begin{equation}\label{a+-=}
    (a^{\infty, \pm}_{1}, \dots, a^{\infty, \pm}_{k})=(a_{1}, \dots, a_{k})\not=0.
\end{equation}
We observe that since $
{x_{h_i}^\pm }/  {\| x_{h_i}^\pm\|} \to e_{\pm} \in R^\pm$, the hypersurfaces
$M_0^{\infty,\pm}$  contain a line $L^{\pm}$ in direction $e_\pm$ respectively ($L_+$ and $L_{-}$ may both equal to $x_{1}$ axis when $\check{K}$ is a half line). Namely
\begin{equation}
    M_t^{\infty,\pm}=L^{\pm}\times \bar{N}^{\infty, \pm}_{t}.
\end{equation}
where $\bar{N}^{\infty \pm}_{t}$ is an ancient noncollapsed flow with $\mathbb{R}^{k-1} \times S^{n-k}$ as tangent flow at $-\infty$ and with dominant unstable mode.

By induction hypothesis and \eqref{a+-=}, $\bar{N}^{\infty \pm}_{t}$ is a $(k-2)$ dimensional subspace $P^{\pm, k-2}$ times a $(n-k+1)$ dimensional round translating bowl soliton  $\Sigma^{n-k+1}_{\pm, t}$. Therefore, $M_t^{\infty,\pm}$ is  a $(k-1)$ dimensional subspace $P^{\pm, k-1}$ times a $(n-k+1)$ dimensional round translating bowl soliton  $\Sigma^{n-k+1}_{\pm, t}$ with the same translating direction and fine cylindrical vector $(a_{1}, \dots, a_{k})$, which is orthogonal to the $P^{\pm, k-1}$ factor. This and the fact that $M^{\infty,\pm}_{0}$ has tip at space-time origin imply
\begin{align}\label{Minfty+=Minfty-}
  M_{t}^{\infty,+} = M_{t}^{\infty,-},
\end{align}
 $\Sigma^{n-k+1}_{\pm, t}=  \Sigma^{n-k+1}_{t}$ and $P^{\pm, k-1}= P^{k-1}$, where $\Sigma^{n-k+1}_{t}$ is  a $(n-k+1)$ dimensional round translating bowl soliton flow along $(a_{1},\dots, a_{k})$ direction  and  $P^{k-1}$ is a $(k-1)$ dimensional subspace. Moreover, $(a_1,...,a_k)$ must be perpendicular to both $L^{+}\subset P^{+, k-1}$ and $L^{-}\subset P^{-, k-1}$, which are lines in $x_1 x_2$ plane and may both equal to $x_{1}$ axis  when $\check{K}$ is a half line. This implies that $(a_1,...,a_k)$ must be perpendicular to the $x_1$ axis, i.e. $a_1=0$. Consequently, we have transversal intersection
\begin{align}\label{structure of x1 cross section}
   M_{t}^{\infty,+} \cap \{x_1=0\}=  M_{t}^{\infty,-}\cap \{x_1=0\} = P^{k-2}\times \Sigma^{n-k+1}_{t}.
\end{align}

On the other hand, the choice of $x_{h_i}^{\pm}$ implies $x_2 \geq 0$ on $M_{t}^{\infty,-} \cap \{x_1=0\}$ and  $x_2 \leq 0$ on $M_{t}^{\infty,-} \cap \{x_1=0\}$.
This and \eqref{Minfty+=Minfty-} imply that the $(n-1)$ dimensional strictly convex translating flow $M_{t}^{\infty,+} \cap \{x_1=0\}= M_{t}^{\infty,-}\cap \{x_1=0\}$ is contained in the $(n-1)$ dimensional subspace $\{x_1=x_2=0\}$, which is a contradiction to (\ref{structure of x1 cross section}). This completes the proof of the assertion and hence Claim \ref{unstable noncompact max splitting}.
\end{proof}
To finish the induction, we also need to rule out the compact solutions by the following claim.
\begin{claim}\label{unstabletononcompact}
  Under the induction assumption, for  all ancient noncollapsed solutions with dominant unstable mode and  $\mathbb{R}^{k}\times S^{n-k}$ as tangent flow at $-\infty$, if unstable mode above is dominant, then the solution is noncompact.
\end{claim}
\begin{proof}[Proof of Claim \ref{unstabletononcompact}]
Suppose towards a contraction $\mathcal{M}=\{M_t\}$ is a compact ancient noncollapsed mean curvature flow. Now, considering any sequence $t_i\to -\infty$, by compactness we can find points $p_{t_i}^\pm\in M_{t_i}$ such that
\begin{equation}\label{choice_min_max}
x_1(p_{t_i}^-)=\min_{p\in M_{t_i}} x_1(p)\, ,\qquad x_1(p_{t_i}^+)=\max_{p\in M_{t_i}} x_1(p)\, .
\end{equation}
Then, by the same discussion of \eqref{boundedZ} as before, we have
\begin{equation}\label{cyl_scale_bd}
\sup_i Z(p^{\pm}_i,t_i)< \infty.
\end{equation}
Then, we consider the following sequence of flows
\begin{equation}
    \mathcal{M}^{\pm,i}=\mathcal{M}-(p^{\pm}_i,t_i) \, ,
\end{equation}
which is obtained by shifting in space-time without rescaling. By \cite[Theorem 1.14]{HaslhoferKleiner_meanconvex} we can pass to subsequential limits $\mathcal{M}^\pm$, which are ancient noncollapsed flows that are weakly convex and smooth until they become extinct.  By \cite[Thm 1.14 (2)]{HaslhoferKleiner_meanconvex} we have $|x_1(p_{t_i}^\pm)| \to \infty$ and  by \eqref{cyl_scale_bd}, we infer that $\mathcal{M}^\pm$
are noncompact and have $\mathbb{R}^{k}\times S^{n-k}$ as tangent flow at $-\infty$. Moreover, since $\mathcal{M}^{\pm, i}$ are not rescaled and have dominant unstable mode, $\mathcal{M}^\pm$ have dominant unstable mode with the same fine cylindrical vector $(a_1, \dots,a_k)$ as $\mathcal{M}$.

By Claim \ref{unstable noncompact max splitting}, $\mathcal{M}^{+}$ and $\mathcal{M}^{-}$
are both $\mathbb{R}^{k-1}$ times $(n-k+1)$ dimensional round bowl soliton. Moreover, they have the same translating direction and fine cylindrical vector $(a_1, \dots,a_k)$, which is perpendicular to the $(k-1)$ splitted lines of $\mathcal{M}^{\pm}$. In addition, the choice of $p^{\pm}_{t_i}$ gives $x_1\leq 0$ on  $M_{0}^{+}$ and $x_1\geq 0$ on  $M_{0}^{-}$. This implies $\mathcal{M}^{\pm}$ split off $(k-1)$ lines in the $x_{2},\dots, x_{k}$ coordinates subspace, so $\mathcal{M}^{\pm}$ translate in both positive $x_{1}$ direction and negative $x_{1}$ direction. This is a contradiction and thus concludes the proof of the Claim \ref{unstabletononcompact}.
\end{proof}

By Claim \ref{unstable noncompact max splitting} and Claim \ref{unstabletononcompact} we proved the conclusion of the Theorem \ref{Rk=0} for the case where the number of $\mathbb{R}$ factor in its cylindrical tangent flow at $-\infty$ is $k$ and finish the induction. This completes the proof of Theorem \ref{Rk=0}.
\end{proof}

Finally,   we give the proof of Corollary \ref{blowdown in unstable mode}.
\begin{proof}[Proof of Corollary \ref{blowdown in unstable mode}]
By Theorem \ref{Rk=0}, we know that all strictly convex ancient noncollapsed  solutions are round translating bowl. In particular their  tangent flow at $-\infty$ is $\mathbb{R}\times S^{n-1}(\sqrt{2(n-1)|t|})$ and their  blowdown is a  half line which  is independent of time.
\end{proof}
\bigskip
\begin{appendix}
\section{Some lemmas for symmetry improvement theorem}
\begin{lemma}\label{Lin Alg Cylinder A}
  Let $\mathcal{K}=\{K_{\alpha}, 1\leq\alpha\leq \frac{(n-k+1)(n-k)}{2}\}$
  be  a  normalized set of rotation vector fields in the form of
  \begin{align}
    K_{\alpha}=SJ_{\alpha}S^{-1}(x-q),
  \end{align}
  where
  $\{J_{\alpha}, 1\leq\alpha\leq \frac{(n-k+1)(n-k)}{2}\}$ is an
  orthonormal basis of $so(n-k+1)\subset so(n+1)$ under the natural identification as in \eqref{J, Jbar} and $S\in \mathrm{O}(n+1), q\in\mathbb{R}^{n+1}$.

  Suppose that either one of the following cases happens:
  \begin{enumerate}
    \item on the cylinder $\mathbb{R}^k\times S^{n-k}\subset \mathbb{R}^{n+1}$, we have:
  \begin{itemize}
     \item $\left<K_{\alpha},\nu\right> = 0$ in $B_{g}(p, 1)$ for each $\alpha$,
     \item $\max_{\alpha}|K_{\alpha}| H \leq 5n$ in
      $B_{g}(p, 10n^{2})$,
  \end{itemize}
    \item on $\mathbb{R}^{k-1}\times \Sigma\subset\mathbb{R}^{n+1}$, where $\Sigma$ is the unique $(n-k+1)$ dimensional round bowl, we have:
  \begin{itemize}
     \item $\left<K_{\alpha},\nu\right> = 0$ in $B_{g}(p, 1)$ for each $\alpha$,
     \item $\max_{\alpha}|K_{\alpha}| H \leq 5n$ at $p$,
  \end{itemize}
  \end{enumerate}
  where $\nu$ is the normal vector, $g$ is the induced metric and
  $p$ is arbitrary point on the above  hypersurfaces.   Then $S,q$ can be chosen such that
  \begin{itemize}
    \item $S\in
\mathrm{O}(n-k+1)\subset \mathrm{O}(n+1)$,
    \item $q=0$.
  \end{itemize}
  In particular, for each orthonormal basis $\{J'_{\alpha}, 1\leq\alpha\leq \frac{(n-k+1)(n-k)}{2}\}$ of $so(n-k+1)$ there is a basis transform matrix
  $(\omega_{\alpha\beta})\in \mathrm{O}(\frac{(n-k+1)(n-k)}{2})$ such that
  $K_{\alpha}  = \sum_{\beta=1}^{\frac{(n-k+1)(n-k)}{2}}\omega_{\alpha\beta} J'_{\beta}x$.
\end{lemma}
\begin{proof}\textbf{Case (1)}:  We use the coordinate $(x_1,...x_{n+1})$ in $\mathbb{R}^{n+1}$ and let the cylinder $\mathbb{R}^k\times S^{n-k}$  be
 represented by $\{x_{k+1}^2+...+x_{n+1}^2=1\}$.

 Now $\langle K_{\alpha},\nu\rangle =0$ is equivalent to:
 \begin{align}\label{Lin Alg 1}
   \nu^{T}SJ_{\alpha}S^{-1}(x-q)=0.
 \end{align}

 Without loss of generality, we may assume that each $J_{\alpha}$ only has two nonzero elements $(J_{\alpha})_{ij} = 1$, $(J_{\alpha})_{ji} = -1$ for some $k+1\leq j<i\leq n+1$.
Then, we choose $q$ such that
 \begin{align}\label{LinAlg Condition for q}
   q\perp \bigcap_{\alpha}\ker (SJ_{\alpha}S^{-1}).
 \end{align}

 Note that $SJ_{\alpha}S^{-1}$ is anti-symmetric and $\nu^{T}=(0,...,0, x_{k+1},...,x_{n+1})$
 on $\mathbb{R}^k\times S^{n-k}$ , (\ref{Lin Alg 1}) is equivalent to
 \begin{align}\label{LinAlg2}
   \sum_{k+1\leq i\leq n+1} x_i \left(\sum_{l=1}^{k}(SJ_{\alpha}S^{-1})_{il}x_l - (SJ_{\alpha}S^{-1}q)_{i}\right) =  0.
 \end{align}
 Since this holds on an open set of the cylinder, we obtain that
 \begin{align}
   (SJ_{\alpha}S^{-1})_{il} =&0, \text{ \  } i=k+1,...,n+1,\ l=1,...,k \label{LinAlg3}\\
   (SJ_{\alpha}S^{-1}q)_{i} =&0, \text{ \  } i=k+1,...,n+1  \label{LinAlg4}
 \end{align}
 hold for each $1\leq\alpha \leq\frac{(n-k+1)(n-k)}{2} $.
 In other words, for each $\alpha$ we have
\begin{align}\label{block decompostion}
       SJ_{\alpha}S^{-1} =\begin{bmatrix}
    B_{\alpha} & 0\\
     0 & A_{\alpha}
    \end{bmatrix}
\end{align}
for some anti-symmetric $(n-k+1)\times (n-k+1)$ matrix $A_{\alpha} $ and  anti-symmetric $ k\times k$ matrix $B_{\alpha}$. Since rank($ SJ_{\alpha}S^{-1}$) = 2, we must have either $A_{\alpha} = 0$ or $B_{\alpha} = 0$. We claim that $A_{\alpha}\not =0$ for all $1\leq\alpha \leq\frac{(n-k+1)(n-k)}{2} $. Indeed, if  $A_{\alpha} = 0$ for some $\alpha$, then rank($B_{\alpha}$) = 2 and we have
\begin{align}
    \sup\limits_{B_{g}(p,10n^2)}|K_{\alpha}| &\geq  \sup\limits_{B_{g}(p,10n^2)\cap \mathbb{R}^{k}} \frac{1}{2} |K_{\alpha}(p+x)-K_{\alpha}(p-x)|\\\nonumber
    &= \sup\limits_{B^k(0,10n^2)}|B_{\alpha}x|=10n^2,
\end{align}
where $B^k(0,10n^2)$ denotes the ball in $\mathbb{R}^k$ with center at the origin and  radius $10n^2$, and the last equality uses the fact that $B_{\alpha}$ has rank $2$ and  norm  upper bound $1$. However, this contradicts our assumption on $K_{\alpha}$.
Hence, we have $B_{\alpha} =0$ for all $1\leq\alpha \leq\frac{(n-k+1)(n-k)}{2} $. This gives
\begin{align}
   (SJ_{\alpha})_{il} =&0, \text{ \  } i=1, \dots, k
\end{align}
for all $1\leq\alpha \leq\frac{(n-k+1)(n-k)}{2} $, which implies
\begin{align}
    S_{il} =&0, \text{ \  } i=1, \dots, k,  \quad  l= k+1, \dots,n+1.
\end{align}
Since $S\in \mathrm{O}(n+1)$, we have $S \in
\mathrm{O}(k)\times
\mathrm{O}(n-k+1) \subset \mathrm{O}(n+1)$. Namely
\begin{align}
       S  =\begin{bmatrix}
    S_1 & 0\\
     0 & S_2
    \end{bmatrix}
\end{align}
where $S_1 \in
\mathrm{O}(k)$ and $S_2\in
\mathrm{O}(n-k+1)$.

Hence,
\begin{align}
    SJ_{\alpha}S^{-1} =\begin{bmatrix}
    Id & 0\\
     0 & S_2
    \end{bmatrix}J_{\alpha}\begin{bmatrix}
    Id & 0\\
     0 &   S_2^{-1}
    \end{bmatrix},
\end{align}
which means $S$ can be chosen to be in $
\mathrm{O}(n-k+1)\subset
\mathrm{O}(n+1)$. Namely,
\begin{align}
       S  =\begin{bmatrix}
   Id & 0\\
     0 & S_2
    \end{bmatrix}
\end{align}
This, \eqref{LinAlg4} and \eqref{LinAlg Condition for q} imply that $q=0$.

Finally, $\{SJ_{\alpha}S^{-1}, 1\leq \alpha \leq \frac{(n-k+1)(n-k)}{2}\}$ forms an orthonormal basis of $so(n-k+1)$. The last claim follows immediately.

\textbf{Case (2):}  Arguing similarly as in \cite{Zhu2}, we obtain that
 \begin{align}\label{Lin Alg II.6}
      (SJ_{\alpha}S^{-1})_{il}=&0 \\
      (SJ_{\alpha}S^{-1}q)_i=&0
 \end{align}
 for all $k\leq i\leq n+1$ and $1\leq l \leq k$.

 Consequently,  we can argue similarly as in the case (1), the only difference is that
 the last row and the last column of $B_{\alpha}$ in the block decomposition (\ref{block decompostion})  vanishes identically. However, this is even stronger than what we concluded, so we can follow the argument exactly as in the case (1) to obtain the result.
\end{proof}

Then, we  have the following lemma.
\begin{lemma}\label{Lin Alg Cylinder B}
  Let  $\{J_{\alpha}, 1\leq\alpha\leq \frac{(n-k+1)(n-k)}{2}\}$ be an
  orthonormal basis of $so(n-k+1)\subset so(n+1)$, vector $b\perp \bigcap_{\alpha} ker(J_{\alpha})$ and $A\in so(n+1)$ such that
  $A\perp so(k)\oplus so(n-k+1)$.
 Suppose that on the cylinder $\mathbb{R}^k\times S^{n-k}\subset \mathbb{R}^{n+1}$  we have
  \begin{align}\label{[AJb]}
      \left<[A,J_{\alpha}]x+J_{\alpha}b,\nu\right> = 0
  \end{align}
in $B_{g}(p, 1)$  for each $\alpha$, where  $g$ is the induced metric on cylinder. Then $A=0$ and $b=0$.
\end{lemma}
\begin{proof}
Since the quadratic function in \eqref{[AJb]} vanishes on open domain of cylinder, we have
\begin{equation}
    [A, J_{\alpha}]=0\quad J_{\alpha}b=0,
\end{equation}
where $b\perp\bigcap_{\alpha} ker(J_{\alpha})$, $A\in so(n+1)$ such that $A\perp so(k)\oplus so(n-k+1)$ and  $1\leq\alpha\leq \frac{(n-k+1)(n-k)}{2}$.

Since $b\perp\bigcap_{\alpha} ker(J_{\alpha})$ and $J_{\alpha}b=0$ for $1\leq\alpha\leq \frac{(n-k+1)(n-k)}{2}$, we have $b=0$. By assumption, $J$ and $A$ can be written a the following form:
\begin{align}\label{[A,J]=0}
        J  =\begin{bmatrix}
    0 & 0\\
     0 &  \tilde{J}
    \end{bmatrix}\in so(n+1)
    \quad
     A  =\begin{bmatrix}
    0 & -\tilde{A}^{t}\\
     \tilde{A}^{t} & 0
    \end{bmatrix}\in so(n+1),
\end{align}
where $\tilde{J} \in so(n-k+1)$. Then by $[A, J]=0$, we obtain $\tilde{J}A^{t}=0$ holds for all $\tilde{J} \in so(n-k+1)$. This implies $A=0$.
\end{proof}
 \begin{lemma}\label{Vector Field closeness on genearalized cylinder_appendix}
  There exist constants $0 < \varepsilon_{c} \ll 1$ and $C>1$ depending only on dimension $n$
  with the following properties. Let $M$ be a hypersurface in $\mathbb{R}^{n+1}$
  which is $\varepsilon_c$-close (in $C^3$ norm) to a geodesic ball  of radius  $20n^{5/2}\sqrt{\frac{2}{n-k}}$ in
  $\mathbb{R}^k\times S^{n-k}(\sqrt{2(n-k)})$
     and
  $\bar{x}\in M $ be a point that is $\varepsilon_c$-close to  $\mathbb{R}^k\times S^{n-k}(\sqrt{2(n-k)})$.
  Suppose that $\varepsilon \leq \varepsilon_c$ and
  $\mathcal{K}^{(1)} = \{K^{(1)}_{\alpha}: 1 \leq \alpha \leq \frac{(n-k+1)(n-k)}{2}\}$,
  $\mathcal{K}^{(2)} = \{K^{(2)}_{\alpha}: 1 \leq \alpha \leq \frac{(n-k+1)(n-k)}{2}\}$, are two    normalized sets of rotation vector fields such that:
  \begin{itemize}
    \item $\max_{\alpha}|\left<K_{\alpha}^{(i)}, \nu\right>|H \leq \varepsilon$ in
    $B_g(\bar{x}, H^{-1}(\bar{x}))\subset M$,
    \item $\max_{\alpha}|K_{\alpha}^{(i)}|H \leq 5n$ in $B_{g}(\bar{x}, 10n^{5/2}H^{-1}(\bar{x}))\subset M$,
  \end{itemize}
  for $i = 1 , 2$, where $g$ denotes the induced metric on $M$ by embedding.
  Then for any $L > 1$, we have
  \begin{align}
    \inf\limits_{(\omega_{\alpha\beta})\in \mathrm{O}(\frac{(n-k+1)(n-k)}{2})}\sup\limits_{B_{LH(\bar{x})^{-1}}(\bar{x})}\max_{\alpha}
    |K_{\alpha}^{(1)}-\sum_{\beta = 1}^{\frac{(n-k+1)(n-k)}{2}}\omega_{\alpha\beta}K^{(2)}_{\beta}|H(\bar{x})
    \leq CL\varepsilon.
  \end{align}
\end{lemma}
\begin{proof}
Since the vector fields are affine functions, it suffices to prove the result for  $L=10n$. Throughout the proof below, the constant $C$ depends only on $n$.

Suppose towards a contradiction that the conclusion does not hold. Set $\bar{p}_0 = (0,..,0.\sqrt{2(n-k)}, 0,...,0)$, where the only nonzero element  $\sqrt{2(n-k)}$ sits in the $(k+1)$-th component. Then there exists a sequence of pointed hypersurfaces $(M_j, p_{j})$ that are $\frac{1}{j}$-close to a geodesic ball of radius $20n^{5/2}H_0^{-1}$ inside $\mathbb{R}^k\times S^{n-k}(\sqrt{2(n-k)})$ with $p_j\in M_{j}$ and $|p_j - \bar{p}_{0}| \leq \frac{1}{j}$,
   where $H_0 = \sqrt{\frac{n - k}{2}}$ is the mean curvature of $\mathbb{R}^k\times S^{n-k}(\sqrt{2(n-k)})$
   and $g_j$ is the metric on $M_j$ induced by embedding.
   Moreover, for $i = 1,2$ there are two sequences of  normalized sets of rotation vector fields
   $\mathcal{K}^{(i,j)} = \{K^{(i,j)}_{\alpha}: 1\leq\alpha\leq \frac{(n-k+1)(n-k)}{2}\}$  and $\varepsilon_j\leq\frac{1}{j}$ such that
   \begin{itemize}
     \item $\max_{\alpha}|\left<K^{(i,j)}_{\alpha}, \nu\right>|H\leq\varepsilon_j$
      in $B_{g_j}(p_j, H(p_j)^{-1})$,
     \item $\max_{\alpha}|K_{\alpha}^{(i,j)}|H\leq 5n$ in
      $B_{g_j}(p_j, 10n^{5/2}H(p_j)^{-1})$,
   \end{itemize}
   but
\begin{equation}\label{contra}
     \inf\limits_{(\omega_{\alpha\beta})\in \mathrm{O}(\frac{(n-k+1)(n-k)}{2})}\sup\limits_{B_{10nH(\bar{x})^{-1}}(\bar{x})}\max_{\alpha}
    |K_{\alpha}^{(1,j)}-\sum_{\beta = 1}^{\frac{(n-k+1)(n-k)}{2}}\omega_{\alpha\beta}K^{(2,j)}_{\beta}|H(\bar{x})
    \geq j\varepsilon_j.
\end{equation}
   Hence, $M_j$ converges to $M_{\infty} = B_{g_{\infty}}(p_{\infty}, 20n^{5/2}H_0^{-1})\subset \mathbb{R}^k\times S^{n-k}(\sqrt{2(n-k)})$
   and $H(p_j)\rightarrow \sqrt{\frac{n-k}{2}}$, where $p_{\infty}=\bar{p}_{0}=(0,...,0, \sqrt{2(n-k)}, 0,\dots, 0)$
   and $g_{\infty}$ is the induced metric on
   $\mathbb{R}^k\times S^{n-k}(\sqrt{2(n-k)})$ .

   Suppose that for each $i=1,2$ and $j\geq 1$,
   $K^{(i,j)}_{\alpha}(x) = S_{(i,j)}J^{(i,j)}_{\alpha}S_{(i,j)}^{-1}(x-b_{(i,j)})$
   where $S_{(i,j)}\in \mathrm{O}(n+1)$ and $b_{(i,j)}\in\mathbb{R}^{n+1}$,
   $\{J^{(i,j)}_{\alpha}: 1\leq\alpha \leq\frac{(n-k+1)(n-k)}{2}\}$ is
   an orthonormal basis of
   $so(n-k+1)\subset so(n+1)$ under the natural identification as in \eqref{J, Jbar}. Without loss of generality we may assume that
   $\displaystyle b_{(i,j)}\perp \bigcap_{\alpha}\ker(S_{(i,j)}J^{(i,j)}_{\alpha}S_{(i,j)}^{-1})$.
This and the expression of $K^{(i,j)}_{\alpha}(x)$ imply
   \begin{align}
     |b_{(i,j)}| \leq C\sum_{\alpha}|S_{(i,j)}J^{(i,j)}_{\alpha}S_{(i,j)}^{-1} x-K^{(i,j)}_{\alpha}(x)| \leq C(n).
   \end{align}
   Therefore we can pass to a subsequence such that $S_{(i,j)}\rightarrow S_{(i,\infty)}$, $J_{\alpha}^{(i,j)}\rightarrow J_{\alpha}^{(i,\infty)}$
   and $b_{(i,j)}\rightarrow b_{(i,\infty)}$ for each $i,\alpha$.
   Consequently for $i=1,2$ as $j\to \infty$, $K^{(i,j)}\rightarrow K^{(i,\infty)}_{\alpha} = S_{(i,\infty)}J_{\alpha}^{(i,\infty)}S_{(i,\infty)}^{-1}(x-b_{(i,\infty)})$.

   The convergence implies that
   \begin{itemize}
     \item $\left<K^{(i,\infty)}_{\alpha},\nu\right> = 0$ in $B_{g_{\infty}}(p_{\infty,}, H_0^{-1})$ for each $\alpha$,
     \item $\max_{\alpha}|K^{(i,\infty)}_{\alpha}| H \leq 5n$ in,
      $B_{g_{\infty}}(p_{\infty,}, 10n^{5/2}H_0^{-1})$.
   \end{itemize}
   Let's fix an orthonormal basis $\{J_{\alpha}: 1\leq\alpha \leq\frac{(n-k+1)(n-k)}{2}\}$  of $so(n-k+1)\subset so(n+1)$ under natural identification as in \eqref{J, Jbar}. By Lemma \ref{Lin Alg Cylinder A}   for each $ 1\leq\alpha \leq\frac{(n-k+1)(n-k)}{2}$ we have
    \begin{flalign}
      K_{\alpha}^{(i,\infty)}(x) = \sum_{\beta=1}^{\frac{(n-k+1)(n-k)}{2}}\omega^{(i,\infty)}_{\alpha\beta} J_{\beta}x
    \end{flalign}
    for some $(\omega_{\alpha\beta}^{(i,\infty)})\in \mathrm{O}(\frac{(n-k+1)(n-k)}{2})$.
    In particular, by
    \begin{equation}
        b_{(i,\infty)}\perp\bigcap_{\alpha}\ker(S_{(i,\infty)}J^{(i,\infty)}_{\alpha}S^{-1}_{(i,\infty)})
    \end{equation}
    and  Lemma \ref{Lin Alg Cylinder A} we have
    \begin{align}\label{sjs-1=omega J}
      S_{(i,\infty)}J^{(i,\infty)}_{\alpha}S_{(i,\infty)}^{-1} = \sum_{\beta=1}^{\frac{(n-k+1)(n-k)}{2}}\omega^{(i,\infty)}_{\alpha\beta} J_{\beta}
    \end{align}
    and $b_{(i,\infty)}=0$ for $i=1,2$.

    Now because $\{J^{(i,\infty)}_{\alpha}, 1\leq\alpha\leq \frac{(n-k+1)(n-k)}{2}\}$ is an orthonormal basis of $so(n-k+1)\subset so(n+1)$, by  Lemma \ref{Lin Alg Cylinder A} we can find $(\eta^{(i,j)}_{\alpha\beta})\in \mathrm{O}(\frac{(n-k+1)(n-k)}{2})$ such that
    $J_{\alpha}^{(i,j)}=\sum_{\beta}\eta^{(i,j)}_{\alpha\beta}J_{\beta}^{(i,\infty)}$. Then we have
    \begin{align}\label{sjs-1=omega J 2}
      S_{(i,\infty)}J^{(i,j)}_{\alpha}S_{(i,\infty)}^{-1}
      =&\sum_{\beta}\eta^{(i,j)}_{\alpha\beta}S_{(i,\infty)}J_{\beta}^{(i,\infty)}S_{(i,\infty)}^{-1}  \\\nonumber
      =& \sum_{\beta,\gamma}\eta^{(i,j)}_{\alpha\beta}\omega^{(i.\infty)}_{\beta\gamma}J_{\gamma}.
    \end{align}
     Now (\ref{sjs-1=omega J 2}) means that $\{S_{(i,\infty)}J^{(i,j)}_{\alpha}S_{(i,\infty)}^{-1}, 1\leq\alpha\leq \frac{(n-k+1)(n-k)}{2}\}$ is an orthonormal basis of $so(n-k+1)\subset so(n+1)$.

    Without loss of generality, we may assume that
    $S_{(1,\infty)}=S_{(2,\infty)} =Id$ and $\omega^{(1,\infty)}_{\alpha\beta} = \delta_{\alpha\beta}$, otherwise we replace $S_{(i,j)}$
    by $S_{(i,j)}S_{(i,\infty)}^{-1}$,  $J_{\alpha}^{(i,j)}$ by $S_{(i,\infty)}J_{\alpha}^{(i,j)}S_{(i,\infty)}^{-1}$ for $i=1,2$
    and $J_{\alpha}$ by
    $\sum_{\beta=1}^{\frac{(n-k+1)(n-k)}{2}}\omega^{(1,\infty)}_{\alpha\beta} J_{\beta}$.

    Therefore, $S_{(1,j)}^{-1}S_{(2,j)}$ is close to Id for large $j$. Then we can find $S_j\in \mathrm{O}(n+1)$
    and $A_j\in so(n+1)$ such that
    \begin{itemize}
     \item $S_{(1,j)}^{-1}S_{(2,j)} = \exp(A_j)S_j $,
      \item $S_j$ preserves the direct sum decomposition $\mathbb{R}^{k}\oplus\mathbb{R}^{n-k+1}$,
      \item $A_j\perp so(k)\oplus so(n-k+1)$,
      \item $S_j\rightarrow Id$ and $A_j\rightarrow 0$.
    \end{itemize}

   Since $S_j$ preserves the direct sum decomposition $\mathbb{R}^{k}\oplus\mathbb{R}^{n-k+1}$,
   by  Lemma \ref{Lin Alg Cylinder A} we can  find basis transform matrix $({\omega}^{(j)}_{\alpha\beta})\in \mathrm{O}(\frac{(n-k+1)(n-k)}{2})$ such that
   \begin{align}
     S_j^{-1} J^{(1,j)}_{\alpha}S_{j} = \sum_{\beta=1}^{\frac{(n-k+1)(n-k)}{2}}{\omega}^{(j)}_{\alpha\beta} J_{\beta}^{(2,j)}
   \end{align}
   for every $j$ and $1\leq\alpha\leq \frac{(n-k+1)(n-k)}{2}$. Equivalently, we have
   \begin{align}
      J^{(1,j)}_{\alpha} = \sum_{\beta=1}^{\frac{(n-k+1)(n-k)}{2}}{\omega}^{(j)}_{\alpha\beta} S_{j}J_{\beta}^{(2,j)}S_j^{-1}.
   \end{align}
   Then, we construct the following vector fields.
   \begin{align}
     &\sum_{\beta=1}^{\frac{(n-k+1)(n-k)}{2}} \omega^{(j)}_{\alpha\beta}K_{\beta}^{(2,j)}(x)\\\nonumber
     =& \sum_{\beta} \omega^{(j)}_{\alpha\beta}S_{(2,j)}J^{(2,j)}_{\beta}S_{(2,j)}^{-1}(x-b_{(2,j)})\\\nonumber
     =& \sum_{\beta}\omega^{(j)}_{\alpha\beta} S_{(1,j)}\exp(A_j) S_j J_{\beta}^{(2,j)}S_j^{-1}\exp(-A_j)S_{(1,j)}^{-1} (x-b_{(2,j)}) \\\nonumber
     =&  S_{(1,j)}\exp(A_j)J^{(1,j)}_{\alpha}\exp(-A_j)S_{(1,j)}^{-1}(x-b_{(2,j)}).
   \end{align}

Now, for each $1\leq\alpha\leq \frac{(n-k+1)(n-k)}{2}$ we define
     \begin{align}
       W_{\alpha}^j  = \frac{K_{\alpha}^{(1,j)} -\sum_{\beta} \omega^{(j)}_{\alpha\beta}K_{\beta}^{(2,j)} }{\sup\limits_{B_{10nH(p_j)^{-1}}(p_j)}\max_{\alpha}|K_{\alpha}^{(1,j)} -\sum_{\beta} \omega^{(j)}_{\alpha\beta}K_{\beta}^{(2,j)} |}.
     \end{align}
Moreover,  we have the following representation
\begin{align}
    W^{j}_{\alpha}= Q_{j}[P^j_{\alpha}(x-b_{(2,j)}) + c^j_{\alpha}],
\end{align}
where
     \begin{itemize}
       \item $P_{\alpha}^j = S_{(1,j)}[J^{(1,j)}_{\alpha}-\exp(A_j)J^{(1,j)}_{\alpha}\exp(-A_j)]S_{(1,j)}^{-1}$,
       \item $c_{\alpha}^j = S_{(1,j)}J^{(1,j)}_{\alpha}S_{(1,j)}^{-1}(b_{(2,j)}-b_{(1,j)})$,
       \item $Q_j$ = $\sup\limits_{B_{10nH(p_j)^{-1}}(p_j)}\max_{\alpha}|K_{\alpha}^{(1,j)} -\sum_{\beta} \omega^{(j)}_{\alpha\beta}K_{\beta}^{(2,j)} |$.
     \end{itemize}
     By definition of $Q_j$ we have $|P^j_{\alpha}| + |c^j_{\alpha}| \leq CQ_j$. Consequently, for sufficiently large $j$, we have
     \begin{align}
       & |P_{\alpha}^j| = |[A_j, J^{(1,j)}_{\alpha}] + o(|A_j|)| \leq CQ_j \\\nonumber
   \Rightarrow & |A_j|\leq C\max_{\alpha}|[A_j,J^{(1,j)}_{\alpha}]| \leq CQ_j+o(|A_j|)\\\nonumber
      \Rightarrow & |A_j|\leq CQ_j,
     \end{align}
     where the second inequality follows from the fact that $A_j \perp so(k)\oplus so(n-k+1)$ and similar computation in \eqref{[A,J]=0}.

     Note that $J^{(1,j)}_{\alpha}\rightarrow J^{(1, \infty)}_{\alpha}$ by the previous discussion. Now we can pass to a subsequence such that $\frac{P_{\alpha}^j}{Q_j}\rightarrow [A,J^{(1, \infty)}_{\alpha}]$ and
     $\frac{c^j_{\alpha}}{Q_j}\rightarrow c^{\infty}_{\alpha}\in Im(J^{(1, \infty)}_{\alpha})=(ker(J^{(1, \infty)}_{\alpha})^{\perp}$.  Consequently, we have
     \begin{align}
       W^j_{\alpha} \rightarrow W_{\alpha}^{\infty} = [A,J^{(1, \infty)}_{\alpha}]x +c^{\infty}_{\alpha}
     \end{align}
and
\begin{align}
\sup\limits_{B_{10H(p_j)^{-1}}(p_j)}\max_{\alpha}|W^{\infty}_{\alpha}|=1.
\end{align}
  On the other hand, by our assumption we have
     \begin{itemize}
       \item $\max_{\alpha}|\left<K^{(1,j)}_{\alpha} - \sum_{\beta}\omega^{(j)}_{\alpha\beta}K_{\beta}^{(2,j)}, \nu\right>|\leq 2H^{-1}\varepsilon_j$ \text{in $B_{g_{j}}(p_{j}, H(p_{j})^{-1})$},
       \item ${\sup\limits_{B_{10nH(p_j)^{-1}}(p_j)}\max_{\alpha}|K_{\alpha}^{(1,j)} -\sum_{\beta} \omega^{(j)}_{\alpha\beta}K_{\beta}^{(2,j)} |\geq jH^{-1}}(p_{j})\varepsilon_j$.
     \end{itemize}
This implies
\begin{align}
    \sup_{ B_{g_{\infty}}(p_{\infty},1)}\max_{\alpha}|\left<\nu, W_{\alpha}^{\infty}\right>|=0.
\end{align}
 By Lemma \ref{Lin Alg Cylinder B}, we have  $W_{\alpha}^{\infty}\equiv 0$ for all $1\leq\alpha\leq \frac{(n-k+1)(n-k)}{2}$, which is a contradiction.

\end{proof}

 \begin{lemma}\label{VF close Bowl times lines}
Given any $\delta>0$,  there exist constants $0 < \varepsilon_{b} \ll 1$ and $C>1$ depending only on  dimension $n$ and $\delta$  with the following properties. Let $M$ be a hypersurface in $\mathbb{R}^{n+1}$ with induced metric $g$ and $\Sigma$ be the unique $(n-k+1)$ dimensional round bowl soliton with maximal mean curvature $1$.
 Suppose that $q \in \mathbb{R}^{k-1}\times \Sigma$, and $M$ is a graph over the geodesic ball $B_{g}(q, 2H(q)^{-1})$ inside  $\mathbb{R}^{k-1}\times \Sigma$ with graphical $C^{3}$
norm  no more than $\varepsilon_{b}$ after rescaling by $H(q)^{-1}$. Let  $\bar{x}\in M$ be a point that has rescaled distance to $q$ no more than $\varepsilon_{b}$.  Suppose that $\varepsilon \leq \varepsilon_b$ and
  $\mathcal{K}^{(1)} = \{K^{(1)}_{\alpha}: 1 \leq \alpha \leq \frac{(n-k+1)(n-k)}{2}\}$,
  $\mathcal{K}^{(2)} = \{K^{(2)}_{\alpha}: 1 \leq \alpha \leq \frac{(n-k+1)(n-k)}{2}\}$, are two     normalized sets of rotation vector fields, and we assume that:
  \begin{itemize}
    \item $\displaystyle \sum_{i=1}^{k}\lambda_i\geq \delta H$ in $B_{g}(\bar{x}, H(\bar{x})^{-1})\subset M$, where $\lambda_1 \leq \lambda_2 \leq ... \leq \lambda_n$ are   principal curvatures,
    \item $\max_{\alpha}|\left<K_{\alpha}^{(i)}, \nu\right>|H \leq \varepsilon$ in
    $B_g(\bar{x}, H^{-1}(\bar{x}))\subset M$,
    \item $\max_{\alpha}|K_{\alpha}^{(i)}|H \leq 5n$ in $B_{g}(\bar{x}, 10n^{5/2}H^{-1}(\bar{x}))\subset M$,
  \end{itemize}
  for $i = 1 , 2$.  Then for any $L > 1$, we have
  \begin{align}
    \inf\limits_{(\omega_{\alpha\beta})\in \mathrm{O}(\frac{(n-k+1)(n-k)}{2})}\sup\limits_{B_{LH(\bar{x})^{-1}}(\bar{x})}\max_{\alpha}
    |K_{\alpha}^{(1)}-\sum_{\beta = 1}^{\frac{(n-k+1)(n-k)}{2}}\omega_{\alpha\beta}K^{(2)}_{\beta}|H(\bar{x})
    \leq CL\varepsilon.
  \end{align}
\end{lemma}
\begin{proof}
With the help of Lemma A.1, the proof is analogous to \cite[Lemma 3.8]{Zhu2}. For the readers' convenience, we include the details here. Since the vector fields are affine functions, it suffices to prove the result for  $L=10n$.

    Let's make the convention that
    the tip of $\Sigma$ is the origin, the rotation axis is $x_{k}$, and
    $\Sigma$ encloses the positive part of $x_{k}$ axis.
    Argue by contradiction, if the assertion is not true, then there exists
    a sequence of points $q_j\in \mathbb{R}^{k-1}\times \kappa_j^{-1}\Sigma$ with
    $H(q_j)=1$ and a sequence of pointed hypersurfaces $(M_j, p_j)$ that are
    $1/j$ close to a geodesic ball $B_{\tilde{g}_j}(q_j,2)$ in
    $\mathbb{R}^{k-1}\times \kappa_j^{-1}\Sigma$, where $\tilde{g}_{j}$ is the induced
    metric on $\mathbb{R}^{k-1}\times \kappa_j^{-1}\Sigma$. Suppose that $|p_j-q_j|\leq 1/j$.

  Without loss of generality we may assume that
     $\left<q_j, e_{l}\right>=0$ for $l=1, \dots, k-1$,
    where $\{e_{1}, ..., e_{k-1}\}$ is the standard orthonormal basis of $\mathbb{R}^{k-1}$.

  Furthermore, for $i=1,2$ and each positive integer $j$,  there exists normalized set of rotation vector fields $\mathcal{K}^{(i,j)}$=$\{K^{(i,j)}_{\alpha}$, $1\leq \alpha\leq \frac{(n-k+1)(n-k)}{2}\}$    and
  $\varepsilon_j<1/j$  such that
  \begin{itemize}
    \item $\max_{\alpha}|\left<K^{(i,j)}_{\alpha},\nu\right>|H\leq \varepsilon_j$ in $B_{g_j}(p_j,H(p_j)^{-1})\subset M_j$,
    \item $\max_{\alpha}|K^{(i,j)}_{\alpha}|H\leq 5$ at $p_j$,
  \end{itemize}
 but

\begin{equation}\label{contra2}
     \inf\limits_{(\omega_{\alpha\beta})\in \mathrm{O}(\frac{(n-k+1)(n-k)}{2})}\sup\limits_{B_{10nH(p_{j})^{-1}}(\bar{x})}\max_{\alpha}
    |K_{\alpha}^{(1,j)}-\sum_{\beta = 1}^{\frac{(n-k+1)(n-k)}{2}}\omega_{\alpha\beta}K^{(2,j)}_{\beta}|H({p_{j}})
    \geq j\varepsilon_j.
\end{equation}

  Now the maximal mean curvature of $\mathbb{R}^{k-1}\times \kappa_j^{-1}\Sigma$ is $\kappa_j$.
  For every $j>2C/\delta$, by the first condition and approximation we know that $\displaystyle \sum_{i=1}^{k}\lambda_i \geq \frac{\delta}{2} H$ around $q_j\in \mathbb{R}^{k-1}\times \kappa_j^{-1}\Sigma$.
  The asymptotic behaviour of the bowl soliton indicates that $\frac{1}{2} < \frac{\kappa_j}{H(q_j)}<C(\delta)$  and $\displaystyle |q_j-\sum_{l=1}^{k-1}\left<q_j,e_{l}\right>e_{l}|\kappa_j<C(\delta)$, thus
  $\kappa_j<C(\delta)$ and $\displaystyle |q_j|=|q_j-\sum_{l=1}^{k-1}\left<q_j, e_{l}\right>e_{l}|<C(\delta)$.

  We can then pass to a subsequence such that $q_j\rightarrow q_{\infty}$ and $\kappa_j\rightarrow\kappa_{\infty} \in [\frac{1}{2}, C(\delta)]$.
  Consequently $\mathbb{R}^{k-1}\times\kappa_j^{-1}\Sigma\rightarrow \mathbb{R}^{k-1}\times\kappa_{\infty}^{-1}\Sigma$ and  $B_{\tilde{g}}(q_j,2)\rightarrow B_{\tilde{g}_{\infty}}(q_{\infty},2)$ smoothly,
  where $B_{\tilde{g}_{\infty}}(q_{\infty},2)$ is the geodesic ball in $\mathbb{R}^{k-1}\times\kappa_{\infty}^{-1}\Sigma$.

  Combing with the assumption that $(M_j,p_j)$ is $1/j$  close to $(B_{\tilde{g}_j}(q_j,2),q_j)$ and $H(q_j)=1$, we have $M_j\rightarrow B_{\tilde{g}_{\infty}}(q_{\infty},2)$ with $p_j\rightarrow q_{\infty}$
  and $H(q_{\infty})=1$.

    We can write $K_{\alpha}^{(i,j)}(x)=S_{(i,j)}J^{(i,j)}_{\alpha}S_{(i,j)}^{-1}(x-b_{(i,j)})$ for some orthonormal basis $\{J^{(i,j)}_{\alpha}, 1\leq \alpha\leq \frac{(n-k+1)(n-k)}{2}\}$ of
    $so(n-k+1)\subset so(n+1)$
    and assume that $(b_{(i,j)}-p_j)\perp \bigcap_{\alpha}\ker S_{(i,j)}J^{(i,j)}_{\alpha}S_{(i,j)}^{-1}$. Then
    $|b_j-p_j|=|S_{(i,j)}J^{(i,j)}_{\alpha}S_{(i,j)}^{-1}(p_j-b_j)|\leq C(n)$ for large $j$.

    Then, we can pass to a subsequence such that $S_{(i,j)}$ and $b_{(i,j)}$ converge to $S_{(i,\infty)}$ and $b_{(i,\infty)}$ respectively.
Hence $K^{(i,j)}\rightarrow K^{(i,\infty)}_{\alpha} = S_{(i,\infty)}J^{(i,\infty)}_{\alpha}S_{(i,\infty)}^{-1}(x-b_{(i,\infty)})$ for $i=1,2$.

  The convergence implies that
  \begin{itemize}
     \item $\left<K^{(i,\infty)}_{\alpha},\nu\right> = 0$ in $B_{g_{\infty}}(p_{\infty,}, H_0^{-1})$,
     \item $\max_{\alpha}|K^{(i,\infty)}_{\alpha}| H \leq 5n$ at
      $p_{\infty}$.
  \end{itemize}
    By Lemma \ref{Lin Alg Cylinder A}  we have
    \begin{flalign}
      K_{\alpha}^{(i,\infty)}(x) = \sum_{\beta=1}^{\frac{(n-k+1)(n-k)}{2}}\omega^{(i)}_{\alpha\beta} J_{\beta}x
    \end{flalign}
    for some fixed orthonormal basis $\{J_{\alpha}, 1\leq \alpha\leq \frac{(n-k+1)(n-k)}{2}\}$ of $so(n-k+1)\subset so(n+1)$ and $(\omega ^{(i)}_{\alpha\beta})\in \mathrm{O}(\frac{(n-k+1)(n-k)}{2})$.

    Finally, we can argue exactly as in the proof of Lemma \ref{Vector Field closeness on genearalized cylinder_appendix} to reach
    a contradiction. This completes the proof of lemma.
\end{proof}

\end{appendix}
\bigskip

\bibliography{hearing_shape}

\bibliographystyle{alpha}


{\sc Wenkui du, Department of Mathematics, University of Toronto, Ontario, Canada}

{\sc Jingze Zhu, Department of Mathematics, Massachusetts Institute of Technology,  Massachusetts, USA}

\emph{E-mail:} wenkui.du@mail.utoronto.ca, zhujz@mit.edu.

\end{document}